\pdfoutput=1 %arxiv

\documentclass[12pt, letterpaper, leqno]{amsart}
\usepackage[letterpaper, margin=1in]{geometry}
\usepackage[utf8]{inputenc}
\usepackage[style=alphabetic, maxnames=99]{biblatex}
\AtBeginBibliography{\small}
\usepackage{microtype}
\microtypesetup{protrusion=false}
\usepackage{xpatch}
\usepackage[english]{babel}
\usepackage{amsfonts}
\usepackage{amsmath}
\usepackage{amssymb}
\usepackage{mathrsfs}
\usepackage{stmaryrd}

\usepackage{amsthm}
\xpatchcmd{\proof}{\itshape}{\bfseries}{}{}
\usepackage{thmtools}
\usepackage{thm-restate}
\usepackage{tikz}
\usepackage{tikz-cd}
\usetikzlibrary{arrows.meta,automata,positioning}
\usetikzlibrary{svg.path}
\usepackage{amscd}
\usepackage{quiver}
\usepackage{tikzit}
\usepackage{graphicx}
\usepackage{float}
\usepackage{enumitem}
\usepackage[labelfont=bf]{caption}
\usepackage[linewidth=0pt]{mdframed}
\usepackage{subfig}
\usepackage{fancyhdr}
\usepackage[bottom, perpage]{footmisc}
\usepackage{etoolbox}

\patchcmd{\section}{\scshape}{\bfseries\scshape}{}{}
\makeatletter
\renewcommand{\@secnumfont}{\bfseries\scshape}
\renewcommand\subsubsection{\@startsection{subsubsection}{3}%
	\z@{.5\linespacing\@plus.7\linespacing}{-.5em}%
	{\normalfont\bfseries}}
\makeatother

\allowdisplaybreaks
\usepackage{subfiles}
\usepackage{setspace}
\numberwithin{equation}{section}
\numberwithin{figure}{section}
\makeatletter
\let\c@equation\c@figure
\makeatother

\usepackage[shortcuts]{extdash}
\usepackage{xstring}
\usepackage{hyperref}
\usepackage[capitalize]{cleveref}
\usepackage{url}
\hypersetup{
	colorlinks   = true, 
	urlcolor     = cyan, 
	linkcolor    = blue, 
	citecolor   = blue
}

\numberwithin{equation}{section}
\numberwithin{figure}{section}
\makeatletter
\let\c@equation\c@figure
\makeatother

\makeatletter
\newcommand{\labeltext}[3][]{%
	\@bsphack%
	\csname phantomsection\endcsname% in case hyperref is used
	\def\tst{#1}%
	\def\labelmarkup{}% How to markup the label itself
	\def\refmarkup{}%
	\ifx\tst\empty\def\@currentlabel{\refmarkup{#2}}{\label{#3}}%
	\else\def\@currentlabel{(\refmarkup{#1})}{\label{#3}}\fi%
	\@esphack%
	\labelmarkup{#2}% visible printed text.
}
\makeatother
\declaretheorem[numberwithin=section,numberlike=equation,style=plain]{theorem}
\declaretheorem[numbered=no,style=plain,name=Theorem]{theorem*}
\declaretheorem[numberwithin=section,numberlike=equation,style=plain]{proposition}
\declaretheorem[numbered=no,style=plain,name=Proposition]{proposition*}
\declaretheorem[numberwithin=section,numberlike=equation,style=plain]{lemma}
\declaretheorem[numbered=no,style=plain,name=Lemma]{lemma*}
\declaretheorem[numberwithin=section,numberlike=equation,style=plain]{corollary}
\declaretheorem[numbered=no,style=plain,name=Corollary]{corollary*}

\declaretheorem[numbered=no,style=plain,name=Conjecture]{conjecture*}

\declaretheorem[numbered=no,style=plain,name=Question]{question*}

\declaretheorem[numberwithin=section,numberlike=equation,style=definition]{definition}
\declaretheorem[numbered=no,style=definition,name=Definition]{definition*}
\declaretheorem[numberwithin=section,numberlike=equation,style=definition]{remark}
\declaretheorem[numbered=no,style=definition,name=Remark]{remark*}
\declaretheorem[numberwithin=section,numberlike=equation,style=definition]{notation}
\declaretheorem[numbered=no,style=definition,name=Notation]{notation*}

\declaretheorem[numbered=no,style=definition,name=Axiom]{axiom*}

\declaretheorem[numbered=no,style=definition,name=Construction]{construction*}

\declaretheorem[numbered=no,style=definition,name=Algorithm]{algorithm*}

\declaretheorem[numbered=no,style=definition,name=Summary]{summary*}

\declaretheorem[numbered=no,style=definition,name=Property]{property*}

\declaretheorem[numbered=no,style=definition,name=Note]{note*}

\declaretheorem[numberwithin=section,numberlike=equation,style=definition,qed=$\diamondsuit$]{example}
\declaretheorem[numbered=no,style=definition,name=Example,qed=$\diamondsuit$]{example*}
\usepackage{mystyle}

\newcommand{\define}[1]{\textcolor{magenta}{\emph{#1}}}

\title{Horospherical stacks and stacky coloured fans}
\author{Sean Monahan}
\address{Department of Pure Mathematics, University of Waterloo, 200 University Ave. W, Waterloo, Ontario, Canada}
\email{sean.monahan@uwaterloo.ca}

\tikzstyle{orange dot}=[fill={rgb,255: red,255; green,128; blue,0}, draw=black, shape=circle, scale=0.75]
\tikzstyle{empty dot}=[fill=none, draw=black, shape=circle, scale=0.75]
\tikzstyle{vertical marker}=[fill=none, draw=none, shape=circle, label={center:{\footnotesize $|$}}]
\tikzstyle{origin}=[fill=none, draw=none, shape=circle, label={center:{\Large $|$}}]
\tikzstyle{dot}=[fill=black, draw=black, shape=circle, scale=0.5]
\tikzstyle{fantastack square}=[fill=none, draw=black, shape=rectangle, scale=1.25]

\tikzstyle{arrow}=[->, line width=1.0]
\tikzstyle{dashed line}=[-, dashed, fill=none]
\tikzstyle{cone fill}=[-, draw=none, fill={rgb,255: red,191; green,191; blue,191}]
\tikzstyle{backed arrow}=[{|->}, line width=1]
\tikzstyle{filled line}=[-, line width=1]

\newcommand{\tikzitfig}[1]{\ctikzfig{./Figures/#1}}
\begin{document}
	
\maketitle
\begin{abstract}
	We introduce a combinatorial theory of \textit{horospherical stacks} which is motivated by the work of Geraschenko and Satriano on toric stacks. A horospherical stack corresponds to a combinatorial object called a \textit{stacky coloured fan}. We give many concrete examples, including a class of easy-to-draw examples called \textit{coloured fantastacks}. The main results in this paper are combinatorial descriptions of horospherical stacks, the morphisms between them, their \textit{decolourations}, and their good moduli spaces. 
\end{abstract}
\setcounter{tocdepth}{1}
\tableofcontents

\section{Introduction}\label{sec:introduction}

Horospherical varieties have been studied for many years as a nice subclass of spherical varieties, the theory of which was pioneered by Brion, Luna, Vust, Vinberg, Popov, and Pauer \cite{brion1986espaces,brion1987structure,luna1983plongements,vinberg1972certain,pauer1981normale}. This class of horospherical varieties consists of certain normal varieties which have an open orbit and finitely many orbits under the action of a reductive algebraic group; this class is broad enough to include both toric varieties and flag varieties. The Luna-Vust theory provides a combinatorial description of these varieties via so-called coloured fans -- a generalization of fans in toric geometry -- which captures this simple orbit structure; see \cref{sec:notation for horospherical varieties} for details. 

This combinatorial theory plays a vital role in understanding horospherical varieties. Much is known about the birational geometry of horospherical varieties, e.g. their minimal models \cite{brion1993mori,pasquier2015approach,pasquier2018log}; many types of singularities (e.g. canonical, terminal, and log terminal) for horospherical varieties have been characterized \cite{langlois2017singularites,pasquier2016klt,pasquier2015survey}; the generalized Mukai conjecture has been proven for horospherical varieties \cite{pasquier2010pseudo}; smooth horospherical varieties with Picard number $1$ have been classified and studied \cite{pasquier2009smooth,gonzales2022geometry,hong2016smooth}; and the Cox ring and Cox GIT construction (as introduced in \cite{cox1995homogeneous} for toric varieties) were described for horospherical varieties \cite{brion2007total,gagliardi2014cox,gagliardi2019luna}.

Since these horospherical objects have proved fruitful in the world of varieties, it seems natural to generalize them to the larger world of algebraic stacks and hope for a similar level of success. There are two fundamental ways in which stacks commonly arise. One is in the study of moduli spaces where one wants to parametrize certain geometric objects up to isomorphism. Horospherical stacks encapsulate several examples of moduli spaces, e.g. Grassmannians (which are flag varieties) and certain moduli spaces of line bundles which can be interpreted as toric stacks (see \cite[Section 7]{satriano2015toric1}). The other important use for stacks is in the construction of quotients of varieties by algebraic group actions. Studying quotients of horospherical varieties is very natural, given that one of the main features of these varieties is their simple orbit structure. 

We are aware of only one theory of horospherical stacks, which is by Javanpeykar, Langlois, and Terpereau in \cite{terpereau2019horospherical}. They focus on developing the structural theory of these objects, generalizing the work of Geraschenko and Satriano on toric stacks \cite{satriano2015toric2}. In particular, they consider \textit{abstract} horospherical stacks, which are algebraic stacks with a reductive group action and a horospherical homogeneous space as an open substack, and they present conditions for when these objects are ``concrete" horospherical stacks, meaning that they are stack quotients of horospherical varieties. 

For the same reasons that one appreciates combinatorial theories for varieties, it is natural to explore such theories for stacks. In this paper, we provide a combinatorial theory for the aforementioned ``concrete" horospherical stacks, which is currently absent in the literature. There has already been significant interest in combinatorial theories for toric stacks \cite{borisov2005orbifold, fantechi2010smooth, gillam2015theory, satriano2015toric1, satriano2015toric2, iwanari2009category, tyomkin2012tropical, lafforgue2002chtoucas}. Expanding on these to include horospherical stacks could assist in tackling problems which are beyond the scope of toric geometry, e.g. problems which involve nontrivial representation theory. 

Our combinatorial theory is primarily motivated by that of toric stacks due to Geraschenko and Satriano \cite{satriano2015toric1}. They relate a toric stack to a stacky fan, which uses the well-known correspondence between toric varieties and fans. For us, horospherical stacks correspond to a new object called a stacky \textit{coloured} fan, since \textit{coloured} fans appear in the classification of horospherical varieties. Due to the strong similarities in these objects, the combinatorics for horospherical stacks has a very familiar feel to the combinatorics for toric stacks. 

These combinatorial descriptions are excellent for working with examples, doing explicit computations, and building intuition for more general varieties and stacks. The goal of this paper is to enlarge the class of examples that one can play with, while keeping the satisfying combinatorics that one appreciates in toric geometry. To begin developing a dictionary between the horospherical stacks and the stacky coloured fans, we do the following:
\begin{itemize}[leftmargin=2.5em]
	\item We show how horospherical stacks correspond to stacky coloured fans, and we show how morphisms of horospherical stacks correspond to maps of the associated stacky coloured fans; see \cref{thm:horospherical morphisms} and \cref{thm:horospherical isomorphisms} for results regarding the morphisms.
	
	\item We introduce an easy-to-draw class of horospherical stacks called \textit{coloured fantastacks} which aim to generalize the toric fantastacks from \cite{satriano2015toric1}; see \cref{subsec:construction of coloured fantastacks}.
	
	\item We give a combinatorial definition of \textit{toroidal} horospherical stacks, which are a class of horospherical stacks that closely resemble toric stacks, and we describe the \textit{decolouration} of a horospherical stack; see \cref{cor:uniqueness of decolouration} for uniqueness of this decolouration. 
	
	\item Lastly, we give a combinatorial description of the good moduli space of a horospherical stack in the sense of \cite{alper2013good}; see \cref{thm:GMS map criteria} and \cref{cor:good moduli space of a horospherical stack}. 
\end{itemize}

\subsection*{Summary of results}

Let $k$ be an algebraically closed field of characteristic $0$; all varieties and stacks for us are over $k$. Let $G$ be a connected reductive linear algebraic group over $k$. We define a horospherical $G$-stack as follows. A \textit{horospherical $G$-variety} is a normal $G$-variety such that there is an open orbit $G/H\subseteq X$ for which the stabilizer $H$ is a horospherical subgroup, i.e. $H$ contains a maximal unipotent subgroup of $G$. This ``horospherical condition" on $H$ is equivalent to insisting that $G/H$ be a principal torus bundle over a flag variety. 

A horospherical variety $X$ corresponds to a coloured fan $\Sigma^c$ which lives on a coloured lattice $N$ associated to the horospherical homogeneous space $G/H$. This coloured fan is a finite set of coloured cones $\sigma^c=(\sigma,\calF)$, where $\sigma$ denotes a polyhedral cone on $N$ and $\calF$ is a \textit{colour set}. We provide more details on horospherical varieties and coloured fans in \cref{sec:notation for horospherical varieties}.

Now a \textit{horospherical $G$-stack} is a quotient stack $[X/K]$ where $X$ is a horospherical $G$-variety and $K\subseteq\Aut^G(X)$ is a closed subgroup (see \cref{def:horospherical stack}). These specialize to horospherical varieties in the case where $K=\{1\}$, and these specialize to the toric stacks from \cite{satriano2015toric1} in the case where $G=T$ is a torus. 

In \cite{satriano2015toric1}, toric stacks correspond to combinatorial gadgets called stacky fans, which are a generalization of fans for toric varieties. We introduce a new combinatorial object called a \textit{stacky coloured fan}: this is a pair $(\Sigma^c,\beta)$ where $\Sigma^c$ is a coloured fan and $\beta:N\to L$ is a $\Z$-linear map of lattices with finite cokernel (see \cref{def:stacky coloured fan}). In the case where $G=T$ is a torus, the coloured lattice $N$ is the usual one-parameter subgroup lattice of the torus $T/H$, the coloured fan $\Sigma^c$ is simply a fan $\Sigma$, and we recover the usual toric stacky fan $(\Sigma,\beta)$. 

In \cref{sec:horospherical stacks and stacky coloured fans} we give formal definitions of these main objects described above and we show how they correspond to each other. Given a stacky coloured fan $(\Sigma^c,\beta)$, we can produce an associated horospherical $G$-stack $\calX_{\Sigma^c,\beta}$ (see \cref{subsec:stacky coloured fan to horospherical stack}); in terms of the quotient presentation $[X/K]$, one should think of the coloured fan $\Sigma^c$ as controlling the horospherical variety $X$ (using the Luna-Vust theory), and the map $\beta:N\to L$ as controlling the subgroup $K$. 

In \cref{sec:horospherical morphisms and maps of stacky coloured fans}, we describe the morphisms between the above objects. A \textit{morphism of horospherical stacks} is a $G$-equivariant morphism of stacks (see \cref{def:horospherical morphism}). On the other hand, a \textit{map of stacky coloured fans} $(\Phi,\phi):(\Sigma_1^c,\beta_1:N_1\to L_1)\to (\Sigma_2^c,\beta_2:N_2\to L_2)$ consists of a linear map $\Phi:N_1\to N_2$ compatible with the coloured fans $\Sigma_1^c$ and $\Sigma_2^c$ and a linear map $\phi:L_1\to L_2$ such that $\beta_2\circ\Phi=\phi\circ\beta_1$ (see \cref{def:map of stacky coloured fans}). 

It is straightforward to describe how a map of stacky coloured fans $$(\Sigma_1^c,\beta_1:N_1\to L_1)\to(\Sigma_2^c,\beta_2:N_2\to L_2)$$ induces a morphism of horospherical stacks $\calX_{\Sigma_1^c,\beta_1}\to\calX_{\Sigma_2^c,\beta_2}$ (see \cref{subsec:map of stacky coloured fans to horospherical morphism}). However the other direction of going from a horospherical morphism to a map of stacky coloured fans is more complicated. We prove the following result to deal with this latter direction; roughly speaking, this result says that any morphism of horospherical stacks is induced by a map of stacky coloured fans, up to changing the presentation of the domain stack.

\begin{theorem*}[{\cref{thm:horospherical morphisms}}]
	Let $\calX_{\Sigma_i^c,\beta_i}$ be horospherical $G$-stacks for $i=1,2$, and suppose that $\pi:\calX_{\Sigma_1^c,\beta_1}\to\calX_{\Sigma_2^c,\beta_2}$ is a morphism of horospherical stacks. Then there exists a stacky coloured fan $(\Sigma_0^c,\beta_0)$ and a tower of stacky coloured fans (below on the left) which induces the following commutative diagram of horospherical stacks (below on the right)
	\begin{equation*}
		\begin{tikzcd}
			{(\Sigma_0^c,\beta_0)} &&& {\calX_{\Sigma_0^c,\beta_0}} \\
			{(\Sigma_1^c,\beta_1)} & {(\Sigma_2^c,\beta_2)} && {\calX_{\Sigma_1^c,\beta_1}} & {\calX_{\Sigma_2^c,\beta_2}}
			\arrow["{\rotatebox{90}{$\sim$}}"', from=1-4, to=2-4]
			\arrow["\pi"', from=2-4, to=2-5]
			\arrow[from=1-4, to=2-5]
			\arrow[from=1-1, to=2-2]
			\arrow[from=1-1, to=2-1]
		\end{tikzcd}
	\end{equation*}
	where $\calX_{\Sigma_0^c,\beta_0}\isomap\calX_{\Sigma_1^c,\beta_1}$ is an isomorphism.
\end{theorem*}

We also provide combinatorial criteria for checking when a morphism of horospherical stacks is an isomorphism. 

\begin{theorem*}[{\cref{thm:horospherical isomorphisms}}]
	Let $\pi:\calX_{\Sigma_1^c,\beta_1}\to\calX_{\Sigma_2^c,\beta_2}$ be a morphism of horospherical stacks which is induced by a map of stacky coloured fans $(\Phi,\phi):(\Sigma_1^c,\beta_1)\to(\Sigma_2^c,\beta_2)$. Then $\pi$ is an isomorphism if and only if $\phi$ is an isomorphism of coloured lattices (see \cref{subsec:horospherical morphisms and maps of coloured fans}) and, for each $(\sigma_2,\calF_2)\in\Sigma_2^c$, the preimage under $\Phi$ (see \cref{not:preimage coloured cone}) is a coloured cone $(\sigma_1,\calF_1)\in\Sigma_1^c$ such that $\calF_1=\calF_2$ and the map $\Phi$ induces an isomorphism of monoids $\sigma_1\cap N_1\to \sigma_2\cap N_2$.
\end{theorem*}

In \cref{sec:coloured fantastacks} we introduce a nice class of easy-to-draw stacky coloured fans, and we call their corresponding horospherical stacks \textit{coloured fantastacks}; these aim to generalize the toric fantastacks introduced in \cite{satriano2015toric1}. In this construction, one starts with a horospherical variety $X$, and puts stacky structure ``on top of" $X$ to yield a horospherical stack $\calX$ in such a way that $\calX\to X$ is a good moduli space in the sense of \cite{alper2013good} (see \cref{rmk:gms of coloured fantastack}). In particular, we review the Cox construction for horospherical varieties (introduced by \cite{cox1995homogeneous} for toric varieties) and show how this construction fits into the coloured fantastack framework. 

Using the Cox ring of a horospherical variety, we are able to give a simple condition to test if a stack is horospherical but not toric.
\begin{proposition*}[\cref{prop:non-toric condition}]
	Let $[X/K]$ be a horospherical $G$-stack such that $\mathscr{O}_X^*(X)=k^*$. If the flag variety $G/N_G(K)$ is not a product of projective spaces, then $[X/K]$ is not a toric stack (in the sense of \cite{satriano2015toric1}). 
\end{proposition*}

Another special class of horospherical stacks are \textit{toroidal} horospherical stacks, which are studied in \cite{terpereau2019horospherical} from a structure theory perspective. We give a combinatorial definition of these stacks in \cref{sec:toroidal horospherical stacks and decolouration}. Also, we give a very simple combinatorial description of \textit{decolouration}, which takes any horospherical stack $\calX$ and produces a toroidal horospherical stack $\wt\calX$ and a proper surjective horospherical morphism $\wt\pi:\wt\calX\to \calX$; \cref{cor:uniqueness of decolouration} shows that $\wt\calX$ is unique to $\calX$. 
	
Lastly, in \cref{sec:good moduli spaces} we provide combinatorial criteria for checking when a morphism of horospherical stacks is a good moduli space morphism in the sense of \cite{alper2013good}; see \cref{thm:GMS map criteria}. One should think of the good moduli space of a stack as a scheme (or algebraic space) ``approximation" of the stack. As a corollary, we prove combinatorial criteria which describes exactly when a horospherical stack has a variety as a good moduli space, and in this case we can describe the combinatorics of the good moduli space explicitly. This explicit description is explained in \cref{not:good moduli space coloured fan}, where we produce a coloured fan $\Sigma^c_\gms$ from a stacky coloured fan $(\Sigma^c,\beta)$. 

\begin{theorem*}[\cref{cor:good moduli space of a horospherical stack}]
	Let $\calX_{\Sigma^c,\beta}$ be a horospherical $G$-stack, and let $\Sigma^c_\gms$ be the coloured fan constructed in \cref{not:good moduli space coloured fan}. Then the good moduli space of $\calX_{\Sigma^c,\beta}$ exists if and only if $(\Sigma^c,\beta)$ is compatible with $\Sigma^c_\gms$ (in the sense of \cref{not:good moduli space coloured fan}). Moreover, in this case, the horospherical variety corresponding to $\Sigma^c_\gms$ is the good moduli space of $\calX_{\Sigma^c,\beta}$. 
\end{theorem*}

\subsection*{Comments on spherical stacks}

As mentioned above, the combinatorial theory of coloured fans holds for a larger class of spherical varieties, but this theory is greatly simplified when restricting to the horospherical setting. It is a natural question to ask if one can extend the theory of horospherical stacks to one of \textit{spherical stacks} and develop a similar combinatorial description using stacky coloured fans. In some areas this extension would be immediate, however there are some aspects of horospherical stacks which rely on the special features of horospherical varieties.

In \cite[Remark 1.7]{terpereau2019horospherical}, the authors highlight one of the main obstacles in extending to spherical stacks, namely that the local structure results for spherical varieties are stronger when specialized to the horospherical subclass. This poses some challenges for extending their structure results to spherical stacks, as it would with extending some of our combinatorial results (we make use of the local structure results in \cref{sec:proof of isomorphism criteria} and \cref{sec:proof of good moduli space criteria}). 

There are other minor obstacles, including the fact that $\Aut^G(X)$ is not necessarily a torus for spherical varieties, and the combinatorial description of the Cox construction would be different from our \cref{subsec:good quotient Cox construction}, which is relevant for extending our coloured fantastacks in \cref{subsec:construction of coloured fantastacks} to include spherical varieties.

\section{Notation for horospherical varieties}\label{sec:notation for horospherical varieties}

Throughout this paper we use the following notation. Fix an algebraically closed field $k$ of characteristic $0$. All schemes and stacks for us are over $k$. Note that, for us, a \textit{variety} is an integral separated scheme of finite type over $k$; so varieties are irreducible.

\subsection{Algebraic group notation}\label{subsec:algebraic group notation}

Fix a connected reductive linear algebraic group $G$, a Borel subgroup $B\subseteq G$, and a maximal torus $T\subseteq B$. Relative to the choice of $(G,B,T)$, fix the set of simple roots $S$. Let $U\subseteq B$ be the maximal unipotent subgroup of $G$ consisting of unipotent elements of $B$. There are two main examples to keep in mind: (1) is when $G$ is a torus, in which case $G=B=T$, $U=\{1\}$, and $S=\varnothing$; and (2) is when $G=SL_n$, in which case we can take $B=B_n$ and $T=T_n$ to be the subgroups of upper triangular and diagonal matrices, respectively, and then $U=U_n$ is the subgroup of upper triangular matrices with all $1$'s on the diagonal, and we can write $S=\{\alpha_1,\ldots,\alpha_{n-1}\}$. For more details on algebraic group theory, see \cite{humphreys1975linear}. 

Given any linear algebraic group $G'$, we let $\frakX(G')$ and $\frakX^*(G')$ denote the abelian groups of characters of $G'$ and of one-parameter subgroups of $G'$, respectively. We generally use additive notation for $\frakX(G')$ and $\frakX^*(G')$, but we write $\chi\in\frakX(G')$ to represent the multiplicative character $\chi:G'\to\Gm$. When $G'$ is a torus, these abelian groups are lattices which are dual to each other, i.e. $\frakX(G')=\frakX^*(G')^\vee$.

\subsection{Horospherical varieties}\label{subsec:horospherical varieties}

In the rest of this section we quickly review some notation for horospherical varieties. For more information, we recommend \cite{pasquier2009introduction}, \cite{knop1991luna}, \cite{perrin2014geometry} and \cite{perrin2018sanya} for the basics, and we recommend \cite{timashev2011homogeneous} for a more extensive resource.

Fix a closed subgroup $H\subseteq G$ which contains $U$; such a subgroup is called a \define{horospherical subgroup}. A \define{$G/H$-horospherical variety} (or in other sources, a \textit{$G/H$-embedding}) is a normal $G$-variety $X$ together with a distinguished \define{base point} $x\in X$ such that the stabilizer $G_x=H$ is the horospherical subgroup and the orbit $G\cdot x=G/H$ is open in $X$; note that the base point is $x=eH\in G/H\subseteq X$. More generally, we say that $X$ is a \define{horospherical $G$-variety} if it is a $G/H$-horospherical variety for some horospherical subgroup $H\subseteq G$. In the case where $G=T$ is a torus and $H=\{1\}$, a toric variety is a $T/\{1\}$-horospherical variety; see \cite{cox2011toric} for more on toric varieties. In the case where $H=P$ is a parabolic subgroup, the unique $G/P$-horospherical variety is the flag variety $G/P$ itself; see \cite{lakshmibai2018flag} or \cite{brion2004lectures} for more on flag varieties. 

Let $P:=N_G(H)$ denote the \define{associated parabolic} to $H$, which is a parabolic subgroup of $G$ containing $B$. Then $H$ is the intersection of the kernels of some characters of $P$. This associated parabolic corresponds to a subset $I\subseteq S$, and we write $P=:P_I$ (see \cite[Section 29.3, Theorem]{humphreys1975linear}); note that we use the convention where $I$ is the set of simple roots for a Levi subgroup of $P=P_I$. 

Then $P/H$ is a torus which we call the \define{associated torus} to $G/H$. Let $N=N(G/H):=\frakX^*(P/H)$ denote the one-parameter subgroup lattice of the torus $P/H$; then $N^\vee=\frakX(P/H)$ is the dual lattice of characters. Note that $N^\vee$ is a sublattice of $\frakX(P)$, and $\frakX(P)$ is a sublattice of $\frakX(B)=\frakX(T)$. This character lattice $N^\vee$ describes $H$ as the intersection of the kernels of some characters of $P$, i.e. $H=\cap_{\chi\in N^\vee} \ker(\chi)$ where $\chi:P\to\G_m$ represents a character of $P$. 

The associated torus is naturally isomorphic to the group of $G$-equivariant automorphisms of the horospherical homogeneous space $G/H$, denoted $\Aut^G(G/H)$. In fact, $\Aut^G(G/H)=\Aut^G(X)$ (see \cite[Lemma 4.1]{altmann2015merging}). In the standard notation for tori, we can write $P/H=T_N$ to indicate that $P/H$ is the torus whose one-parameter subgroup lattice is $N$.\footnote{There is a natural isomorphism $N\otimes_\Z \G_m\isomap T_N$ via $\lambda\otimes t\mapsto \lambda(t)$, where $\lambda:\G_m\to T_N$ denotes a one-parameter subgroup. See the top of page 12 in \cite{cox2011toric}.} Thus, for any $G/H$-horospherical variety $X$, we have
\begin{align*}
	P/H = T_N = \Aut^G(G/H) = \Aut^G(X).
\end{align*}
This associated torus acts on $X$ via $G$-equivariant automorphisms.

\subsection{Colours}\label{subsec:colours} 

Recall that the \textit{opposite} Borel subgroup to $B$ is defined as the unique Borel subgroup $B^-$ such that $B\cap B^-=T$; e.g. when $G=SL_n$, with $B_n$ the upper triangular matrices, the opposite $B_n^-$, relative to the torus $T_n$, is the lower triangular matrices. The orbit $B^-\cdot eH\subseteq X$ of the base point under the opposite Borel subgroup is open in $X$. The complement $X\setminus (B^-\cdot eH)$ is a union of $B^-$-invariant prime (i.e. irreducible) divisors. These $B^-$-invariant prime divisors come in two flavours: either they are $G$-invariant, or they are $B^-$-invariant but not $G$-invariant; the latter is a \define{colour divisor} of $X$. In the case of toric varieties where $G=T$ is a torus and $H=\{1\}$, there are no colour divisors since $G=B^-=T$. 

It is worth noting that most sources define colour divisors to be $B$-invariant rather than $B^-$-invariant. We choose the latter so that the colour divisors are disjoint with the open Borel orbit of the base point. One could swap the roles of $B$ and $B^-$ (since $(B^-)^-=B$) to make colour divisors $B$-invariant, but then the horospherical subgroup $H$ would contain the maximal unipotent subgroup $U^-$ of $B^-$, and the associated parabolic $P$ would contain $B^-$.

Using the Bruhat decomposition (see \cite[Section 28]{humphreys1975linear} or \cite[Proposition 1.18]{pasquier2009introduction}), the flag variety $G/P$ has colour divisors $D_\alpha:=\ol{B^-r_\alpha P/P}\subseteq G/P$ for each $\alpha\in S\setminus I$, where $r_\alpha$ denotes a representative of the reflection in the Weyl group associated to the root $\alpha$. Then the colour divisors of $G/H$ are the pullbacks of the colour divisors of $G/P$ under the natural projection $G/H\to G/P$, and the colour divisors of $X$ are the closures of the colour divisors of $G/H\subseteq X$. 

In particular, the colour divisors of $G/P$, $G/H$, and $X$ are all in bijection with the set $\calC=\calC(G/H):=S\setminus I$. We call this set the \define{universal colour set} for $G/H$, and we call an element $\alpha\in\calC$ a \define{colour}. There is a \define{colour map} $\xi:\calC\to N$ which sends a colour $\alpha$ to a \define{colour point} $\xi(\alpha)=:u_\alpha\in N$, defined as follows. There is a natural isomorphism between the character lattice $N^\vee$ of $P/H$ and $k(G/H)^{(B^-)}/k^*$, the group of $B^-$-eigenfunctions of $k(G/H)$ (see \cite[Section 1]{knop1991luna}), which sends a character $\chi$ to the $B^-$-eigenfunction $f$ (defined up to an invertible constant) for which $b\cdot f=\chi(b)f$ for all $b\in B^-$. Via this isomorphism, the point $u_\alpha\in N=(N^\vee)^\vee$ is defined to be the functional sending $f\in k(G/H)^{(B^-)}$ to $\nu_{D_\alpha}(f)$, where $\nu_{D_\alpha}$ is the order of vanishing valuation for the colour divisor $D_\alpha$ in $G/H$.

Now we can think of the lattice $N=N(G/H)$ equipped with this ``colour structure" $\xi:\calC\to N$ as the \define{coloured lattice} associated to $G/H$. Note that this coloured lattice is the same for any $G/H$-horospherical variety $X$ since it only depends on $G/H$. In the case of toric varieties where $G=T$ and $H=\{1\}$, the coloured lattice $N$ is just the usual one-parameter subgroup lattice for the torus $T$ (there are no colour points in this case).

\subsection{Coloured fans}\label{subsec:coloured fans}

We define a \define{coloured cone} on $N$ as a pair $\sigma^c=(\sigma,\calF)$ where $\sigma\subseteq N_\R:=N\otimes_\Z\R$ is a cone\footnote{For us, a \textit{cone} means a strongly convex (rational) polyhedral cone. We specify ``possibly non-strongly convex" if we wish to talk about general (convex rational) polyhedral cones.} and $\calF=\calF(\sigma^c)\subseteq \calC$ is a \define{colour set} such that $0\notin\xi(\calF)$ and $\xi(\calF)\subseteq\sigma$. A \define{coloured face} of $\sigma^c$ is a coloured cone $\tau^c=(\tau,\calF(\tau^c))$ where $\tau\subseteq\sigma$ is a face of $\sigma$ and $\calF(\tau^c)=\{\alpha\in\calF:u_\alpha\in\tau\}$. Containment and intersections of coloured cones are defined componentwise, i.e. $(\sigma_1,\calF_1)\subseteq (\sigma_2,\calF_2)$ if and only if $\sigma_1\subseteq\sigma_2$ and $\calF_1\subseteq\calF_2$, and $(\sigma_1,\calF_1)\cap(\sigma_2,\calF_2)=(\sigma_1\cap\sigma_2,\calF_1\cap\calF_2)$. 

A \define{coloured fan} $\Sigma^c$ on $N$ is a finite collection of coloured cones which is closed under taking coloured faces, and the intersection of two coloured cones in $\Sigma^c$ is a coloured face of both. Note that a coloured cone $\sigma^c$ can be viewed as a coloured fan by taking the (finite) collection of all coloured faces of $\sigma^c$, i.e. $\sigma^c$ can be viewed as the coloured fan whose unique maximal coloured cone is $\sigma^c$; we say that such a coloured fan is ``generated by a single coloured cone". 

We let $\calF(\Sigma^c):=\cup_{(\sigma,\calF)\in\Sigma^c} \calF$ denote the colour set for $\Sigma^c$, we let $\Sigma:=\{\sigma:(\sigma,\calF)\in\Sigma^c\}$ denote the underlying fan for $\Sigma^c$, and we let $|\Sigma^c|:=|\Sigma|:=\cup_{\sigma\in\Sigma} \sigma$ denote the support of $\Sigma^c$. When $G=T$ is a torus, there are no colours, so coloured cones are the same as cones, and coloured fans are the same as fans.

There is an explicit bijective correspondence between $G/H$-horospherical varieties (up to $G$-equivariant isomorphism) and coloured fans on $N$ (see \cite[Theorem 3.3]{knop1991luna}), generalizing the well-known correspondence between toric varieties and fans in the case where $G=T$. We let $X_{\Sigma^c}$ denote the $G/H$-horospherical variety (defined up to $G$-equivariant isomorphism) associated to the coloured fan $\Sigma^c$, or conversely we let $\Sigma^c(X)$ denote the coloured fan associated to the $G/H$-horospherical variety $X$. We let $\calF(X)$ denote the \define{colour set} for $X$, i.e. $\calF(X):=\calF(\Sigma^c(X))$. 

There is also an explicit bijective correspondence between the $G$-orbits of $X_{\Sigma^c}$ and the coloured cones of $\Sigma^c$, which is order-reversing with respect to containment of coloured cones and containment of $G$-orbit closures. In particular, the \define{non-coloured rays} of $\Sigma^c$, i.e. coloured cones of the form $\rho^c=(\rho,\varnothing)\in\Sigma^c$ where $\rho\in\Sigma$ is a ray, correspond to the $G$-invariant prime divisors in $X_{\Sigma^c}$.

A \define{simple} horospherical variety is one in which there is a unique closed $G$-orbit. The aforementioned correspondence between horospherical varieties and coloured fans restricts to a bijective correspondence between simple $G/H$-horospherical varieties (up to $G$-equivariant isomorphism) and coloured cones (viewed as coloured fans) on $N$. In particular, the horospherical homogeneous space $G/H$ is simple and corresponds to the trivial coloured cone $0^c:=(0,\varnothing)$. Moreover, $X_{\Sigma^c}$ admits an open cover by the simple horospherical varieties $X_{\sigma^c}$ as $\sigma^c$ ranges over all the coloured cones in $\Sigma^c$. In the case of toric varieties, the simple toric varieties are just the affine toric varieties. However, simple horospherical varieties need not be affine. In fact, a $G/H$-horospherical variety $X$ is affine if and only if $X$ is simple and $\calF(X)=\calC$.

\subsection{Horospherical morphisms and maps of coloured fans}\label{subsec:horospherical morphisms and maps of coloured fans}

Let $X_i$ be $G/H_i$\-/horospherical varieties for $i=1,2$, where the $H_i\subseteq G$ are horospherical subgroups. A \define{horospherical morphism} $\pi:X_1\to X_2$ is a $G$-equivariant morphism which preserves the base point. These conditions imply that $\pi$ restricts to a map of open $G$-orbits $G/H_1\to G/H_2$ sending $eH_1\mapsto eH_2$. This is equivalent to saying that $H_1\subseteq H_2$ and $G/H_1\to G/H_2$ is the natural projection. 

We use all the same notation developed for horospherical varieties, but now with a subscript $i=1,2$. For instance, $N_i$ is the coloured lattice associated to $G/H_i$ for $i=1,2$. If $H_1\subseteq H_2$, then the natural projection $G/H_1\to G/H_2$ descends to a map $G/P_1\to G/P_2$. In particular, we have $P_1\subseteq P_2$, so the map $G/H_1\to G/H_2$ restricts to a surjective group homomorphism of associated tori $T_{N_1}=P_1/H_1\to T_{N_2}=P_2/H_2$. This map of tori induces an injective map of character lattices $N_2^\vee\hookrightarrow N_1^\vee$,\footnote{Going between maps of tori and maps of their character lattices is done explicitly via the \textit{Cartier dual}; recall that this is an anti-equivalence of categories $K\mapsto \Hom_\Z(K,\G_m)$ between diagonalizable linear algebraic groups and finitely generated abelian groups.} and a dual map (with finite cokernel) of one-parameter subgroup lattices $\Phi:N_1\to N_2$. In general, for a map $\phi:N_1\to N_2$ of lattices, we let $\phi_\Gm:T_{N_1}\to T_{N_2}$ denote the corresponding map of tori; in a similar way, we let $\phi_\R:(N_1)_\R:=N_1\otimes_\Z \R \to (N_2)_\R:=N_2\otimes_\Z\R$ denote the $\R$-linear extension of $\phi$. 

We call the above map $\Phi:N_1\to N_2$ the \define{associated map} of coloured lattices (associated to $G/H_1\to G/H_2$). Since $P_1\subseteq P_2$, we have $I_1\subseteq I_2$, so we have an inclusion of universal colour sets $\calC_2\subseteq \calC_1$. Intuitively, we think of horospherical morphisms as maintaining or dropping colours, but not gaining colours. For the colours which are dropped, i.e. the colours in $I_2\setminus I_1=\calC_1\setminus\calC_2=:\calC_\Phi$, we say that these are \define{mapped dominantly}. Note that $\calC_1=\calC_\Phi\sqcup\calC_2$ (disjoint union). An important property of $\Phi$ is that it sends non-dominantly mapped colour points to colour points, i.e. $\Phi\circ\xi_1=\xi_2$ on $\calC_1\setminus\calC_\Phi=\calC_2$. On the other hand, dominantly mapped colour points are sent to zero, i.e. $\Phi\circ\xi_1=0$ on $\calC_\Phi$. 

Any horospherical morphism $\pi:X_1\to X_2$ restricts to $G/H_1\to G/H_2$. On the other hand, if we are given coloured fans $\Sigma_i^c$ on $N_i$ for $i=1,2$, we say that the map of coloured lattices $\Phi:N_1\to N_2$ associated to $G/H_1\to G/H_2$ is \define{compatible with $\Sigma_1^c$ and $\Sigma_2^c$} if, for each coloured cone $\sigma_1^c=(\sigma_1,\calF_1)\in\Sigma_1^c$, there exists $\sigma_2^c=(\sigma_2,\calF_2)\in\Sigma_2^c$ such that $\Phi_\R(\sigma_1)\subseteq \sigma_2$ and $\calF_1\setminus\calC_\Phi\subseteq \calF_2$. An important result is that $\Phi:N_1\to N_2$ is compatible with $\Sigma_1^c$ and $\Sigma_2^c$ if and only if $G/H_1\to G/H_2$ extends to a horospherical morphism $X_{\Sigma_1^c}\to X_{\Sigma_2^c}$ (see \cite[Theorem 4.1]{knop1991luna}). The dominantly mapped colours $\alpha\in\calC_\Phi$ under $\Phi:N_1\to N_2$ correspond to the colour divisors which are mapped dominantly by $\pi:X_1\to X_2$. 

If $\Phi:N_1\to N_2$ denotes the map of coloured lattices associated to $G/H_1\to G/H_2$, then this map $G/H_1\to G/H_2$ is an isomorphism if and only if $\Phi$ is an \define{isomorphism of coloured lattices}, i.e. $\Phi$ is a $\Z$-linear isomorphism and $\calC_1=\calC_2$.\footnote{Compare this criterion with \cite[Theorem 1]{losev2009uniqueness}. In this theorem, there is a condition on the stabilizers of colour divisors. This condition is satisfied in our situation when $H_1$ and $H_2$ have the same associated parabolic, which is the case when $\calC_1=\calC_2$.}

\subsection{Expanding groups acting on horospherical varieties}\label{subsec:expanding groups acting on horospherical varieties}

Let $X$ be a $G/H$\-/horospherical variety. Suppose that $G'$ is a connected reductive linear algebraic group, $T'$ is a torus, and we have a homomorphism of algebraic groups $\varepsilon:G'\times T'\to G$ which restricts to an isogeny $G'\to G$.\footnote{This means $G'\to G$ is a homomorphism of algebraic groups which is surjective and has finite kernel.} Then $G'\times T'$ acts on $X$ via $\varepsilon$ and we can view $X$ as a $G'\times T'/\varepsilon^{-1}(H)$-horospherical variety. Note that $G'\times T'/\varepsilon^{-1}(H)$ is $\varepsilon$-equivariantly isomorphic to $G/H$. 

Since $G'$ is a finite cover of $G$, it has the same root system as $G$, and $T'$ has a trivial (i.e. empty) root system. Thus, $G'\times T'$ has the same root system as $G$ (relative to a maximal torus and Borel subgroup which map onto $T$ and $B$, respectively, under $\varepsilon$). This combined with the fact that $G'\times T'/\varepsilon^{-1}(H)\cong G/H$ tells us that the combinatorics of $X$ is the same whether we view it as a horospherical $G$-variety or a horospherical $G'\times T'$-variety. In particular, the coloured lattice $N$ is unchanged, and the coloured fan for $X$ is unchanged.

This allows us to compare a horospherical $G$-variety $X$ with a horospherical $G'\times T'$-variety $X'$, because we can view $X$ as a horospherical $G'\times T'$-variety so that $X$ and $X'$ are horospherical varieties with respect to the same group. In particular, this allows us to consider horospherical morphisms $X'\to X$. We implicitly use this method throughout the paper.

\section{Horospherical stacks and stacky coloured fans}\label{sec:horospherical stacks and stacky coloured fans}

In this section we define the central objects of this paper: horospherical stacks and stacky coloured fans. These objects generalize the toric stacks and stacky fans defined in \cite[Section 2]{satriano2015toric1}. We describe the connection between the two objects (in \cref{subsec:stacky coloured fan to horospherical stack} and \cref{subsec:horospherical stack to stacky coloured fan}), give some basic properties, and give many examples at the end (\cref{subsec:examples - horospherical stacks and stacky coloured fans}).

For the reader who is unfamiliar with stacks, we suggest treating the quotient stack $[X/K]$ as a formal quotient of a variety $X$ by the action of a linear algebraic group $K$. If one wishes to know more about algebraic stacks, we recommend \cite{olsson2016algebraic}.

\subsection{Horospherical stacks}\label{subsec:horospherical stacks}

We define the first main object of this paper: horospherical stacks. The idea is to take a quotient stack of a $G/H$-horospherical variety $X$ by a closed subgroup $K$ of the associated torus $P/H=T_N=\Aut^G(X)$. For any subgroup $K$ of $\Aut^G(X)$, we let $\ol{K}$ denote the corresponding subgroup of $G$, i.e. so that $K=\ol{K}/H\subseteq P/H=\Aut^G(X)$. Note that $\ol{K}$ is a horospherical subgroup since it contains the horospherical subgroup $H$. 

\begin{definition}[Horospherical stack]\label{def:horospherical stack}
	A \define{$G/\ol{K}$-horospherical stack} is an algebraic $G$-stack of the form $[X/K]$ where $X$ is a horospherical $G$-variety and $K$ is a closed subgroup of $\Aut^G(X)$. 
\end{definition}

A \define{horospherical $G$-stack} is a stack which is a $G/\ol{K}$-horospherical stack for some horospherical subgroup $\ol{K}\subseteq G$. A \define{horospherical stack} is a stack which is a horospherical $G$-stack for some connected reductive linear algebraic group $G$.

\begin{remark}[Open $G$-orbit of a horospherical stack]\label{rmk:open G-orbit of horospherical stack}
	If $\calX$ is a $G/\ol{K}$-horospherical stack, then $G/\ol{K}$ is its open $G$-orbit.
\end{remark}

The natural base point of the $G/H$-horospherical variety $X$ is $eH\in G/H\subseteq X$. We use this to define the \define{base point} of the $G/\ol{K}$-horospherical stack $[X/K]$ as the point $e\ol{K}\in G/\ol{K}$ of the open $G$-orbit of $[X/K]$. 

\begin{remark}[Horospherical varieties are horospherical stacks]\label{rmk:horospherical varieties are horospherical stacks}
	If $X$ is a $G/H$-horospherical variety, then $X=[X/\{1\}]$ is a $G/H$-horospherical stack. 
\end{remark}

\begin{remark}[Toric stacks are horospherical stacks]\label{rmk:toric stacks are horospherical stacks}
	We use the definition of \textit{toric stack} from \cite[Definition 1.1]{satriano2015toric1}. If $[X/K]$ is a toric stack where $X$ is a toric variety with torus $T$, then $[X/K]$ is a $T/K$-horospherical stack since $K\subseteq T=\Aut^T(X)$. In this case, $K=\ol{K}$. 
	
	Note, however, that we do not give an analogue of \textit{non-strict} toric stacks, which were defined to allow for toric stacks with non-trivial generic stabilizer. 
\end{remark}

\begin{remark}[Alternative definition of horospherical stack]\label{rmk:comparing definition of horospherical stack to other paper}
	Note that our definition of horospherical stack is slightly different than \cite[Definition 3.6]{terpereau2019horospherical}. They define a horospherical $G$-stack as $[X/K]$ where $X$ is a horospherical $G\times T_0$-variety, where $T_0$ is a torus acting faithfully on $X$, and $K\subseteq \Aut^{G\times T_0}(X)$ is a closed subgroup containing $T_0$; this definition just allows for a torus factor to be added to the group $G$, which is then ``cancelled out" in the quotient by $K$. Although, using the idea from \cref{subsec:expanding groups acting on horospherical varieties}, one can view a horospherical $G$-stack $\calX$ (as per \cref{def:horospherical stack}) as a horospherical $G\times T_0$-stack for any torus $T_0$, where we let $T_0$ act trivially. 
\end{remark}

\subsection{Stacky coloured fans}\label{subsec:stacky coloured fans}

Now we introduce the second main object of this paper: stacky coloured fans. These are the combinatorial counterpart to horospherical stacks.

\begin{definition}[Stacky coloured fan]\label{def:stacky coloured fan}
	A \define{stacky coloured fan} is a pair $(\Sigma^c,\beta)$ satisfying the following:
	\begin{enumerate}[leftmargin=2cm, label={(SCF\arabic*)}]
		\item\label{item:SCF1} $\Sigma^c$ is a coloured fan on $N$; recall that $N$ is the coloured lattice associated to some horospherical homogeneous space $G/H$.
		\item\label{item:SCF2} $\beta:N\to L$ is a $\Z$-linear map of lattices with finite cokernel.
	\end{enumerate}
\end{definition}

\begin{remark}[Characterizations of finite cokernel]\label{rmk:characterizations of finite cokernel}
	If $\beta:N\to L$ is a $\Z$-linear map, then it induces a map of tori $\beta_\Gm:T_N\to T_L$ naturally identifying $\beta$ with the map of one-parameter subgroup lattices. Note that the dual map $\beta^\vee:L^\vee\to N^\vee$ is the corresponding map of character lattices of the tori. The following conditions are equivalent:
	\begin{enumerate}
		\item $\beta:N\to L$ has finite cokernel.
		\item $\beta_\R:N_\R\to L_\R$ is surjective. 
		\item $\beta^\vee:L^\vee\to N^\vee$ is injective.
		\item $\beta_\Gm:T_N\to T_L$ is surjective.
	\end{enumerate}
\end{remark}

\begin{remark}[Base of a stacky coloured fan is a coloured lattice]\label{rmk:base of a stacky coloured fan is a coloured lattice}
	If $(\Sigma^c,\beta:N\to L)$ is a stacky coloured fan, then the base lattice $L$ inherits a colour structure through $\beta$: the universal colour set for $L$ is the same as $N$ (i.e. $\calC$) and the colour points of $L$ are the images of the colour points of $N$ under $\beta$. Moreover, $L$ is the coloured lattice associated to some horospherical homogeneous space. 
	
	Indeed, since $\beta:N\to L$ has finite cokernel, it induces a surjection of tori $\beta_\Gm:T_N\to T_L$. Letting $K_\beta:=\ker(\beta_\Gm)$, which we can write as $K_\beta=\ol{K_\beta}/H\subseteq P/H$, we have $T_L=P/\ol{K_\beta}$. Since $H\subseteq\ol{K_\beta}$, we see that $\ol{K_\beta}$ is horospherical, so the homogeneous space $G/\ol{K_\beta}$ is horospherical. It is easy to see that $\beta$ is the associated map to the natural projection $G/H\to G/\ol{K_\beta}$, so $L$ is the coloured lattice associated to $G/\ol{K_\beta}$, i.e. $L$ is $N(G/\ol{K_\beta})$. 
\end{remark}

Throughout this paper, we use the notation of \cref{def:stacky coloured fan} and \cref{rmk:base of a stacky coloured fan is a coloured lattice} for stacky coloured fans. If we want to highlight the full group, then we sometimes call $(\Sigma^c,\beta)$ a \define{$G$-stacky coloured fan}.

\begin{remark}[Stacky fans are stacky coloured fans]\label{rmk:stacky fans are stacky coloured fans}
	We use the definition of \textit{stacky fan} from \cite[Definition 2.4]{satriano2015toric1}. If $(\Sigma,\beta)$ is a stacky fan, then we can view this as a $T$-stacky coloured fan (the lattices $N$ and $L$ have no colour points in this case).
	
	Since we do not define an analogue of \textit{non-strict} toric stacks, we also do not define an analogue of \textit{non-strict} stacky fans. 
\end{remark}

\subsection{Stacky coloured fan $\to$ horospherical stack}\label{subsec:stacky coloured fan to horospherical stack}

We describe how a stacky coloured fan determines a horospherical stack. Let $(\Sigma^c,\beta:N\to L)$ be a stacky coloured fan. Then the coloured fan $\Sigma^c$ determines a $G/H$-horospherical variety $X_{\Sigma^c}$. The map $\beta:N\to L$ induces a homomorphism of tori $\beta_\Gm:T_N\to T_L$. Since $\coker(\beta)$ is finite, $\beta_\Gm$ is surjective. As in \cref{rmk:base of a stacky coloured fan is a coloured lattice}, consider $K_\beta=\ker(\beta_\Gm)\subseteq T_N=\Aut^G(X)$, which is a closed subgroup.

\begin{notation}\label{not:horospherical stack from stacky coloured fan}
	Using the notation above, if $(\Sigma^c,\beta)$ is a stacky coloured fan, we define the $G/\ol{K_\beta}$-horospherical stack $\calX_{\Sigma^c,\beta}$ to be the quotient stack $[X_{\Sigma^c}/K_\beta]$. 
\end{notation}

\begin{remark}\label{rmk:isomorphic stacky coloured fans vs isomorphic horospherical stacks}
	One can consider two stacky coloured fans to be ``the same" if they are isomorphic in the sense of \cref{def:map of stacky coloured fans}. The stacky coloured fan $(\Sigma^c,\beta)$ determines the pair ($X_{\Sigma^c}$,$K_\beta$), which provides a presentation for the horospherical stack $\calX_{\Sigma^c,\beta}$. It is important to note that different stacky coloured fans can yield different presentations of the same horospherical stack; see \cref{rmk:horospherical stacks can arise from different covers}. We describe exactly when different stacky coloured fans yield isomorphic horospherical stacks in \cref{thm:horospherical isomorphisms}. 
\end{remark}

\subsection{Horospherical stack $\to$ stacky coloured fan}\label{subsec:horospherical stack to stacky coloured fan}

Now we describe how every horospherical stack arises from a stacky coloured fan. Any $G/\ol{K}$-horospherical stack is of the form $[X/K]$ where $X$ is a $G/H$-horospherical variety (for some horospherical subgroup $H\subseteq G$) and $K=\ol{K}/H$ is a closed subgroup of $\Aut^G(X)=T_N$. Then the $G/H$-horospherical variety $X$ corresponds to a coloured fan $\Sigma^c$ on $N$. The surjection of tori $T_N\surjmap T_N/K=:T_L$ induces a map of one-parameter subgroup lattices $\beta:N\to L$. Since this map of tori is surjective, $\coker(\beta)$ is finite. As in \cref{rmk:base of a stacky coloured fan is a coloured lattice}, $L$ is the coloured lattice associated to the horospherical homogeneous space $G/\ol{K}$. Thus $(\Sigma^c,\beta)$ is a stacky coloured fan, and we have $[X/K]=\calX_{\Sigma^c,K_\beta}$.

\subsection{Examples}\label{subsec:examples - horospherical stacks and stacky coloured fans}

We give many examples of stacky coloured fans and their corresponding horospherical stacks. Diagrams for stacky coloured fans are drawn as follows: the coloured fan $\Sigma^c$ on the coloured lattice $N$ is on the left, the map $\beta:N\to L$ is shown at the top, and the coloured lattice $L$ is shown on the right. For a colour $\alpha$, we draw an ``un-filled" circle point marked $\alpha$ to indicate $u_\alpha$ in the coloured lattice; if we want to indicate that $\alpha\in\calF(\Sigma^c)$, then we ``fill in" this circle point (with the colour orange). 

\begin{example}\label{ex:stacky coloured fan with identity map}
	Let $\Sigma^c$ be a coloured fan on $N$. If $\beta:N\to N$ is an isomorphism, then the stacky coloured fan $(\Sigma^c,\beta)$ yields the $G/H$-horospherical stack $\calX_{\Sigma^c,\beta}=X_{\Sigma^c}$.
\end{example}

\begin{example}\label{ex:stacky coloured fan with trivial coloured fan}
	Given a $\Z$-linear map $\beta:N\to L$ with finite cokernel, the stacky coloured fan $(0^c,\beta)$ yields the $G/\ol{K_\beta}$-horospherical stack $[(G/H)/K_\beta]=G/\ol{K_\beta}$. 
\end{example}

\begin{remark}[Basis for character lattice with $G=SL_n$]\label{rmk:basis for character lattice with G=SL_n}
	When drawing examples of stacky coloured fans, we need to choose a basis for the character lattice $N^\vee$; then the basis for $N$ is the dual basis. Throughout this paper, unless otherwise specified, we use the following basis of $N^\vee$ when $G=SL_n$.
	
	Consider a horospherical homogeneous space $SL_n/H$. Note that $N(SL_n/H)^\vee$ is a sublattice of $\frakX(B_n)$. For $\frakX(B_n)\cong\Z^{n-1}$, we use the basis $\{e_1,\ldots,e_{n-1}\}$ where $e_i$ corresponds to the character $\chi_i:B_n\to \G_m$ sending a matrix $A$ to the determinant of the upper-left $i\times i$-submatrix, i.e. the product of the first $i$ diagonal entries of $A$ (since $A$ is upper triangular). Notice that the full $n\times n$ determinant is just the trivial character since we are in $SL_n$.
\end{remark}
		
\begin{remark}[Colour points with $G=SL_n$]\label{rmk:colour points with G=SL_n}
	Suppose that $G=SL_n$, and consider a horospherical homogeneous space $SL_n/H$. Recall that $\# S=n-1$ in this case, so we can write $S=\{\alpha_1,\ldots,\alpha_{n-1}\}$. Note that $S$ is the universal colour set for $N(SL_n/U_n)$. In the basis described in \cref{rmk:basis for character lattice with G=SL_n}, we have $u_{\alpha_i}=e_i\in N(SL_n/U_n)$ (here $e_i$ also denotes the dual basis vector to $e_i\in N(SL_n/U_n)^\vee$). Thus, the colour points of $N(SL_n/H)$ are the images of the $e_i$ under the map $N(SL_n/U_n)\to N(SL_n/H)$, except for the colours which are mapped dominantly (if $\alpha_i$ is mapped dominantly, then $e_i\mapsto 0$). 
\end{remark}

\begin{example}\label{ex:stacky coloured fan for A^2 mod Z_2 as SL_2}
	Consider $G=SL_2$ and the horospherical subgroup $H=U_2$. The lattice $N$ is isomorphic to $\Z$, the universal colour set is $\{\alpha\}$, and $u_\alpha=e_1$. Let $(\Sigma^c,\beta)$ be the stacky coloured fan given in the diagram below. 
	
	\tikzitfig{stacky_coloured_fan_for_A2_mod_Z_2_as_SL_2_example}
	
	In this case, $\Sigma^c$ is generated by a single coloured cone $(\Cone(e_1),\{\alpha\})$. We have $X_{\Sigma^c}=\A^2$, where $SL_2$ acts by matrix multiplication; we can take the base point to be $(1,0)\in\A^2$. 
	
	The corresponding map of tori $\beta_\Gm:T_N=\G_m\to T_L=\G_m$ sends $t\mapsto t^2$. Thus, $K_\beta=\{\zeta:\zeta^2=1\}\subseteq \G_m$. Based on \cref{rmk:basis for character lattice with G=SL_n}, this yields the subgroup 
	\begin{align*}
		\ol{K_\beta} = \left\{\begin{bmatrix}
			\zeta & *
			\\ 0 & \zeta^{-1}
		\end{bmatrix} ~:~ \zeta^2=1 \right\}\subseteq SL_2.
	\end{align*}
	
	The action of $T_N$ on $X_{\Sigma^c}=\A^2$ is via $SL_2$-equivariant automorphisms. Thus, $T_N=\G_m$ acts by scaling the coordinates of $\A^2$, so $K_\beta\cong \Z/2\Z$ acts by scaling (changing the sign of) the coordinates of $\A^2$. 
	
	Therefore, we have $\calX_{\Sigma^c,\beta}=[\A^2/(\Z/2\Z)]$, which is a $SL_2/\ol{K_\beta}$-horospherical stack. Note that this is different from the singular cone $\{xy=z^2\}\subseteq \A^3_{(x,y,z)}$ (although this variety is the good moduli space of $\calX_{\Sigma^c,\beta}$; see \cref{subsec:good moduli space of a horospherical stack}). 
\end{example}

\begin{example}\label{ex:stacky coloured fan for P^2 as SL_2 quotient}
	Consider $G=SL_2\times\G_m$ and the horospherical subgroup $H=U_2\times\{1\}$. The lattice $N$ is isomorphic to $\Z^2$, the universal colour set is $\calC=\{\alpha\}$, and $u_\alpha=e_1$. Let $(\Sigma^c,\beta)$ be the stacky coloured fan given in the diagram below. 
	
	\tikzitfig{stacky_coloured_fan_for_P2_as_SL_2_example}
	
	In this case, $\Sigma^c$ has two maximal coloured cones: $(\Cone(e_1),\{\alpha\})$ and $(\Cone(e_2),\varnothing)$. We have $X_{\Sigma^c}=\A^3\setminus\{0\}$, where $SL_2$ acts by matrix multiplication on the first two coordinates and $\G_m$ acts by scaling the last coordinate; we can take the base point to be $(1,0,1)\in\A^3$.
	
	The corresponding map of tori $\beta_\Gm:T_N=\G_m^2\to T_L=\G_m$ sends $(s,t)\mapsto st^{-1}$. Thus $K_\beta=\{(s,s):s\in\G_m\}$, which is the diagonal $\G_m$ inside of $T_N=\G_m^2$. Based on \cref{rmk:basis for character lattice with G=SL_n}, this yields the subgroup
	\begin{align*}
		\ol{K_\beta} = \left\{\left(\begin{bmatrix}
			s & *
			\\ 0 & s^{-1}
		\end{bmatrix} ~,~ s\right) ~:~ s\in\G_m \right\}\subseteq SL_2\times\G_m.
	\end{align*}
	
	The action of $T_N$ on $X_{\Sigma^c}=\A^3\setminus\{0\}$ is via $SL_2\times\G_m$-equivariant automorphisms. Thus, $(s,t)\in T_N=\G_m^2$ acts by scaling the first two coordinates of $\A^3\setminus\{0\}$ by $s$ and scaling the last coordinate by $t$. So $K_\beta=\G_m$ acts by scaling the coordinates of $\A^3\setminus\{0\}$. 
	
	Therefore, we have $\calX_{\Sigma^c,\beta}=[(\A^3\setminus\{0\})/\G_m]=\P^2$. Note that this is a $SL_2\times\G_m/\ol{K_\beta}\cong SL_2/U_2$-horospherical stack. 
\end{example}

\begin{remark}[Horospherical stacks can arise from different covers]\label{rmk:horospherical stacks can arise from different covers}
	From \cref{ex:stacky coloured fan with identity map} and \cref{ex:stacky coloured fan for P^2 as SL_2 quotient}, we see that the $SL_2/U_2$-horospherical stack $\P^2$ can be written in (at least) two different ways as the quotient stack of a horospherical variety by a diagonalizable group. In particular, these two different ways of writing $\P^2$ use different (non-isomorphic) stacky coloured fans. Note that the coloured fan for $\P^2$ as a $SL_2/U_2$-horospherical variety has two maximal coloured cones: $(\Cone(e_1),\{\alpha\})$ and $(\Cone(-e_1),\varnothing)$ on the coloured lattice $N=\Z$. 
\end{remark}

\begin{example}\label{ex:stacky coloured fan for non-toric SL_3 example with beta=[1 1]}
	Consider $G=SL_3$ and the horospherical subgroup $H=U_3$. The lattice $N$ is isomorphic to $\Z^2$, the universal colour set is $\calC=\{\alpha_1,\alpha_2\}$, and we have $u_{\alpha_1}=e_1$ and $u_{\alpha_2}=e_2$. Let $(\Sigma^c,\beta)$ be the stacky coloured fan given in the diagram below. 
	
	\tikzitfig{stacky_coloured_fan_for_non-toric_SL_3_example_with_beta=11}
	
	In this case, $\Sigma^c$ is generated by a single coloured cone $(\Cone(e_1,e_2),\{\alpha_1,\alpha_2\})$. The variety $X_{\Sigma^c}$ is $\{ux+vy+wz=0\}\subseteq \A^3_{(u,v,w)}\times\A^3_{(x,y,z)}$, where $SL_3$ acts by matrix multiplication on the $(u,v,w)$-coordinates and acts by contragredient matrix multiplication (i.e. the inverse transpose action) on the $(x,y,z)$-coordinates; we can take the base point to be $(1,0,0,0,0,1)$ in the $(u,v,w,x,y,z)$-coordinates.

	The corresponding map of tori $\beta_\Gm:T_N=\G_m^2\to T_L=\G_m$ sends $(s,t)\mapsto st$. Thus, $K_\beta=\{(s,s^{-1}):s\in\G_m\}$ inside of $T_N=\G_m^2$. Based on \cref{rmk:basis for character lattice with G=SL_n}, this yields the subgroup
	\begin{align*}
		\ol{K_\beta} = \left\{\begin{bmatrix}
			s & * & *
			\\ 0 & * & *
			\\ 0 & 0 & s
		\end{bmatrix} ~:~ s\in\G_m \right\}\subseteq SL_3.
	\end{align*}
	
	The action of $T_N$ on $X_{\Sigma^c}$ is via $SL_3$-equivariant automorphisms. Thus, $(s,t)\in T_N=\G_m^2$ acts by scaling the $(u,v,w)$-coordinates by $s$ and scaling the $(x,y,z)$-coordinates by $t$. So $s\in K_\beta=\G_m$ acts by scaling the $(u,v,w)$-coordinates by $s$ and scaling the $(x,y,z)$-coordinates by $s^{-1}$. 
	
	This gives the $SL_3/\ol{K_\beta}$-horospherical stack $\calX_{\Sigma^c,\beta}=[X_{\Sigma^c}/K_\beta]$. 
\end{example}

\begin{example}\label{ex:stacky coloured fan for non-toric SL_3 example with beta=[2 1 0 1]}
	Consider $G=SL_3$ and the horospherical subgroup $H=U_3$. The lattice $N$ is isomorphic to $\Z^2$, the universal colour set is $\calC=\{\alpha_1,\alpha_2\}$, and we have $u_{\alpha_1}=e_1$ and $u_{\alpha_2}=e_2$. Let $(\Sigma^c,\beta)$ be the stacky coloured fan given in the diagram below.
	
	\tikzitfig{stacky_coloured_fan_for_non-toric_SL_3_example_with_beta=21,01}
	
	In this case, $\Sigma^c$ is generated by a single coloured cone $(\Cone(e_1,e_2),\varnothing)$. Using \cref{subsec:local structure and affine reduction}, we can describe $X_{\Sigma^c}$ as $SL_3\times^{B_3} \A^2$.
	
	The corresponding map of tori $\beta_\Gm:T_N=\G_m^2\to T_L=\G_m^2$ sends $(s,t)\mapsto (s^2t,t)$. Thus, $K_\beta=\{(\zeta,1):\zeta^2=1\}$ inside of $T_N=\G_m^2$. Based on \cref{rmk:basis for character lattice with G=SL_n}, this yields the subgroup
	\begin{align*}
		\ol{K_\beta} = \left\{ \begin{bmatrix}
			\zeta & * & *
			\\0 & \zeta^{-1} & *
			\\ 0 & 0 & 1
		\end{bmatrix} ~:~ \zeta^2=1 \right\} \subseteq SL_3.
	\end{align*}

	The action of $T_N$ on $X_{\Sigma^c}$ is via $SL_3$-equivariant automorphisms; this can be viewed as usual $\G_m^2$ multiplication on the $\A^2$ in $SL_3\times^{B_3}\A^2$. So $K_\beta\cong \Z/2\Z$ acts by swapping the sign of the first coordinate in this $\A^2$. 
	
	This gives the $SL_3/\ol{K_\beta}$-horospherical stack $\calX_{\Sigma^c,\beta}=[X_{\Sigma^c}/K_\beta]$. 
\end{example}

\begin{remark}[Toric vs. non-toric examples]\label{rmk:toric vs non-toric examples - stacky coloured fans}
	The horospherical stacks in \cref{ex:stacky coloured fan for A^2 mod Z_2 as SL_2} and \cref{ex:stacky coloured fan for P^2 as SL_2 quotient} are both toric, so they can also be studied using the theory of toric stacks. However, the other examples in this subsection are not toric; see \cref{prop:non-toric condition}.
\end{remark}

\section{Horospherical morphisms and maps of stacky coloured fans}\label{sec:horospherical morphisms and maps of stacky coloured fans}

In this section, we discuss morphisms of horospherical stacks and how they relate to maps of the corresponding stacky coloured fans (one direction is \cref{subsec:map of stacky coloured fans to horospherical morphism}, and the other is described in \cref{thm:horospherical morphisms}). In particular, we give a combinatorial description of horospherical isomorphisms (in \cref{thm:horospherical isomorphisms}).

\subsection{Horospherical morphisms}\label{subsec:horospherical morphisms}

We first define the types of morphisms that we want to consider between two horospherical stacks; these should preserve the group action and the base points. 

\begin{definition}[Horospherical morphism]\label{def:horospherical morphism}
	Let $\calX_i=[X_i/K_i]$ be $G/\ol{K_i}$-horospherical stacks for $i=1,2$. A \define{horospherical morphism} $\pi:\calX_1\to \calX_2$ is a $G$-equivariant morphism of stacks which preserves the base point (i.e. sends the base point of $\calX_1$ to the base point of $\calX_2$).
\end{definition}

Note that, for a horospherical morphism $\pi:\calX_1\to \calX_2$, the condition of preserving the base point means that $\ol{K_1}\subseteq\ol{K_2}$ and $\pi$ restricts to a map of open orbits $G/\ol{K_1}\to G/\ol{K_2}$ which is the natural projection map. This is similar to horospherical morphisms of varieties (see \cref{subsec:horospherical morphisms and maps of coloured fans}). Clearly horospherical morphisms of horospherical varieties are horospherical morphisms in the sense of \cref{def:horospherical morphism}, when viewing the horospherical varieties as horospherical stacks via \cref{rmk:horospherical varieties are horospherical stacks}.

\begin{remark}[Dominant toric morphisms are horospherical morphisms]\label{rmk:dominant toric morphisms are horospherical morphisms}
	A morphism of toric stacks $\pi:\calX_1\to\calX_2$ restricts to a group homomorphism $T_1\to T_2$ of the open tori such that $\pi$ is equivariant with respect to this group homomorphism; see \cite[Definition 3.1]{satriano2015toric1}. If this restricted map $T_1\to T_2$ is surjective, then we can view $\calX_1$ and $\calX_2$ as horospherical $T_1$-stacks (since $T_2$ is a quotient of $T_1$), and $\pi$ is a horospherical morphism.
\end{remark}

\subsection{Maps of stacky coloured fans}\label{subsec:maps of stacky coloured fans}

The combinatorial counterpart to horospherical morphisms are maps of stacky coloured fans.

\begin{definition}[Map of stacky coloured fans]\label{def:map of stacky coloured fans}
	Let $(\Sigma^c_i,\beta_i)$ be stacky coloured fans for $i=1,2$; we use the notation of \cref{subsec:stacky coloured fans} but with subscripts $i=1,2$. A \define{map of stacky coloured fans} is a pair of $\Z$-linear maps $\Phi:N_1\to N_2$ and $\phi:L_1\to L_2$ which satisfy the following conditions:
	\begin{enumerate}[leftmargin=2.5cm, label={(MSCF\arabic*)}]
		\item\label{item:MSCF1} $\Phi$ is a map of coloured lattices (in the sense of \cref{subsec:horospherical morphisms and maps of coloured fans}).
		\item\label{item:MSCF2} $\Phi$ is compatible with $\Sigma_1^c$ and $\Sigma_2^c$.
		\item\label{item:MSCF3} We have $\beta_2\circ\Phi=\phi\circ\beta_1$.
	\end{enumerate}
\end{definition}

We can draw a map of stacky coloured fans $(\Phi,\phi):(\Sigma_1^c,\beta_1)\to (\Sigma_2^c,\beta_2)$ as the following commutative diagram:
\begin{equation*}
	\begin{tikzcd}
		{\Sigma_1^c} & {\Sigma_2^c} \\[-2em]
		{N_1} & {N_2} \\
		{L_1} & {L_2}
		\arrow["\Phi", from=2-1, to=2-2]
		\arrow["\phi", from=3-1, to=3-2]
		\arrow["{\beta_1}"', from=2-1, to=3-1]
		\arrow["{\beta_2}", from=2-2, to=3-2]
		\arrow[from=1-1, to=1-2]
	\end{tikzcd}
\end{equation*}

\begin{remark}[Base map in a map of stacky coloured fans]\label{rmk:base map in map of stacky coloured fans}
	Let $(\Phi,\phi):(\Sigma_1^c,\beta_1)\to (\Sigma_2^c,\beta_2)$ be a map of stacky coloured fans using the notation in \cref{def:map of stacky coloured fans}. By \cref{rmk:base of a stacky coloured fan is a coloured lattice}, we know that $L_i$ is the coloured lattice associated to $G/\ol{K_i}$. We necessarily have $\ol{K_1}\subseteq\ol{K_2}$ and the map $\phi:L_1\to L_2$ is the map associated to the natural projection $G/\ol{K_1}\to G/\ol{K_2}$. 
\end{remark}

Throughout this paper, we use the notation of \cref{def:map of stacky coloured fans} and \cref{rmk:base map in map of stacky coloured fans} for maps of stacky coloured fans.

\begin{remark}[Dominantly mapped colours for a map of stacky coloured fans]\label{rmk:dominantly mapped colours for a map of stacky coloured fans}
	Given a map of stacky coloured fans $(\Phi,\phi):(\Sigma_1^c,\beta_1)\to (\Sigma_2^c,\beta_2)$, the commuting condition \labelcref{item:MSCF3} and the fact that $\calC_{\beta_1}=\calC_{\beta_2}=\varnothing$ imply $\calC_\Phi=\calC_\phi$. 
\end{remark}

\begin{remark}[Map of stacky coloured fans has finite cokernel]\label{rmk:map of stacky coloured fan has finite cokernel}
	If $(\Phi,\phi)$ is a map of stacky coloured fans, then $\Phi$ and $\phi$ both have finite cokernel. Indeed, this follows from \cref{rmk:characterizations of finite cokernel} and because $\Phi:N_1\to N_2$ and $\phi:L_1\to L_2$ correspond to surjections of tori $\Phi_\Gm:T_{N_1}\to T_{N_2}$ and $\phi_\Gm:T_{L_1}\to T_{L_2}$, respectively.
\end{remark}

\subsection{Map of stacky coloured fans $\to$ horospherical morphism}\label{subsec:map of stacky coloured fans to horospherical morphism}

We describe how a map of stacky coloured fans determines a horospherical morphism of the corresponding horospherical stacks. Let $\calX_{\Sigma_i^c,\beta_i}$ be $G/\ol{K_i}$-horospherical stacks coming from stacky coloured fans $(\Sigma_i^c,\beta_i)$ for $i=1,2$. Let $(\Phi,\phi):(\Sigma_1^c,\beta_1)\to(\Sigma_2^c,\beta_2)$ be a map of the stacky coloured fans. By \labelcref{item:MSCF2}, we know that $\Phi$ is compatible with $\Sigma_1^c$ and $\Sigma_2^c$. Thus, the map $\Phi:N_1\to N_2$ induces a horospherical morphism $X_{\Sigma_1^c}\to X_{\Sigma_2^c}$ of horospherical $G$-varieties. By the commuting condition \labelcref{item:MSCF3}, the map of tori $\Phi_\Gm:T_{N_1}\to T_{N_2}$ restricts to a map $K_{\beta_1}\to K_{\beta_2}$. It follows that the map $X_{\Sigma_1^c}\to X_{\Sigma_2^c}$ induces a map on the quotient stacks 
\begin{align*}
	\calX_{\Sigma_1^c,\beta_1}=[X_{\Sigma_1^c}/K_{\beta_1}]\to [X_{\Sigma_2^c}/K_{\beta_2}]=\calX_{\Sigma_2^c,\beta_2}.
\end{align*}
By construction, this map is $G$-equivariant and preserves the base point (since the map $X_{\Sigma_1^c}\to X_{\Sigma_2^c}$ is horospherical), so this map is horospherical.

\begin{notation}\label{not:horospherical morphism from map of stacky coloured fans}
	Using the notation above, if $(\Phi,\phi):(\Sigma_1^c,\beta_1)\to(\Sigma_2^c,\beta_2)$ is a map of stacky coloured fans, we let $\pi_{\Phi,\phi}$ denote the associated horospherical morphism $\calX_{\Sigma_1^c,\beta_1}\to\calX_{\Sigma_2^c,\beta_2}$. 
\end{notation}

\subsection{Horospherical morphism $\to$ map of stacky coloured fans}\label{subsec:horospherical morphism to map of stacky coloured fans}

Now we describe how any horospherical morphism of horospherical stacks relates to a map of stacky coloured fans. This process is not as simple as the other direction above. The reason for this is that a morphism $X_1\to X_2$ can determine a morphism of the quotient stacks $[X_1/K_1]\to [X_2/K_2]$, but a morphism of the quotient stacks may not correspond to a morphism of \textit{these specific covers} $X_1$ and $X_2$. The following theorem tells us that, up to changing the presentation of the domain quotient stack, a horospherical morphism of horospherical stacks comes from a map of the corresponding stacky coloured fans; this is a natural generalization of \cite[Theorem 3.4]{satriano2015toric1}.

\begin{theorem}[Horospherical morphisms]\label{thm:horospherical morphisms}
	Let $\calX_i:=\calX_{\Sigma_i^c,\beta_i}$ be $G/\ol{K_i}$-horospherical stacks coming from stacky coloured fans $(\Sigma_i^c,\beta_i:N_i\to L_i)$ for $i=1,2$. Note that we can regard each $\calX_i$ as a horospherical $G\times K_{\beta_2}^\circ$-stack by letting the torus $K_{\beta_2}^\circ$ act trivially. Suppose that we have a horospherical morphism $\pi:\calX_1\to \calX_2$. Then there exists a stacky coloured fan $(\Sigma_0^c,\beta_0:N_0\to L_0)$, yielding a horospherical $G\times K_{\beta_2}^\circ$-stack $\calX_0:=\calX_{\Sigma_0^c,\beta_0}$, and there exist maps of stacky coloured fans $(\Phi_i:N_0\to N_i,\phi_i:L_0\to L_i)$ for $i=1,2$ such that the following diagram commutes
	\begin{equation*}
		\begin{tikzcd}
			{\calX_0} \\
			{\calX_1} & {\calX_2}
			\arrow["{\pi_{\Phi_2,\phi_2}}", from=1-1, to=2-2]
			\arrow["\pi"', from=2-1, to=2-2]
			\arrow["\rotatebox{90}{$\sim$}", "{\pi_{\Phi_1,\phi_1}}"', from=1-1, to=2-1]
		\end{tikzcd}
	\end{equation*}
	and $\pi_{\Phi_1,\phi_1}$ is an isomorphism. 
\end{theorem}

\begin{remark}\label{rmk:extension of quotienting group}
	Before proving \cref{thm:horospherical morphisms}, we note a general fact about quotient stacks. If a group $K$ is an extension of a quotient $K_2$ by a normal subgroup $K_1$, then $[X/K]$ is identified with $[[X/K_1]/K_2]$.
\end{remark}
	
\begin{proof}[Proof of \cref{thm:horospherical morphisms}]%[Proof of \cref{thm:horospherical morphisms} (cf. {\cite[Theorem 3.4]{satriano2015toric1}})]			
	Set $X_i:=X_{\Sigma_i^c}$, which is a $G/H_i$-horospherical variety for $i=1,2$, and let $K_i:=K_{\beta_i}$. By the definition of horospherical morphism, we have a natural projection map $G/\ol{K_1}\to G/\ol{K_2}$. Thus, we can compose natural projections to get a map $G/H_1\to G/\ol{K_1}\to G/\ol{K_2}$. 
		
	Let $Y:=X_1\times_{\calX_2} X_2$ and consider the open subset $\Omega:=(G/H_1)\times_{G/\ol K_2} (G/H_2)\subseteq Y$. Let $\Omega_0\subseteq \Omega$ be the connected component of the base point $(eH_1,eH_2)$, and let $Y_0\subseteq Y$ be the connected component of $Y$ containing $\Omega_0$. We have the following commutative diagram (in which the hook arrows are open embeddings):
	\begin{equation*}
		\begin{tikzcd}
			{\Omega_0} && \Omega && {G/H_2} \\[-1em]
			& {Y_0} && Y && {X_2} \\[-1em]
			&& {G/H_1} && {G/\ol{K_2}} \\[-1em]
			&&& {X_1} && {\calX_2}
			\arrow[hook, from=3-3, to=4-4]
			\arrow[hook, from=3-5, to=4-6]
			\arrow[from=4-4, to=4-6]
			\arrow[from=3-3, to=3-5]
			\arrow[from=2-4, to=4-4]
			\arrow[from=1-3, to=3-3]
			\arrow[from=1-3, to=1-5]
			\arrow[hook, from=1-5, to=2-6]
			\arrow["K_2\text{-torsor}", from=2-6, to=4-6]
			\arrow[from=1-5, to=3-5]
			\arrow[hook, from=1-3, to=2-4]
			\arrow[from=2-4, to=2-6]
			\arrow[hook, from=1-1, to=1-3]
			\arrow[hook, from=2-2, to=2-4]
			\arrow[hook, from=1-1, to=2-2]
		\end{tikzcd}
	\end{equation*}
	
	Notice that $\Omega$ is a $K_2$-torsor over $G/H_1$. Based on the $K_2$-action and \cite[Lemma 3.3]{satriano2015toric1}, it follows that the connected component $\Omega_0$ is a $K_2^\circ$-torsor over $G/H_1$. In particular, $\Omega_0$ is a $G\times K_2^\circ$-homogeneous space, so it is of the form $(G\times K_2^\circ)/H_0$ for some closed subgroup $H_0$ of the connected reductive linear algebraic group $G\times K_2^\circ$. In fact, this $H_0$ can be taken to be the stabilizer of the base point $(eH_1,eH_2)$. Then $H_0$ is clearly horospherical since the maximal unipotent subgroup $U\times\{1\}\subseteq G\times K_2^\circ$ stabilizes the base point. Furthermore, for $i=1,2$ we have $H_0\subseteq H_i\times K_2^\circ$, and the map $$\Omega_0=(G\times K_2^\circ)/H_0\map G/H_i=(G\times K_2^\circ)/(H_i\times K_2^\circ)$$ is the natural projection, so these are horospherical morphisms of horospherical $G\times K_2^\circ$-homogeneous spaces.
		
	Since $X_1$ is normal and separated, and $Y$ is a $K_2$-torsor over $X_1$, we see that $Y$ is also normal and separated. Thus, $Y_0$ is a normal, separated, and connected scheme. This implies that $Y_0$ is irreducible, so it is a normal variety. Hence, $Y_0$ is a horospherical $G\times K_2^\circ$-variety with open horospherical homogeneous space $\Omega_0$. Let $\Sigma_0^c$ denote the coloured fan associated to $Y_0$, which is on the coloured lattice $N_0$ associated to $\Omega_0$. 
		
	For $i=1,2$, let $\Phi_i:N_0\to N_i$ be the coloured lattice maps associated to the projections $\Omega_0\to G/H_i$. The morphisms $Y_0\to X_i$ for $i=1,2$ are base point preserving and $G\times K_2^\circ$-equivariant (where $K_2^\circ$ acts trivially on the $X_i$), so they are horospherical morphisms. Thus, the maps of coloured lattices $\Phi_i:N_0\to N_i$ are compatible with the coloured fans $\Sigma_0^c$ and $\Sigma_i^c$. Set $\phi_1:=\id_{L_1}$, $\beta_0:=\beta_1\circ\Phi_1$, and $K_0:=K_{\beta_0}$. We have the following morphisms of stacky coloured fans:
	\begin{equation*}
		\begin{tikzcd}
			{\Sigma_1^c} & {\Sigma_0^c} & {\Sigma_2^c} \\[-2em]
			{N_1} & {N_0} & {N_2} \\
			{L_1} & {L_1} & {L_2}
			\arrow[from=1-2, to=1-1]
			\arrow[from=1-2, to=1-3]
			\arrow["{\Phi_1}"', from=2-2, to=2-1]
			\arrow["{\Phi_2}", from=2-2, to=2-3]
			\arrow["{\beta_1}"', from=2-1, to=3-1]
			\arrow["{\beta_2}", from=2-3, to=3-3]
			\arrow["{\phi_1}"', Rightarrow, no head, from=3-2, to=3-1]
			\arrow["{\phi_2}", from=3-2, to=3-3]
			\arrow["{\beta_0}", from=2-2, to=3-2]
		\end{tikzcd}
	\end{equation*}
	
	Note that, since $\Omega_0$ is a $K_2^\circ$-torsor over $G/H_1$, we see that $\ol{\Omega_0}=Y_0$ is a $K_2^\circ$-torsor over $\ol{G/H_1}=X_1$. Thus $Y_0/K_2^\circ\cong X_1$. 
		
	Since each $\Phi_i$ is a map of coloured lattices, they both have finite cokernel. Now \cite[Lemma A.1]{satriano2015toric1} implies that $K_0$ is an extension of $K_1$ by $K_2^\circ$. Therefore, the map of stacky coloured fans $(\Phi_1,\phi_1):(\Sigma_0^c,\beta_0)\to(\Sigma_1^c,\beta_1)$ induces the isomorphism
	\begin{align*}
		\pi_{\Phi_1,\phi_1}:\calX_0=[Y_0/K_0] \isomap [(Y_0/K_2^\circ)/K_1]=\calX_1.
	\end{align*}
	On the other hand, the map of stacky coloured fans $(\Phi_2,\phi_2):(\Sigma_0^c,\beta_0)\to (\Sigma_2^c,\beta_2)$ induces the morphism $\pi_{\Phi_2,\phi_2}:\calX_0\to \calX_2$. 
\end{proof}

\subsection{Isomorphisms}\label{subsec:isomorphisms}

In \cref{thm:horospherical isomorphisms} (stated below) we give a combinatorial characterization for a horospherical morphism (induced by a map of stacky coloured fans) to be an isomorphism; this is a natural generalization of \cite[Theorem B.3]{satriano2015toric1}. The proof is in \cref{sec:proof of isomorphism criteria}. 

\begin{notation}[Preimage coloured cone]\label{not:preimage coloured cone}
	Suppose that $\Phi:N_1\to N_2$ is a map of coloured lattices which is compatible with coloured fans $\Sigma_1^c$ and $\Sigma_2^c$. Given $\sigma_2^c\in\Sigma_2^c$, we let $\Phi^{-1}(\sigma_2^c)$ denote the sub-coloured fan of $\Sigma_1^c$ consisting of all coloured cones in $\Sigma_1^c$ which map into $\sigma_2^c$. Similarly, we let $\Phi^{-1}(\sigma_2)$ denote the underlying fan for $\Phi^{-1}(\sigma_2^c)$, i.e. this is the sub-fan of $\Sigma_1$ consisting of all cones in $\Sigma_1$ which map into $\sigma_2$. 
\end{notation}

\begin{remark}[Preimage of simple opens]\label{rmk:preimage coloured cone yields preimage variety}
	Suppose that $\Phi:N_1\to N_2$ is a map of coloured lattices which is compatible with coloured fans $\Sigma_1^c$ and $\Sigma_2^c$. Then we have a horospherical morphism $X_{\Sigma_1^c}\to X_{\Sigma_2^c}$. Given $\sigma_2^c\in\Sigma_2^c$, we have an open simple horospherical subvariety $X_{\sigma_2^c}\subseteq X_{\Sigma_2^c}$. The sub-coloured fan $\Phi^{-1}(\sigma_2^c)\subseteq \Sigma_1^c$ yields the open horospherical subvariety $X_{\Phi^{-1}(\sigma_2^c)}\subseteq X_{\Sigma_1^c}$ which is naturally the preimage of $X_{\sigma_2^c}$, i.e. $X_{\Sigma_1^c}\times_{X_{\Sigma_2^c}} X_{\sigma_2^c}$. 
\end{remark}

\begin{theorem}[Horospherical isomorphisms]\label{thm:horospherical isomorphisms}
	Let $(\Sigma_i^c,\beta_i:N_i\to L_i)$ be stacky coloured fans for $i=1,2$, and let $(\Phi,\phi):(\Sigma_1^c,\beta_1)\to(\Sigma_2^c,\beta_2)$ be a map of stacky coloured fans. Then the induced horospherical morphism $\pi_{\Phi,\phi}:\calX_{\Sigma_1^c,\beta_1}\to \calX_{\Sigma_2^c,\beta_2}$ is an isomorphism if and only if the following conditions hold:
	\begin{enumerate}[leftmargin=1.5cm, label={(Iso\arabic*)}]
		\item\label{item:Iso1} $\phi$ is an isomorphism of coloured lattices. Thus $\calC_\Phi=\calC_\phi=\varnothing$ and $\calC_1=\calC_2$. 
		\item\label{item:Iso2} For each $\sigma_2^c\in\Sigma_2^c$, $\Phi^{-1}(\sigma_2^c)$ is generated by a single coloured cone and $\calF(\Phi^{-1}(\sigma_2^c))=\calF(\sigma_2^c)$. 
		\item\label{item:Iso3} For each $\sigma_2^c=(\sigma_2,\calF_2)\in\Sigma_2^c$, the map $\Phi$ induces an isomorphism of monoids $\Phi^{-1}(\sigma_2)\cap N_1\to \sigma_2\cap N_2$, where here we can view $\Phi^{-1}(\sigma_2)$ as a cone. 
	\end{enumerate}
\end{theorem}

\subsection{Products of horospherical morphisms}\label{subsec:products of horospherical morphisms}

We finish this section by examining products of horospherical morphisms. If $X_i$ are $G_i/H_i$-horospherical varieties for $i=1,2$, then $X_1\times X_2$ is a $(G_1\times G_2)/(H_1\times H_2)$-horospherical variety. In this setting, the associated coloured lattice to a product of two horospherical homogeneous spaces is the product of the individual associated coloured lattices. Then, like fans and toric varieties, the product of two coloured fans corresponds to the product of the two corresponding horospherical varieties, i.e. $X_{\Sigma_1^c\times\Sigma_2^c}=X_{\Sigma_1^c}\times X_{\Sigma_2^c}$. The following proposition is an immediate generalization of \cite[Proposition 3.6]{satriano2015toric1}. 

\begin{proposition}[Products of horospherical morphisms]\label{prop:products of horospherical morphisms}
	Let $\mathbf{P}$ be a property of morphisms (of algebraic stacks) which is stable under composition and base change. Consider two morphisms of stacky coloured fans:
	\begin{equation*}
		\begin{tikzcd}
			{\Sigma_1^c} & {\Sigma_2^c} \\[-2em]
			{N_1} & {N_2} \\
			{L_1} & {L_2}
			\arrow["\Phi", from=2-1, to=2-2]
			\arrow["\phi", from=3-1, to=3-2]
			\arrow["{\beta_1}"', from=2-1, to=3-1]
			\arrow["{\beta_2}", from=2-2, to=3-2]
			\arrow[from=1-1, to=1-2]
		\end{tikzcd} \qquad\qquad\begin{tikzcd}
		{(\Sigma_1')^c} & {(\Sigma_2')^c} \\[-2em]
		{N_1'} & {N_2'} \\
		{L_1'} & {L_2'}
		\arrow["\Phi'", from=2-1, to=2-2]
		\arrow["\phi'", from=3-1, to=3-2]
		\arrow["{\beta_1'}"', from=2-1, to=3-1]
		\arrow["{\beta_2'}", from=2-2, to=3-2]
		\arrow[from=1-1, to=1-2]
	\end{tikzcd}
	\end{equation*}
	Then the product morphism $\pi_{\Phi\times\Phi',\phi\times\phi'}=\pi_{\Phi,\phi}\times \pi_{\Phi',\phi'}$ induced by the map of stacky coloured fans
	\begin{equation*}
		\begin{tikzcd}
			{\Sigma_1^c\times(\Sigma_1')^c} & {\Sigma_2^c\times(\Sigma_2')^c} \\[-2em]
			{N_1\times N_1'} & {N_2\times N_2'} \\
			{L_1\times L_1'} & {L_2\times L_2'}
			\arrow["\Phi\times\Phi'", from=2-1, to=2-2]
			\arrow["\phi\times\phi'", from=3-1, to=3-2]
			\arrow["{\beta_1\times\beta_1'}"', from=2-1, to=3-1]
			\arrow["{\beta_2\times\beta_2'}", from=2-2, to=3-2]
			\arrow[from=1-1, to=1-2]
		\end{tikzcd}
	\end{equation*}
	has property $\mathbf{P}$ if and only if both $\pi_{\Phi,\phi}$ and $\pi_{\Phi',\phi'}$ have property $\mathbf{P}$. 
	
	\begin{proof}%[Proof (cf. {\cite[Proposition 3.6]{satriano2015toric1}})]		
		The reverse direction holds because $\pi_{\Phi,\phi}\times \pi_{\Phi',\phi'}=(\pi_{\Phi,\phi}\times \id)\circ(\id\times \pi_{\Phi',\phi'})$, and $\pi_{\Phi,\phi}\times \id$ (resp. $\id\times \pi_{\Phi',\phi'}$) is a base change of $\pi_{\Phi,\phi}$ (resp. $\pi_{\Phi',\phi'}$). 
		
		For the forward direction, suppose that $\pi_{\Phi\times\Phi',\phi\times\phi'}=\pi_{\Phi,\phi}\times \pi_{\Phi',\phi'}$ has property $\mathbf{P}$. Note that the base point of $\calX_{\Sigma_2^c,\beta_2}$ induces a horospherical morphism $\calX_{(\Sigma_2')^c,\beta_2'}\to \calX_{\Sigma_2^c,\beta_2}\times\calX_{(\Sigma_2')^c,\beta_2'}$. Base changing $\pi_{\Phi,\phi}\times \pi_{\Phi',\phi'}$ by this morphism yields $\pi_{\Phi',\phi'}$. Thus $\pi_{\Phi',\phi'}$ has property $\mathbf{P}$. A similar argument shows that $\pi_{\Phi,\phi}$ has property $\mathbf{P}$. 
	\end{proof}
\end{proposition}

\section{Coloured fantastacks}\label{sec:coloured fantastacks}

In this section, we introduce a specific class of horospherical stacks called \textit{coloured fantastacks}. These (in some ways) generalize toric fantastacks from \cite[Section 4]{satriano2015toric1}. Coloured fantastacks are created through a very hands-on, ``bottom-up" combinatorial construction; this contrasts the ``top-down" approach in \cref{sec:horospherical stacks and stacky coloured fans}.

\subsection{Construction of coloured fantastacks}\label{subsec:construction of coloured fantastacks}

Let $\ell:=\# \calC$ be the number of universal colours, so we can write $\calC=\{\alpha_1,\ldots,\alpha_\ell\}$. For a chosen $n\geq \ell$, consider the lattice $\wh N:=\Z^n=\Z^\ell\oplus \Z^{n-\ell}$ equipped with the universal colour set $\calC$ where the colour points are $\wh u_{\alpha_i}=e_i\in \wh N$ for $1\leq i\leq\ell$. 

Let $\Sigma^c$ be a coloured fan on $N$ and let $\beta:\wh N\to N$ be a $\Z$-linear map which satisfy the following properties:
\begin{enumerate}[leftmargin=1.5cm, label={(CF\arabic*)}]
	\item\label{item:CF1} $\{u_\alpha\in N : \alpha\in\calC\}\cup |\Sigma^c|$ spans $N_\R$.
	\item\label{item:CF2} $\beta(e_i)\in |\Sigma^c|$ for all $\ell+1\leq i\leq n$.
	\item\label{item:CF3} For each $1\leq i\leq \ell$, we have $\beta(e_i)=u_{\alpha_i}$.
	\item\label{item:CF4} For each non-coloured ray $\rho^c=(\rho,\varnothing)\in\Sigma^c$, there exists $\ell+1\leq i\leq n$ such that $\beta(e_i)\in\rho\setminus\{0\}$. 
\end{enumerate}
We make some observations regarding these conditions: first, \labelcref{item:CF1} is a property of just $\Sigma^c$, not $\beta$; second, $\beta$ necessarily has finite cokernel because $\beta_\R:\wh N_\R\to N_\R$ is surjective by \labelcref{item:CF1}, \labelcref{item:CF3}, and \labelcref{item:CF4}; and third, $n$ must be greater than or equal to the number of $B^-$-invariant prime divisors in $X_{\Sigma^c}$ by \labelcref{item:CF3} and \labelcref{item:CF4}. 

Now for each $\sigma^c=(\sigma,\calF)\in\Sigma^c$, let $\wh\sigma^c:=(\wh\sigma,\calF)$ be the coloured cone on $\wh N$ where
\begin{align*}
	\wh \sigma := \Cone(e_i : 1\leq i\leq \ell \text{ and } \alpha_i\in\calF, \text{ or } \ell+1\leq i\leq n \text{ and } \beta(e_i)\in\sigma).
\end{align*}
Then each $\wh\sigma^c$ is a coloured cone on $\wh N$. Finally, let $\wh\Sigma^c$ be the coloured fan on $\wh N$ generated by all $\wh\sigma^c$, i.e. $\wh\Sigma^c$ is the collection of all $\wh\sigma^c$ together with all of their coloured faces.

Using the theory of reductive groups and \cref{subsec:expanding groups acting on horospherical varieties}, we may assume that $G=G^{ss}\times T_0$ where $G^{ss}$ is semisimple simply connected and $T_0$ is a torus. Let $T_1$ be the torus of rank $n-\rank(N)$, and let $\wh G:=G\times T_1=G^{ss}\times (T_0\times T_1)$. Note that we can identify the root system of $\wh G$ with that of $G$ using an appropriate maximal torus and Borel subgroup of $\wh G$ which map onto $T$ and $B$, respectively, via the projection $\wh G\to G$. Then $\wh N$ is the coloured lattice associated to the horospherical homogeneous space $\wh G/\wh H$, where $\wh H$ is the intersection of the kernels of all characters of the parabolic subgroup $P\times T_1\subseteq \wh G$ (this is summarized in \cite[pg. 550]{gagliardi2014cox}, using the fact that $\Cl(\wh G/\wh H)$ is trivial because $\wh H$ is connected). 

The coloured fan $\wh\Sigma^c$ yields a $\wh G/\wh H$-horospherical variety $X_{\wh\Sigma^c}$, and $\beta:\wh N\to N$ induces a horospherical morphism $X_{\wh\Sigma^c}\to X_{\Sigma^c}$ extending the projection $\wh G/\wh H\to G/H$. Therefore, we have a map of stacky coloured fans
\begin{equation*}
	\begin{tikzcd}
		{\wh\Sigma^c} & {\Sigma^c} \\[-2em]
		{\wh N} & {N} \\
		{N} & {N}
		\arrow["\beta", from=2-1, to=2-2]
		\arrow[Rightarrow, no head, from=3-1, to=3-2]
		\arrow["{\beta}"', from=2-1, to=3-1]
		\arrow[Rightarrow, no head, from=2-2, to=3-2]
		\arrow[from=1-1, to=1-2]
	\end{tikzcd}
\end{equation*}
which induces a horospherical morphism $\calX_{\wh\Sigma^c,\beta}\to X_{\Sigma^c}$.

\begin{definition}[Coloured fantastack]\label{def:coloured fantastack}
	Given $\Sigma^c$ and $\beta$ as above which satisfy \labelcref{item:CF1,item:CF2,item:CF3,item:CF4}, the horospherical $\wh G$-stack $\mathscr F_{\Sigma^c,\beta}:=\calX_{\wh\Sigma^c,\beta}=[X_{\wh\Sigma^c}/K_\beta]$ is called a \define{coloured fantastack}. 
\end{definition}

\begin{remark}[Some fantastacks are coloured fantastacks]\label{rmk:fantastacks are coloured fantastacks}
	We use the definition of \textit{fantastack} from \cite[Definition 4.1]{satriano2015toric1}. Let $(\Sigma,\beta:N\to L)$ be a stacky fan which satisfies the following: $|\Sigma|$ spans $N_\R$, and each ray of $\Sigma$ contains a \textit{nonzero} image point under $\beta$. Then the associated fantastack $\mathscr F_{\Sigma,\beta}$ is a coloured fantastack in the setting where $G=T$ is a torus and $H=\{1\}$ (in this case, $N$ and $L$ have no colour points). 
\end{remark}

\begin{remark}[Cover for coloured fantastack is quasiaffine]\label{rmk:cover for coloured fantastack may not be affine space}
	Let $\mathscr F_{\Sigma^c,\beta}=[X_{\wh\Sigma^c,\beta}/K_\beta]$ be a coloured fantastack. Then the cover $X_{\wh\Sigma^c}$ is a quasiaffine horospherical variety, however it might not be an open subvariety of affine space. This is different from the case of fantastacks: the cover of a fantastack is an open subvariety of affine space. In particular, if $\Sigma^c=\sigma^c$ is generated by a single coloured cone, then the cover of the fantastack $\mathscr F_{\sigma,\beta}$ is affine space, however the cover of the coloured fantastack $\mathscr F_{\sigma^c,\beta}$ might not even be affine, e.g. see \cref{ex:coloured fantastack on A^3 as SL_2xG_m variety}.
\end{remark}

\begin{remark}[Removal of \labelcref{item:CF1}]\label{rmk:removal of (CF1)}
	The condition \labelcref{item:CF1} guarantees that $X_{\Sigma^c}$ has no torus factors in the sense of \cite[Proposition 3.3.9]{cox2011toric}; compare with the proof of \cref{prop:characterization of constant global units}. To construct a coloured fantastack without assuming \labelcref{item:CF1}, one would remove the torus factors from $X_{\Sigma^c}$, i.e. write $X_{\Sigma^c}=X_{(\Sigma')^c}\times T'$ where $T'$ is a torus and $(\Sigma')^c$ satisfies \labelcref{item:CF1}, and then perform the coloured fantastack construction on $X_{(\Sigma')^c}$; compare with ``The General Case" in \cite[Section 5.1]{cox2011toric}. 
\end{remark}

\subsection{Good quotient Cox construction}\label{subsec:good quotient Cox construction}

In \cite{cox1995homogeneous}, Cox developed the \textit{homogeneous coordinate ring} of a toric variety, and used this to express a given toric variety as the good quotient of a quasiaffine toric variety by a subgroup of its torus. In \cite[Section 5]{satriano2015toric1}, it is shown how this construction by Cox fits into the framework of fantastacks for toric stacks. 

There is a similar ``Cox construction" that works for horospherical varieties. We give an outline of this construction and show how it fits into the framework of coloured fantastacks. Let $X$ be a $G/H$-horospherical variety with associated coloured fan $\Sigma^c$ on $N$ for which \labelcref{item:CF1} holds (to remove this condition, see \cref{rmk:removal of (CF1)}). 

We define a map $\beta$ to use the construction in \cref{subsec:construction of coloured fantastacks}. Choose $n$ to be the number of $B^-$-invariant prime divisors of $X$; note that to achieve \labelcref{item:CF3} and \labelcref{item:CF4}, this is the smallest possible value for $n$. Then $n-\ell$ is the number of $G$-invariant prime divisors, so $\Sigma^c$ has $n-\ell$ non-coloured rays, say $\rho_i^c=(\rho_i,\varnothing)$ for $\ell+1\leq i\leq n$. Let $\beta:\wh N=\Z^\ell\oplus \Z^{n-\ell}\to N$ be the map given by
\begin{align}\label{eq:beta map for Cox construction}
	\beta(e_i) = \begin{cases}
		u_{\alpha_i} \qquad 1\leq i\leq \ell
		\\ u_{\rho_i} \qquad \ell+1\leq i\leq n
	\end{cases}
\end{align}
where $u_{\rho_i}$ is the minimal ray generator for $\rho_i$. It is easy to check that \labelcref{item:CF1,item:CF2,item:CF3,item:CF4} hold for $\Sigma^c$ and $\beta$. As in \cref{subsec:construction of coloured fantastacks}, we have a horospherical homogeneous space $\wh G/\wh H$ whose associated coloured lattice is $\wh N$, and the map of coloured lattices associated to the projection $\wh G/\wh H\to G/H$ is $\beta$. 

Now we follow \cite[Section 2]{gagliardi2019luna} to summarize the Cox construction. Let 
\begin{align*}
	\Cox(X) := \bigoplus_{[D]\in\Cl(X)} \mathscr{O}_X(D)(X)
\end{align*}
denote the \define{Cox ring} of $X$. When $X$ is toric, this is the same as the homogeneous coordinate ring constructed in \cite{cox1995homogeneous}. Since $X$ is horospherical, $\Cl(X)$ is a finitely generated abelian group, and since $\scrO_X^*(X)=k^*$ (by \labelcref{item:CF1} and \cref{prop:characterization of constant global units}), $\Cox(X)$ is a finitely generated $k$-algebra. Thus, $\K:=\Spec(k[\Cl(X)])$ is a diagonalizable group and $\Spec(\Cox(X))$ is an affine variety. With respect to $\beta$ from \cref{eq:beta map for Cox construction}, $\Spec(\Cox(X))$ is an affine $\wh G/\wh H$-horospherical variety. Moreover, there exists a $\wh G$-invariant open subset $\wh X\subseteq \Spec(\Cox(X))$ such that there is a $\wh G$-equivariant good quotient $\pi:\wh X\to X$ for the action of $\K$; note that there is a surjection $\wh N\to \Cl(X)$ (see \cite[Theorem 4.2.1]{perrin2018sanya}), which induces an inclusion $\K\subseteq T_{\wh N}$, so $\K$ acts on $\Spec(\Cox(X))$ by $\wh G$-equivariant automorphisms. This map $\pi$ restricts to the projection $\wh G/\wh H\to G/H$, so the associated map of coloured lattices is $\beta:\wh N\to N$. 

The following lemma shows that the coloured fan associated to the $\wh G/\wh H$-horospherical variety $\wh X$ is precisely $\wh\Sigma^c$, as constructed in \cref{subsec:construction of coloured fantastacks} based on our choice of $\beta$. 

\begin{lemma}[Combinatorics of Cox construction]\label{lemma:Cox construction}
	Let $X$ be a $G/H$-horospherical variety with coloured fan $\Sigma^c$ which satisfies \labelcref{item:CF1}, and let $\beta$ be the map constructed in \cref{eq:beta map for Cox construction}. Then we have the following:
	\begin{enumerate}
		\item The coloured fan associated to the $\wh G/\wh H$-horospherical variety $\Spec(\Cox(X))$ is generated by the single coloured cone $(\Cone(e_i:1\leq i\leq n),\calC)$ on $\wh N$. 
		\item The coloured fan associated to the $\wh G/\wh H$-horospherical variety $\wh X$ is $\wh\Sigma^c$ (constructed in \cref{subsec:construction of coloured fantastacks} using $\Sigma^c$ and $\beta$). 
		\item $\K=K_\beta$. 
	\end{enumerate}
	
	\begin{proof}		
		Since $\Spec(\Cox(X))$ is an \textit{affine} horospherical variety, it corresponds to a coloured cone whose colour set is equal to the universal colour set $\calC=\{\alpha_1,\ldots,\alpha_\ell\}$. The fact that the underlying cone is equal to $\Cone(e_i:1\leq i\leq n)$ is shown in \cite[Proposition 2.1]{gagliardi2019luna}. This proves (1).
		
		Let $(\Sigma')^c$ denote the coloured fan associated to $\wh X$, which lives on $\wh N$. To prove (2), we show that $(\Sigma')^c=\wh\Sigma^c$. Since $\wh X$ is an open horospherical subvariety of $\Spec(\Cox(X))$, it follows that $(\Sigma')^c$ is a sub-coloured fan of $(\Cone(e_i:1\leq i\leq n),\calC)$.
		
		It is clear that $(\Sigma')^c$ is a sub-coloured fan of $\wh\Sigma^c$. Indeed, by construction of $\wh\Sigma^c$, it is the maximal sub-coloured fan of $(\Cone(e_i:1\leq i\leq n),\calC)$ with the property that $\beta$ is compatible with it and $\Sigma^c$. Since $\beta$ is compatible with $(\Sigma')^c$ and $\Sigma^c$, we must have $(\Sigma')^c\subseteq\wh\Sigma^c$.
		
		It remains to show that $\wh\Sigma^c\subseteq(\Sigma')^c$. It suffices to show that, for each $\sigma^c\in\Sigma^c$, we have $\wh\sigma^c\in(\Sigma')^c$. Fix $\sigma^c\in\Sigma^c$ and let $X_{\sigma^c}$ be the corresponding simple horospherical variety, which is open in $X$. Let $\calO$ be the unique closed $G$-orbit in $X_{\sigma^c}$. 
		
		We first show that $\pi^{-1}(X_{\sigma^c})$ is simple. By way of contradiction, assume that $\calO_1,\calO_2$ are distinct closed $\wh G$-orbits in $\pi^{-1}(X_{\sigma^c})$. Then each $\calO_i$ is $T_{\wh N}$-invariant, so they are $\K$-invariant. Thus, $\pi(\calO_1)$ and $\pi(\calO_2)$ are closed, disjoint, and $\K$-invariant, so $\pi$ being a good quotient for the action of $\K$ implies that $\pi(\calO_1)$ and $\pi(\calO_2)$ are disjoint. But $\pi(\calO_1)=\pi(\calO_2)=\calO$ by $\wh G$-equivariance, which is a contradiction. Thus, $\pi^{-1}(X_{\sigma^c})$ has a unique closed $G$-orbit, say $\wh\calO$. 
		
		Now let $\alpha\in\calC$ be a colour, and let $D$ and $\wh D$ denote the corresponding colour divisors in $X$ and $\wh X$, respectively. We show that $D\supseteq\calO$ if and only if $\wh D\supseteq\wh \calO$. The reverse direction holds by compatibility. For the forward direction, suppose that $\calO\subseteq D$. Then $\wh\calO\subseteq \pi^{-1}(\calO)\subseteq \pi^{-1}(D)=\K\cdot \wh D$. Note that $T_{\wh N}$ acts by $\wh G$-equivariant automorphisms of $\wh X$, so by connectedness it must preserve the (finitely many) $B^-$-invariant prime divisors. Thus $T_{\wh N}\cdot \wh D=\wh D$, which implies $\K\cdot\wh D=\wh D$, which gives what we want. 
		
		To finish the proof of (2), we need to show that $\wh\sigma^c\in(\Sigma')^c$. Note that the coloured fan for $\pi^{-1}(X_{\sigma^c})$ is a sub-coloured fan of $(\Sigma')^c$ (see \cref{rmk:preimage coloured cone yields preimage variety}). Since $\wh X$ has complementary codimension at least $2$ inside $\Spec(\Cox(X))$ (see \cite[Section 2]{gagliardi2019luna}), and using the result in the previous paragraph, we deduce that all rays of $\wh\sigma^c$ must be contained in the coloured fan for $\pi^{-1}(X_{\sigma^c})$. Since we showed that $\pi^{-1}(X_{\sigma^c})$ is simple, its coloured fan is generated by a single coloured cone, so the previous paragraph and the previous sentence tell us that it must equal $\wh\sigma^c$. Therefore $\wh\sigma^c\in(\Sigma')^c$. 
				
		Finally, note that we have an exact sequence of abelian groups
		\begin{align*}
			0 \map N^\vee \map \wh N^\vee \cong \bigoplus_{D} \Z D \map \Cl(X) \map 0
		\end{align*}
		as described in \cite[Theorem 4.2.1]{perrin2018sanya}, where the direct sum is over all $B^-$-invariant prime divisors $D$ of $X$; exactness on the left is guaranteed by \labelcref{item:CF1}. Thus, (3) follows from an easily-adapted proof of \cite[Lemma 5.1.1(c)]{cox2011toric}. 
	\end{proof}
\end{lemma}

As a result of \cref{lemma:Cox construction}, if we are given $\Sigma^c$ satisfying \labelcref{item:CF1} and $\beta$ from \cref{eq:beta map for Cox construction}, then we have a good quotient $\wh X=X_{\wh\Sigma^c}\to X=X_{\Sigma^c}$ for the action of $\K=K_\beta$. 

\begin{remark}[Geometric quotient / free action]\label{rmk:geometric quotient and free action in Cox construction}
	For $\sigma^c\in\Sigma^c$, consider the \textit{multiset} of points $\{u_\rho:(\rho,\varnothing)\subseteq\sigma^c\}\cup\{u_\alpha:\alpha\in\calF(\sigma^c)\}$ of non-coloured ray generators together with the colour points from $\calF(\sigma^c)$. 
	\begin{enumerate}[leftmargin=0.85cm]
		\item Suppose that $\Sigma^c$ is \define{simplicial}, i.e. for each $\sigma^c\in\Sigma^c$, the multiset above is $\R$-linearly independent in $N_\R$. This supposition is equivalent to $X=X_{\Sigma^c}$ being $\Q$-factorial; see \cite[Corollary 4.2.2]{perrin2018sanya}. In this case (and only this case), the good quotient $\wh X\to X$ described above is a geometric quotient; see \cite[Corollary 1.6.2.7]{arzhantsev2015cox}.
		
		\item Suppose that $\Sigma^c$ is \define{regular}, i.e. for each $\sigma^c$ in $\Sigma^c$, the multiset above can be extended to a $\Z$-basis for $N$. This supposition is equivalent to $X=X_{\Sigma^c}$ being factorial; see \cite[Corollary 4.2.2]{perrin2018sanya}. In this case (and only this case), the good quotient $\wh X\to X$ described above is a geometric quotient and $\K=K_\beta$ acts freely on $\wh X$; see \cite[Corollary 1.6.2.7]{arzhantsev2015cox}. 
	\end{enumerate}
\end{remark}

To finish this subsection, we observe that the Cox construction above fits into the framework of coloured fantastacks. The map $\beta$ from \cref{eq:beta map for Cox construction} satisfies \labelcref{item:CF2,item:CF3,item:CF4}, so we get a coloured fantastack $\mathscr F_{\Sigma^c,\beta}=[X_{\wh\Sigma^c}/K_\beta]=[\wh X/\K]$. When $\Sigma^c$ is regular, the quotient $\wh X/\K$ is free and geometric, so we have $[\wh X/\K]=X$ in this case, i.e. $\mathscr F_{\Sigma^c,\beta}=X_{\Sigma^c}$.

\subsection{Non-toric condition}\label{subsec:non-toric condition}

Before finishing this section with examples of coloured fantastacks, we use the Cox ring from \cref{subsec:good quotient Cox construction} to give a condition for a horospherical stack to be non-toric. These conditions are useful for finding examples of horospherical stacks that cannot be described using the theory of toric stacks. 

For convenience, we say that a horospherical variety $X$ satisfies \labelcref{item:CF1} if its associated coloured fan satisfies \labelcref{item:CF1}; note that this is equivalent to $\mathscr O_X^*(X)=k^*$ (see \cref{prop:characterization of constant global units}). 

\begin{proposition}[Non-toric condition]\label{prop:non-toric condition}
	Let $\calX=[X/K]$ be a $G/\ol{K}$-horospherical stack where $X$ is a $G/H$-horospherical variety satisfying \labelcref{item:CF1}. If the associated flag variety $G/P$ (with $P=N_G(\ol{K})=N_G(H)$) is not a product of projective spaces, then $\calX$ is not a toric stack. 
	
	\begin{proof}
		We first show that $X$ is not a toric variety. Indeed, if it were, then $\Cox(X)$ would be a polynomial ring over $k$. By \cite[Theorem 3.8]{gagliardi2014cox}, $\Cox(X)$ is a polynomial ring over $\Cox(G/P)$, so this implies that $\Cox(G/P)$ is a polynomial ring over $k$. Thus, $G/P$ is a toric variety by \cite[Proposition 6.1]{gagliardi2017generalized} (since $G/P$ is complete), so $G/P$ is a product of projective spaces by \cite[Theorem 1]{thomsen1997affinity}. This contradicts our assumption, so $X$ is indeed not a toric variety. 
		
		Now we use this to show that $\calX$ is not a toric stack. By way of contradiction, suppose otherwise. Then we can write $\calX=[V/K_V]$ for some toric variety $V$ with open torus $T_V$ and some closed subgroup $K_V\subseteq T_V$. Consider the following pullback diagram:
		\begin{equation*}
			\begin{tikzcd}
				Y:=V\times_\calX X && X \\
				V && \calX
				\arrow["{K\text{-torsor}}", two heads, from=1-3, to=2-3]
				\arrow["{K_V\text{-torsor}}"', two heads, from=2-1, to=2-3]
				\arrow[from=1-1, to=2-1]
				\arrow[from=1-1, to=1-3]
			\end{tikzcd}
		\end{equation*}
		Then $Y\to V$ is a $K$-torsor. In particular, this implies that $Y$ is a normal, separated scheme of finite type over $k$. 
		
		Since $K$ is a diagonalizable group, we can write 
		\begin{align*}
			Y = \Spec_V \left(\bigoplus_{m\in \frakX(K)} \mathscr L_m\right) 
		\end{align*}
		for some line bundles $\mathscr L_m$ on $V$ indexed by $m$ in the character group $\frakX(K)$ of $K$; this is a well-known result, e.g. this is referenced in \cite[Proposition 7.1]{satriano2015toric1}. Since $V$ is a toric variety, all of its line bundles admit a $T_V$-equivariant structure. Therefore, the sheaf of $\mathscr O_V$-algebras $\oplus_{m\in \frakX(K)} \mathscr L_m$ admits a $T_V$-equivariant structure, i.e. we have a $T_V$-action on $Y$ such that $Y\to V$ is $T_V$-equivariant. Furthermore, this commutes with the $K$-action on $Y$, so we have a $T_V\times K$-action on $Y$. 
			
		Now consider the following pullback diagram where the hook arrows are open embeddings:
		\begin{equation*}
			\begin{tikzcd}
				{T_V\times_V Y} && Y \\
				{T_V} && V
				\arrow["{K\text{-torsor}}", two heads, from=1-3, to=2-3]
				\arrow[hook, two heads, from=2-1, to=2-3]
				\arrow[from=1-1, to=2-1]
				\arrow[hook, from=1-1, to=1-3]
			\end{tikzcd}
		\end{equation*}
		Then $T_V\times_V Y \to T_V$ is a $K$-torsor. Since $T_V$ is a torus and $K$ is diagonalizable, this implies that $T_V\times_V Y$ is a diagonalizable group containing $K$ and the map $T_V\times_V Y\to T_V$ is the quotient map by $K$. Furthermore, we have an embedding $T_V\times_V Y\hookrightarrow T_V\times K$, which gives a $T_V\times_V Y$-action on $Y$. 
		
		Set $T_0:=(T_V\times_V Y)^\circ$, which is a torus, and let $Y_0$ be the connected component of $Y$ containing $T_0$. Note that $Y_0$ is a normal, separated, connected scheme, so it is irreducible, which means that $Y_0$ is a normal variety. The torus $T_0$ acts on $Y$, and since it is connected it must preserve $Y_0$, i.e. we have a $T_0$-action on $Y_0$. It follows that $Y_0$ is a toric variety with open torus $T_0$. 
		
		Lastly, since $Y\to X$ is a $K_V$-torsor, we see that $Y_0\hookrightarrow Y\to X$ is a $K_V'$-torsor for some $K_V'\subseteq K_V$ (see \cite[Lemma 3.3]{satriano2015toric1}). Since $K_V'$ is diagonalizable, we conclude that $Y_0/K_V'\cong X$ is a toric variety. This yields the contradiction, so we are done. 
	\end{proof}
\end{proposition}

\begin{example}\label{ex:some examples of non-toric stacks}
	Consider $G=SL_n$. The condition $\#\calC\leq 1$ is necessary for $G/P$ to be a (product of) projective space(s); recall that $P=P_I$ and $\calC=S\setminus I$. Therefore, if $[X/K]$ is a $G/\ol{K}$-horospherical stack such that $X$ satisfies \labelcref{item:CF1} and $\#\calC>1$, then $[X/K]$ is not toric. For example, if $n>2$, then we have $\#\calC>1$ when the associated parabolic $P$ is $B_n$. 
\end{example}

\begin{remark}[Toric stacks must use toric presentations]\label{rmk:presentations of toric stacks}
	The proof of \cref{prop:non-toric condition} shows that every presentation of a toric stack must use a toric variety. That is, if $\calX$ is a toric stack and $[X/K]$ is a presentation for $\calX$ with $K$ a diagonalizable group, then $X$ must be a toric variety. 
\end{remark}

\subsection{Examples}\label{subsec:examples - coloured fantastacks}

In the following examples, we draw a stacky coloured fan $\Sigma^c$ on $N$ with marked ``square points" to indicate the map $\beta$: a square point marked with $i$ is the image of $e_{i+\ell}$ under $\beta$, for $1\leq i\leq n-\ell$; recall that $\ell=\#\calC$ is the number of universal colours. Note that we automatically require $\beta(e_i)=u_{\alpha_i}$ for all $1\leq i\leq \ell$ by \labelcref{item:CF3}. To ensure that \labelcref{item:CF2} and \labelcref{item:CF4} are satisfied, we need at least one ``square point" on each non-coloured ray of $\Sigma^c$, and we need all ``square points" to be in $|\Sigma^c|$. 

\begin{example}\label{ex:coloured fantastack on A^3 as SL_2xG_m variety}
	Consider $G=SL_2\times\G_m$ and the horospherical subgroup $H=U_2\times\{1\}$. The lattice $N$ is isomorphic to $\Z^2$, the universal colour set is $\calC=\{\alpha\}$, and we have $u_\alpha=e_1$. Let $\Sigma^c$ be the coloured fan on $N$ given in the diagram below.
	
	\tikzitfig{coloured_fantastack_on_A3_as_SL_2xG_m_variety}
	
	In this case, $\Sigma^c$ is generated by a single coloured cone $(\Cone(e_1,e_2),\varnothing)$. We have $X_{\Sigma^c}=\Bl_0\A^2\times\A^1$, where $SL_2$ acts on $\Bl_0\A^2$ (lifting the matrix action of $SL_2$ on $\A^2$) and $\G_m$ acts by scaling $\A^1$. There are three ``square points" in the diagram above, so
	\begin{align*}
		\beta = \begin{bmatrix}
			1 & 1 & 1 & 0
			\\ 0 & 0 & 1 & 2
		\end{bmatrix} : \wh N=\Z^4\to N=\Z^2.
	\end{align*} 
	Note that \labelcref{item:CF1,item:CF2,item:CF3,item:CF4} are satisfied.
	
	The group $\wh G$ is $SL_2\times \G_m^3$, and we have $\wh G/\wh H\cong (SL_2/U_2)\times\G_m^3$. The coloured fan $\wh\Sigma^c$ is generated by a single coloured cone $(\Cone(e_2,e_3,e_4),\varnothing)$ on $\wh N=\Z^4$. Then $X_{\wh\Sigma^c}=(SL_2/U_2)\times\A^3=(\A^2\setminus\{0\})\times\A^3$. 
	
	By construction, $\ol{K_\beta}$ is a subgroup of $\wh G=SL_2\times\G_m^3$ such that $\wh G/\ol{K_\beta}\cong (SL_2/U_2)\times\G_m$. The subgroup $K_\beta=\{(s,t,(st)^{-1},\zeta_{st}):s,t\in\G_m,~ \zeta_{st}^2=st\}$ of $T_{\wh N}=\G_m^4$ acts by scaling the $\A^2\setminus\{0\}$ coordinates by $s$ and multiplying the $\A^3$ coordinates by $(t,(st)^{-1},\zeta_{st})$. 
	
	All this yields the coloured fantastack $\mathscr F_{\Sigma^c,\beta}=[((\A^2\setminus\{0\})\times\A^3)/K_\beta]$. 
\end{example}

\begin{example}\label{ex:coloured fantastack for non-toric projective SL_3 example}
	Consider $G=SL_3$ and the horospherical subgroup $H=U_3$. The lattice $N$ is isomorphic to $\Z^2$, the universal colour set is $\calC=\{\alpha_1,\alpha_2\}$, and we have $u_{\alpha_1}=e_1$ and $u_{\alpha_2}=e_2$. Let $\Sigma^c$ be the coloured fan on $N$ given in the diagram below.
	
	\tikzitfig{coloured_fantastack_for_non-toric_projective_SL_3_example}
	
	Then $\Sigma^c$ has three maximal coloured cones: $(\Cone(e_1,e_2),\{\alpha_1,\alpha_2\})$, $(\Cone(e_2,-e_1-e_2),\{\alpha_2\})$, and $(\Cone(e_1,-e_1-e_2),\{\alpha_1\})$. Since there is one ``square point" in the diagram above, we get
	\begin{align*}
		\beta = \begin{bmatrix}
			1 & 0 & -1
			\\ 0 & 1 & -1
		\end{bmatrix} : \wh N=\Z^3\to N=\Z^2.
	\end{align*} 
	Note that \labelcref{item:CF1,item:CF2,item:CF3,item:CF4} are satisfied.
	
	The group $\wh G$ is $SL_3\times\G_m$, and we have $\wh G/\wh H\cong (SL_3/U_3)\times\G_m$. The coloured fan $\wh\Sigma^c$ has three maximal coloured cones on $\wh N=\Z^3$: $(\Cone(e_1,e_2),\{\alpha_1,\alpha_2\})$, $(\Cone(e_2,e_3),\{\alpha_2\})$, and $(\Cone(e_1,e_3),\{\alpha_1\})$. The variety $X_{\wh\Sigma^c}$ is $(Y\times\A^1)\setminus\{0\}$ where $Y$ is $\{ux+vy+wz=0\}\subseteq \A^3_{(u,v,w)}\times\A^3_{(x,y,z)}$. 
	
	By construction, $\ol{K_\beta}$ is a subgroup of $\wh G=SL_3\times\G_m$ such that $\wh G/\ol{K_\beta}\cong SL_3/U_3$. The subgroup $K_\beta=\{(t,t,t):t\in\G_m\}$ of $T_{\wh N}=\G_m^3$ acts by scaling $Y\times\A^1$ by $t$. 
	
	All this yields the coloured fantastack $\mathscr F_{\Sigma^c,\beta}=[((Y\times\A^1)\setminus\{0\})/\G_m]$. This is exactly the coloured fantastack which captures the Cox construction (\cref{subsec:good quotient Cox construction}) for $X_{\Sigma^c}$. By \cref{rmk:geometric quotient and free action in Cox construction}, we have 
	\begin{align*}
		\mathscr F_{\Sigma^c,\beta}&=X_{\Sigma^c}=\{ux+vy+wz=0\}\subseteq\P^6_{(u:v:w:x:y:z:z')}.
		\qedhere
	\end{align*}
\end{example}

\begin{example}\label{ex:coloured fantastack for non-toric SL_3 example with beta=[1 0 2 0 1 2]}
	Consider $G=SL_3$ and the horospherical subgroup $H=U_3$. The lattice $N$ is isomorphic to $\Z^2$, the universal colour set is $\calC=\{\alpha_1,\alpha_2\}$, and we have $u_{\alpha_1}=e_1$ and $u_{\alpha_2}=e_2$. Let $\Sigma^c$ be the coloured fan on $N$ given in the diagram below.
	
	\tikzitfig{coloured_fantastack_for_non-toric_SL_3_example_with_beta=102,012}
	
	In this case, $\Sigma^c$ has two maximal coloured cones: $(\Cone(e_1,e_1+e_2),\{\alpha_1\})$ and $(\Cone(e_2,e_1+e_2),\{\alpha_2\})$. There is one ``square point" at $(2,2)$ in the diagram above, so 
	\begin{align*}
		\beta = \begin{bmatrix}
			1 & 0 & 2
			\\ 0 & 1 & 2
		\end{bmatrix} : \wh N=\Z^3 \to N=\Z^2.
	\end{align*}
	Note that \labelcref{item:CF1,item:CF2,item:CF3,item:CF4} are satisfied. 
	
	As in the previous example, let $Y$ be $\{ux+vy+wz=0\}\subseteq \A^3_{(u,v,w)}\times\A^3_{(x,y,z)}$. Then $X_{\Sigma^c}$ is the blow-up of $Y$ at the origin; the exceptional divisor $E$ is the $SL_3$-invariant divisor of $X_{\Sigma^c}$ corresponding to the non-coloured ray $(\Cone(e_1+e_2),\varnothing)$ in $\Sigma^c$.
	
	The group $\wh G$ is $SL_3\times \G_m$, and $\wh G/\wh H\cong (SL_3/U_3)\times\G_m$. The coloured fan $\wh\Sigma^c$ has two maximal coloured cones on $N=\Z^3$: $(\Cone(e_1,e_3),\{\alpha_1\})$ and $(\Cone(e_2,e_3),\{\alpha_2\})$. The variety $X_{\wh\Sigma^c}$ is $(Y\setminus\{0\})\times\A^1$. 
	
	By construction, $\ol{K_\beta}$ is a subgroup of $\wh G=SL_3\times\G_m$ such that $\wh G/\ol{K_\beta}\cong SL_3/U_3$. The subgroup $K_\beta=\{(t,t,\zeta_t^{-1}):t\in\G_m,~ \zeta_t^2=t\}$ of $T_{\wh N}=\G_m^3$ acts by scaling the coordinates of $Y$ by $t$ and scaling the coordinate of $\A^1$ by $\zeta_t^{-1}$. 
	
	All this yields the coloured fantastack $\mathscr F_{\Sigma^c,\beta}=[((Y\setminus\{0\})\times\A^1)/K_\beta]$. This is \textit{almost} the coloured fantastack which captures the Cox construction (\cref{subsec:good quotient Cox construction}) for $X_{\Sigma^c}$, except that $\beta$ sends $e_3\mapsto (2,2)$ rather than $e_3\mapsto (1,1)$. The effect here is that $\mathscr F_{\Sigma^c,\beta}$ adds a $\Z/2\Z$ stabilizer to the exceptional divisor $E$, i.e. $\mathscr F_{\Sigma^c,\beta}$ is the $2$-root stack over $E$. 
\end{example}

\begin{remark}[Root stacks]\label{rmk:root stacks}
	The root stack construction in \cref{ex:coloured fantastack for non-toric SL_3 example with beta=[1 0 2 0 1 2]} works more generally. Let $X_{\Sigma^c}$ be a $G/H$-horospherical variety satisfying \labelcref{item:CF1}, let $r\in\mathbb Z_{\geq 1}$, and let $D$ be a $G$-invariant prime divisor in $X_{\Sigma^c}$. For $D$, we have a corresponding non-coloured ray, say $\rho_i^c=(\rho_i,\varnothing)$ for some $\ell+1\leq i\leq n$ (see \cref{subsec:good quotient Cox construction}). Consider the coloured fantastack $\calF_{\Sigma^c,\beta}$ where $\beta$ is \textit{almost} the map \cref{eq:beta map for Cox construction} except that $\beta(e_i)=ru_{\rho_i}$. Then $\calF_{\Sigma^c,\beta}$ is the $r$-root stack over $D$.  
\end{remark}

\begin{remark}[Toric vs. non-toric examples]\label{rmk:toric vs non-toric examples - coloured fantastacks}
	The horospherical stack in \cref{ex:coloured fantastack on A^3 as SL_2xG_m variety} is toric, so it can also be studied using the theory of toric stacks. However, the other examples in this subsection are not toric; see \cref{prop:non-toric condition}.
\end{remark}

\section{Toroidal horospherical stacks and decolouration}\label{sec:toroidal horospherical stacks and decolouration}

In this section, we study a certain class of \textit{toroidal} horospherical stacks. These are meant to be horospherical stacks which are ``close" to toric stacks (in a certain local sense; see \cref{subsec:local structure and affine reduction}).

\subsection{Toroidal horospherical stacks}\label{subsec:toroidal horospherical stacks}

We first give a very simple combinatorial definition of these stacks and discuss some obvious properties. 

\begin{definition}[Colour set]\label{def:colour set}
	Given a horospherical $G$-stack $\calX=[X/K]$, we define its \define{colour set} as $\calF(\calX):=\calF(X)$. This definition is independent of the choice of horospherical $G$-variety $X$ by \cref{thm:horospherical isomorphisms}. 
\end{definition}

\begin{definition}[Toroidal horospherical stack]\label{def:toroidal horospherical stack}
	We say that a horospherical stack $\calX$ is \define{toroidal} if $\calF(\calX)=\varnothing$. 
\end{definition}

\begin{proposition}[Characterizations of toroidal]\label{prop:characterizations of toroidal}
	Let $\calX=[X/K]$ be a $G/\ol{K}$-horospherical stack, where $X$ is a $G/H$-horospherical variety. Then the following are equivalent:
	\begin{enumerate}
		\item $\calX$ is toroidal.
		\item $X$ is toroidal.
		\item The $G$-equivariant rational map $\calX\dashrightarrow G/P$, induced by the open embedding $G/\ol{K}\hookrightarrow \calX$, is a horospherical morphism defined on all of $\calX$.
	\end{enumerate}
	
	\begin{proof}
		Based on the definition of the colour set of a horospherical stack, it is clear that (1) and (2) are equivalent.
		
		Based on the definition of coloured fan compatibility, the projection $G/H\to G/P$ extends to a horospherical morphism $X\to G/P$ if and only if $\calF(X)=\varnothing$. It follows that $G/\ol{K}\to G/P$ extends to a horospherical morphism $\calX\to G/P$ if and only if $\calF(\calX)=\varnothing$, so (1) and (3) are equivalent.
	\end{proof}
\end{proposition}

\begin{remark}[Alternative definition of toroidal horospherical stack]\label{rmk:comparing definition of toroidal to other paper}
	Condition (3) in \cref{prop:characterizations of toroidal} is the definition of toroidal horospherical stacks given in \cite[Definition 3.4]{terpereau2019horospherical}. Therefore, the results in \cite{terpereau2019horospherical} regarding toroidal horospherical stacks are applicable in our situation (e.g. \cite[Sections 4, 5]{terpereau2019horospherical}). 
\end{remark}

\begin{remark}[Toric stacks are toroidal]\label{rmk:toric stacks are toroidal}
	If $\calX$ is a toric stack, then $\calF(\calX)=\varnothing$ since the universal colour set associated to a torus is empty. Therefore, toric stacks are toroidal. 
\end{remark}

\begin{remark}[Open toroidal substack]\label{rmk:open toroidal substack}
	Consider a horospherical stack $\calX_{\Sigma^c,\beta}=[X_{\Sigma^c}/K_\beta]$. If $X'\subseteq X_{\Sigma^c}$ is the complement of all $G$-orbits of codimension at least $2$, then $\calF(X')=\varnothing$. Indeed, this is because $X'$ corresponds to the sub-coloured fan of $\Sigma^c$ consisting of the trivial coloured cone and all non-coloured rays of $\Sigma^c$. Therefore, this construction gives an open toroidal horospherical substack $[X'/K_\beta]\hookrightarrow \calX_{\Sigma^c,\beta}$. 
\end{remark}

\subsection{Decolouration}\label{subsec:decolouration}

Given a $G/H$-horospherical variety $X$, we can form its \textit{decolouration} $\wt{X}$ as follows. Suppose that $X$ corresponds to the coloured fan $\Sigma^c$ on $N$. Let $\wt\Sigma^c$ denote the coloured fan on $N$ which has the same underlying fan $\Sigma$ as $\Sigma^c$, but whose colour set $\calF(\wt\Sigma^c)$ is empty (so all coloured cones in $\wt\Sigma^c$ have empty colour set). We define $\wt{X}$ as the $G/H$-horospherical variety with coloured fan $\wt\Sigma^c$. 

\begin{definition}[Decolouration of a horospherical stack]\label{def:decolouration of a horospherical stack}
	Given a $G/\ol{K}$-horospherical stack $\calX=[X/K]$, we define its \define{decolouration} $\wt{\calX}$ to be the $G/\ol{K}$-horospherical stack $[\wt{X}/K]$, where $\wt{X}$ is the decolouration of $X$ described above. 
\end{definition}

Note that the \textit{construction} of the decolouration $\wt\calX$ of $\calX$ depends on the presentation of the quotient stack $\calX$. However, we show in \cref{cor:uniqueness of decolouration} that the decolourations of isomorphic stacks are also isomorphic.

The decolouration of $\calX_{\Sigma^c,\beta}$ is $\calX_{\wt\Sigma^c,\beta}$. We have a map of stacky coloured fans
\begin{equation*}
	\begin{tikzcd}
		{\wt\Sigma^c} & {\Sigma^c} \\[-2em]
		{N} & {N} \\
		{L} & {L}
		\arrow[Rightarrow, no head, from=2-1, to=2-2]
		\arrow[Rightarrow, no head, from=3-1, to=3-2]
		\arrow["\beta"', from=2-1, to=3-1]
		\arrow["\beta", from=2-2, to=3-2]
		\arrow[from=1-1, to=1-2]
	\end{tikzcd}
\end{equation*}
which induces a \define{decolouration morphism} $\wt\pi:\calX_{\wt\Sigma^c,\beta}\to \calX_{\Sigma^c,\beta}$. Note that $\wt\pi$ is a surjective proper horospherical morphism. 

If $(\Phi,\phi):(\Sigma_1^c,\beta_1)\to(\Sigma_2^c,\beta_2)$ is a map of stacky coloured fans, then $\Phi$ is compatible with $\Sigma_1^c$ and $\Sigma_2^c$. It is easy to see that $\Phi$ is also compatible with $\wt\Sigma_1^c$ and $\wt\Sigma_2^c$, so $(\Phi,\phi)$ is also a map of stacky coloured fans $(\wt\Sigma_1^c,\beta_1)\to (\wt\Sigma_2^c,\beta_2)$. Therefore, $(\Phi,\phi)$ yields the following commutative diagram of horospherical morphisms:
\begin{equation*}
	\begin{tikzcd}
		{\calX_{\wt\Sigma_1^c,\beta_1}} & {\calX_{\wt\Sigma_2^c,\beta_2}} \\
		{\calX_{\Sigma_1^c,\beta_1}} & {\calX_{\Sigma_2^c,\beta_2}}
		\arrow["{\wt\pi_1}"', from=1-1, to=2-1]
		\arrow["{\pi_{\Phi,\phi}}"', from=2-1, to=2-2]
		\arrow["{\wt\pi_2}", from=1-2, to=2-2]
		\arrow[from=1-1, to=1-2]
	\end{tikzcd}
\end{equation*}

\begin{remark}[Decolouration of toroidal stack]\label{rmk:decolouration of toroidal stack}
	If $\calX$ is a toroidal horospherical stack, then the decolouration $\wt\calX$ is just $\calX$, i.e. decolouration does nothing to toroidal stacks. In particular, decolouration does nothing to toric stacks. 
\end{remark}

\begin{remark}[Decolourations are toroidal]\label{rmk:decolourations are toroidal}
	For any horospherical stack $\calX$, its decolouration $\wt\calX$ is a toroidal horospherical stack. 
\end{remark}

\begin{proposition}[Universal property of decolouration]\label{prop:universal property of decolouration}
	Let $\calX_{\Sigma^c,\beta}$ be a horospherical $G$-stack, and let $\calY$ be a toroidal horospherical $G$-stack. If we have a horospherical morphism $\calY\to \calX_{\Sigma^c,\beta}$, then there exists a unique horospherical morphism $\calY\to \calX_{\wt\Sigma^c,\beta}$ which makes the following diagram commute:
	\begin{equation*}
		\begin{tikzcd}
			\calX_{\wt\Sigma^c,\beta} && \calY \\
			& \calX_{\Sigma^c,\beta}
			\arrow["\wt\pi"', from=1-1, to=2-2]
			\arrow[from=1-3, to=2-2]
			\arrow["{\exists!}"', from=1-3, to=1-1]
		\end{tikzcd}
	\end{equation*}
	
	\begin{proof}		
		We can write $\calY=\calX_{(\Sigma')^c,\beta'}$ for some stacky coloured fan $((\Sigma')^c,\beta')$. By \cref{thm:horospherical morphisms}, we can create the following commutative diagram for some horospherical stack $\calY_0=[Y_0/K_0]$ and some maps of stacky coloured fans $(\Phi_0,\phi_0)$ and $(\Phi,\phi)$, where $\pi_{\Phi_0,\phi_0}$ is an isomorphism:
		\begin{equation*}
			\begin{tikzcd}
				{\calY_0} \\
				{\calY} & {\calX_{\Sigma^c,\beta}}
				\arrow["{\pi_{\Phi,\phi}}", from=1-1, to=2-2]
				\arrow[from=2-1, to=2-2]
				\arrow["\rotatebox{90}{$\sim$}", "{\pi_{\Phi_0,\phi_0}}"', from=1-1, to=2-1]
			\end{tikzcd}
		\end{equation*}
		Here, $Y_0$ is a horospherical variety with coloured fan $\Sigma_0^c$ on a coloured lattice $N_0$. Note that, since $\calY$ is toroidal, $(\Sigma')^c$ must have empty colour set, so by compatibility of $\Phi_0$ we see that $\Sigma_0^c$ must also have empty colour set. 
		
		Now we have the following commutative diagram of horospherical morphisms:
		\begin{equation*}
			\begin{tikzcd}
				{\calX_{\wt\Sigma^c,\beta}} && \calY && {\calY_0} \\
				& {\calX_{\Sigma^c,\beta}}
				\arrow["\wt\pi"', from=1-1, to=2-2]
				\arrow["{\pi_{\Phi,\phi}}", from=1-5, to=2-2]
				\arrow["{\pi_{\Phi_0,\phi_0}}"', from=1-5, to=1-3]
				\arrow[from=1-3, to=2-2]
				\arrow[curve={height=30pt}, dashed, from=1-5, to=1-1]
			\end{tikzcd}
		\end{equation*}
		Since $\Phi$ is compatible with $\Sigma_0^c$ and $\Sigma^c$, and since $\Sigma_0^c$ has empty colour set, it follows that $\Phi$ is compatible with $\Sigma_0^c$ and $\wt\Sigma^c$. Therefore, we get a dotted arrow in the above commutative diagram, i.e. a horospherical morphism $\calY_0\to\calX_{\wt\Sigma^c,\beta}$ induced by $(\Phi,\phi)$. The composition of $\pi_{\Phi_0,\phi_0}^{-1}$ followed by this map gives a horospherical morphism $\calY\to\calX_{\wt\Sigma^c,\beta}$ making the diagram commute, which proves the existence part of the proposition.
		
		Now we prove uniqueness, so suppose that there are two morphisms $\pi_1,\pi_2:\calY\to \calX_{\wt\Sigma^c,\beta}$ which make the original diagram commute. We can apply \cref{thm:horospherical morphisms} three times to get a common tower $\calY_0$ in the following diagram:
		\begin{equation*}
			\begin{tikzcd}
				{\calX_{\wt\Sigma^c,\beta}} && \calY && {\calY_0} \\
				& {\calX_{\Sigma^c,\beta}}
				\arrow[from=1-3, to=2-2]
				\arrow["\wt\pi"', from=1-1, to=2-2]
				\arrow["{\pi_2}"{description}, shift left=2, from=1-3, to=1-1]
				\arrow["{\pi_1}"{description}, shift right=1, from=1-3, to=1-1]
				\arrow["{\pi_{\Phi_1,\phi_1}}"', shift right=2, curve={height=30pt}, from=1-5, to=1-1]
				\arrow["{\pi_{\Phi_2,\phi_2}}", curve={height=30pt}, from=1-5, to=1-1]
				\arrow["{\sim}"', "{\pi_{\Phi_0,\phi_0}}", from=1-5, to=1-3]
			\end{tikzcd}
		\end{equation*}
		where $\pi_{\Phi_0,\phi_0}$ is an isomorphism and $\pi_i\circ\pi_{\Phi_0,\phi_0}=\pi_{\Phi_i,\phi_i}$ for $i=1,2$. 
		
		Since $\wt\pi\circ\pi_{\Phi_1,\phi_1}=\wt\pi\circ\pi_{\Phi_2,\phi_2}$, we get a commuting condition on the maps of stacky fans, which simply implies $(\Phi_1,\phi_1)=(\Phi_2,\phi_2)$ based on the decolouration map. Hence $\pi_{\Phi_1,\phi_1}=\pi_{\Phi_2,\phi_2}$, and it follows that $\pi_1=\pi_2$. 
	\end{proof}
\end{proposition}

\begin{corollary}[Uniqueness of decolouration]\label{cor:uniqueness of decolouration}
	Let $\calX_{\Sigma_1^c,\beta_1}$ and $\calX_{\Sigma_2^c,\beta_2}$ be horospherical $G$-stacks. If we have a horospherical isomorphism $\calX_{\Sigma_1^c,\beta_1}\to \calX_{\Sigma_2^c,\beta_2}$, then the induced horospherical morphism $\calX_{\wt\Sigma_1^c,\beta_1}\to \calX_{\wt\Sigma_2^c,\beta_2}$ is also an isomorphism. 
\end{corollary}

\section{Good moduli spaces}\label{sec:good moduli spaces}

The notion of a good moduli space was first introduced by Alper \cite{alper2013good}. Given a stack $\calX$, its good moduli space (when it exists) is a scheme (or more generally, an algebraic space) $X$ which is meant to be a ``scheme approximation" of $\calX$; there is a morphism $\calX\to X$ which transports much the algebro-geometric data of $\calX$ to $X$, such as the structure sheaf, but it omits the ``stacky structure". Affine locally, the good moduli space of $[\Spec(A)/K]$ is the GIT quotient $\Spec(A^K)$. 

In this section, we give a combinatorial description of when a morphism of horospherical stacks is a good moduli space morphism, extending the criteria for toric good moduli space morphisms from \cite[Section 6]{satriano2015toric1}. With this description, we are able to combinatorially describe the good moduli space of a horospherical stack, when it exists (in \cref{cor:good moduli space of a horospherical stack}).

\subsection{Good moduli space morphisms}\label{subsec:good moduli space morphisms}

We first state the definition of a good moduli space (``good moduli space" is sometimes abbreviated ``GMS"), as developed by Alper \cite[Definition 4.1]{alper2013good}; see \cite[Section 4]{alper2013good} for general properties of good moduli spaces. 

\begin{definition}[Good moduli space]\label{def:good moduli space}
	A quasi-compact and quasi-separated morphism of algebraic stacks $\pi:\calX\to\calY$ is a \define{good moduli space morphism} if the following conditions hold:
	\begin{enumerate}
		\item ($\pi$ is \define{Stein}) The morphism $\mathscr O_\calY\to \pi_*\mathscr O_\calX$ is an isomorphism.
		\item ($\pi$ is \define{cohomologically affine}) The pushforward functor $\pi_*:\opn{QCoh}(\mathscr O_\calX)\to \opn{QCoh}(\mathscr O_\calY)$ on quasicoherent sheaves is exact. 
	\end{enumerate}
	Given an algebraic stack $\calX$ and an algebraic space $Y$, we say that $Y$ is the \define{good moduli space} of $\calX$ if there exists a good moduli space morphism $\calX\to Y$.\footnote{Note that good moduli spaces are unique (when they exist), so we can say ``\textit{the} good moduli space" of $\calX$; see \cite[Section 6]{alper2013good}.}
\end{definition}

\begin{remark}[Cohomologically affine stack]\label{rmk:cohomologically affine stack}
	An algebraic stack $\calX$ is \define{cohomologically affine} if the structure morphism $\calX\to\Spec(k)$ is cohomologically affine. A horospherical stack $[X/K]$ is cohomologically affine if and only if $X$ is an affine variety. 
	
	Indeed, we can see this by an argument made in \cite[Remark B.6]{satriano2015toric1}. If $X$ is affine, then $[X/K]\to\Spec(k)$ is cohomologically affine by \cite[Proposition 3.13]{alper2013good} since $K$ affine. Conversely, suppose that $[X/K]\to\Spec(k)$ is cohomologically affine. Since $X\to [X/K]$ is a $K$-torsor, we see that $X$ is affine over $[X/K]$, so $X\to\Spec(k)$ is cohomologically affine. Thus, $X$ is affine by Serre's criteria (see \cite[Chapter II, Corollary 5.2.2]{grothendieck1971elements}). 
\end{remark}

\begin{definition}[Unstable cone; cf. {\cite[Defintion 6.2]{satriano2015toric1}}]\label{def:unstable cone}
	Let $(\Sigma^c,\beta:N\to L)$ be a stacky coloured fan. We say that a coloured cone $\tau^c\in\Sigma^c$ (or simply the underlying cone $\tau$) is \define{unstable} if any of the following equivalent conditions hold:
	\begin{enumerate}
		\item Every linear functional $L\to\Z$ which is nonnegative on $\beta_\R(\tau)$ vanishes on $\beta_\R(\tau)$. Said another way: for each $l\in L^\vee$, if $\beta^\vee(l)\in\tau^\vee$, then $\beta^\vee(l)\in\tau^\perp$. 
		\item The relative interior of $\beta_\R(\tau)$ contains $0$.
		\item $\tau\cap\ker(\beta)$ is not contained in any proper face of $\tau$. 
	\end{enumerate}
\end{definition}

\begin{notation}\label{not:cone group}
	Given a polyhedral cone (possibly non-strongly convex) $\tau$ on $N$, we let $\tau^{\gp}$ denote the group $\Z(N\cap\tau)$. This is a saturated sublattice of $N$. 
\end{notation}

The following is the main theorem of this section, which we prove in \cref{sec:proof of good moduli space criteria}; this is a natural generalization of \cite[Theorem 6.3]{satriano2015toric1}. 

\begin{theorem}[GMS map criteria]\label{thm:GMS map criteria}
	Let $(\Sigma_i^c,\beta_i)$ be stacky coloured fans for $i=1,2$, and let $(\Phi,\phi):(\Sigma_1^c,\beta_1)\to (\Sigma_2^c,\beta_2)$ be a map of stacky coloured fans. Then the induced morphism $\pi_{\Phi,\phi}:\calX_{\Sigma_1^c,\beta_1}\to\calX_{\Sigma_2^c,\beta_2}$ is a good moduli space morphism if and only if the following conditions hold:
	\begin{enumerate}[leftmargin=2cm, label={(GMS\arabic*)}]
		\item\label{item:GMS1} For each $\sigma_2^c=(\sigma_2,\calF_2)\in\Sigma_2^c$, the preimage $\Phi^{-1}(\sigma_2^c)$ is generated by a single coloured cone $\sigma_1^c=(\sigma_1,\calF_1)\in\Sigma_1^c$ such that $\Phi_\R(\sigma_1)=\sigma_2$ and $\calF_1=\calF_2\cup\calC_\Phi$. In particular, the preimage of the trivial coloured cone is generated by a single coloured cone $\tau^c\in\Sigma_1^c$.
		\item\label{item:GMS2} $\tau^c$ is unstable.
		\item\label{item:GMS3} $\phi$ is surjective.
		\item\label{item:GMS4} $\ker(\phi)=(\beta_1)_\R(\tau)^{\gp}$.\footnote{$(\beta_1)_\R(\tau)$ is a polyhedral cone (possibly non-strongly convex) in $(L_1)_\R$. Also, $(\beta_1)_\R(\tau)^{\gp}$ is equal to the saturation of $\beta_1(\tau^{\gp})$ inside $L_1$.}
	\end{enumerate}
\end{theorem}

\begin{remark}\label{rmk:gms stable under base change and composition}
	By \cite[Lemma 6.9]{satriano2015toric1} and \cite[Proposition 4.7(i)]{alper2013good}, it follows that good moduli space morphisms are stable under composition and base change, so we can use \cref{prop:products of horospherical morphisms} for products of good moduli space morphisms of horospherical stacks.
\end{remark}

\subsection{Good moduli space of a horospherical stack}\label{subsec:good moduli space of a horospherical stack}

Now we use \cref{thm:GMS map criteria} to give a combinatorial characterization of when a horospherical stack has a horospherical variety as a good moduli space, and in this case we use the following notation to describe its coloured fan; this generalizes \cite[Corollary 6.5]{satriano2015toric1}. . 

\begin{remark}[Quotient by coloured sublattice]\label{rmk:quotient by coloured sublattice}
	Suppose that $N$ is the coloured lattice for the horospherical homogeneous space $G/H$, and let $\calC_N$ be the universal colour set. Let $N'\subseteq N$ be a saturated sublattice whose universal colour set is $\calC_{N'}\subseteq\calC_N$; so we must have $u_\alpha\in N'$ for all $\alpha\in\calC_{N'}$. Then $N/N'$ is a lattice. Let $\Phi:N\to N/N'$ be the natural quotient map. As in \cite[Section 4]{knop1991luna}, $N/N'$ is the coloured lattice associated to some horospherical homogeneous space $G/H'$, we have $H\subseteq H'$ and $\Phi$ is the map of coloured lattices associated to the projection $G/H\to G/H'$, and the set of dominantly mapped colours is $\calC_\Phi=\calC_{N'}$; so the universal colour set for $N/N'$ is $\calC_{N/N'}=\calC_N\setminus\calC_{N'}$. 
	
	In particular, if $N'=\tau^{\gp}$ for some coloured cone $\tau^c=(\tau,\calF)$ on $N$, then we implicitly take $\calC_{N'}$ to be $\calF$. Thus, $N/\tau^{\gp}$ is a coloured lattice with universal colour set $\calC_N\setminus\calF$. If $(\tau^c,\beta:N\to L)$ is a stacky coloured fan, then we use the same convention for $L/\beta_\R(\tau)^{\gp}$, i.e. this has universal colour set $\calC_N\setminus\calF$ (note that $\calC_N=\calC_L$). 
\end{remark}

\begin{notation}[GMS coloured fan]\label{not:good moduli space coloured fan}
	Let $(\Sigma^c,\beta:N\to L)$ be a stacky coloured fan. Suppose that, among unstable coloured cones of $\Sigma^c$, there is a unique maximal one $\tau^c$. Let $L_{\gms}:=L/\beta_\R(\tau)^{\gp}$, which is a coloured lattice with universal colour set $\calC\setminus\calF(\tau^c)$ (see \cref{rmk:quotient by coloured sublattice}). Let $\phi:L\to L_{\gms}$ be the quotient map and set $\Phi:=\phi\circ\beta$. Let $\Sigma_{\gms}^c$ be the set of coloured cones $\sigma_{\gms}^c=(\sigma_{\gms},\calF_{\gms})$ on $L_{\gms}$ which satisfy the following conditions:
	\begin{enumerate}
		\item $\Phi^{-1}(\sigma_{\gms}^c)$ is generated by a single coloured cone.
		\item $\sigma_{\gms}=\Phi_\R(\Phi^{-1}(\sigma_{\gms}))$.
		\item $\calF(\Phi^{-1}(\sigma_{\gms}^c))=\calF_{\gms}\cup\calC_\Phi$, or equivalently, $\calF_{\gms}=\calF(\Phi^{-1}(\sigma_{\gms}^c))\setminus\calC_\Phi$.
	\end{enumerate}
	Then $\Sigma_{\gms}^c$ is a coloured fan on $L_{\gms}$.\footnote{See the footnote of \cite[Notation 4]{satriano2015toric1} for why $\Sigma_{\gms}$ is a fan. The fact that $\Sigma_{\gms}^c$ is a \textit{coloured} fan follows immediately from this.}
\end{notation}

\begin{corollary}[GMS of a horospherical stack]\label{cor:good moduli space of a horospherical stack}
	Let $(\Sigma^c,\beta:N\to L)$ be a stacky coloured fan. Then $\calX_{\Sigma^c,\beta}$ has a variety as a good moduli space if and only if the following conditions are satisfied:
	\begin{enumerate}
		\item Among unstable coloured cones of $\Sigma^c$, there is a unique maximal one, denoted $\tau^c$.
		\item Using \cref{not:good moduli space coloured fan}, $\Phi$ is compatible with $\Sigma^c$ and $\Sigma_{\gms}^c$. 
	\end{enumerate}
	Moreover, if these two conditions are satisfied, then the good moduli space of $\calX_{\Sigma^c,\beta}$ is $X_{\Sigma_{\gms}^c}$, and the good moduli space morphism is the horospherical morphism $\calX_{\Sigma^c,\beta}\to X_{\Sigma_{\gms}^c}$ induced by $(\Phi,\phi)$ from \cref{not:good moduli space coloured fan}.
	
	\begin{proof}%[Proof (cf. {\cite[Corollary 6.5]{satriano2015toric1}})]		
		If conditions (1) and (2) are satisfied, then $X_{\Sigma_{\gms}^c}=\calX_{\Sigma_{\gms}^c,\id_{L_{\gms}}}$ is a horospherical variety and, by \cref{thm:GMS map criteria}, $(\Phi,\phi)$ induces a good moduli space morphism.
		
		Now we prove the forward direction, so suppose that $\calX_{\Sigma^c,\beta}$ has a good moduli space which is a variety, say $X'$. So we have a good moduli space morphism $\pi:\calX_{\Sigma^c,\beta}\to X'$. By \cite[Theorem 4.16(viii)]{alper2013good}, $X'$ is normal. By \cite[Theorem 6.6]{alper2013good}, this good moduli space morphism $\pi$ induces a $G$-action on $X'$ making $\pi$ a $G$-equivariant morphism. We choose the base point of $X'$ to be the image of the base point of $\calX_{\Sigma^c,\beta}$. By $G$-equivariance, it follows that the stabilizer of the base point of $X'$ is horospherical. Since $\pi$ is surjective (see \cite[Theorem 4.16(i)]{alper2013good}) and $G$-equivariant, the image of the open $G$-orbit in $\calX_{\Sigma^c,\beta}$ is a dense $G$-orbit in $X'$. Since orbits are open in their closure, we deduce that $X'$ has an open $G$-orbit (the orbit of the base point). Therefore, $X'$ is a horospherical $G$-variety, so we can write $X'=X_{(\Sigma')^c}$ for some coloured fan $(\Sigma')^c$ on a coloured lattice $N'$. 
		
		By applying \cref{thm:horospherical morphisms} to $\pi$, there exist maps of stacky coloured fans
		\begin{equation*}
			\begin{tikzcd}
				{\Sigma_0^c} & {\Sigma^c} \\[-2em]
				{N_0} & {N} \\
				{L_0} & {L}
				\arrow["\Phi_0", from=2-1, to=2-2]
				\arrow["\phi_0", from=3-1, to=3-2]
				\arrow["{\beta_0}"', from=2-1, to=3-1]
				\arrow["{\beta}", from=2-2, to=3-2]
				\arrow[from=1-1, to=1-2]
			\end{tikzcd}\qquad\qquad\begin{tikzcd}
			{\Sigma_0^c} & {(\Sigma')^c} \\[-2em]
			{N_0} & {N'} \\
			{L_0} & {N'}
			\arrow["\Phi'", from=2-1, to=2-2]
			\arrow["\phi'", from=3-1, to=3-2]
			\arrow["{\beta_0}"', from=2-1, to=3-1]
			\arrow[Rightarrow, no head, from=2-2, to=3-2]
			\arrow[from=1-1, to=1-2]
		\end{tikzcd}
		\end{equation*}
		such that $\pi_{\Phi_0,\phi_0}$ is an isomorphism and $\pi=\pi_{\Phi',\phi'}\circ \pi_{\Phi_0,\phi_0}^{-1}$. By \cref{thm:horospherical isomorphisms}, we can see that $(\Sigma^c,\beta)$ satisfies conditions (1) and (2) if and only if $(\Sigma_0^c,\beta_0)$ does. Therefore, we may assume that $\pi$ is induced by a map of stacky coloured fans $(\Psi,\psi):(\Sigma^c,\beta)\to ((\Sigma')^c,\id_{N'})$. 
		
		By \cref{thm:GMS map criteria}, properties \labelcref{item:GMS1} and \labelcref{item:GMS2} imply that $\tau^c:=\Psi^{-1}(0^c)$ is the unique maximal unstable coloured cone in $\Sigma^c$, which establishes (1). By \labelcref{item:GMS3} and \labelcref{item:GMS4}, we must have $\psi=\phi$ and $\Psi=\psi\circ\beta=\phi\circ\beta=\Phi$. Combining this with \labelcref{item:GMS1} implies that (2) holds. 
	\end{proof}
\end{corollary}

\subsection{Examples}\label{subsec:examples - good moduli spaces} 

Our first example highlights a fundamental property of coloured fantastacks. 

\begin{example}[GMS of coloured fantastack]\label{rmk:gms of coloured fantastack}
	Let $\scrF_{\Sigma^c,\beta}$ be a coloured fantastack. Using \cref{thm:GMS map criteria}, one can check that the horospherical morphism $\scrF_{\Sigma^c,\beta}\to X_{\Sigma^c}$ is a good moduli space morphism. Therefore, $X_{\Sigma^c}$ is the good moduli space of $\scrF_{\Sigma^c,\beta}$. 
\end{example}

In the following examples, we use \cref{cor:good moduli space of a horospherical stack} with \cref{not:good moduli space coloured fan} to construct the good moduli space $X_{\Sigma_{\gms}^c}$ of a given $\calX_{\Sigma^c,\beta}$.

\begin{example}\label{ex:good moduli space does not exist for P^2 modulo G_m}
	Consider $G=SL_2$ and the horospherical subgroup $H=U_2$. The lattice $N$ is isomorphic to $\Z$, the universal colour set is $\calC=\{\alpha\}$, and we have $u_\alpha=e_1$. Let $(\Sigma^c,\beta)$ be the stacky coloured fan given in the diagram below.
	
	\tikzitfig{good_moduli_space_does_not_exist_for_P2_modulo_G_m} 
	
	In this case, $\Sigma^c$ has two maximal coloured cones: $(\Cone(e_1),\{\alpha\})$ and $(\Cone(-e_1),\varnothing)$. The variety $X_{\Sigma^c}$ is $\P^2$, where $SL_2$ acts by matrix multiplication on a standard open $\A^2\hookrightarrow \P^2$. The subgroup $K_\beta=\G_m$ of $T_N=\G_m$ is acting by scaling the coordinates of this open $\A^2$. 
	
	Observe that both maximal coloured cones in $\Sigma^c$ are unstable, so there is not a unique maximal unstable coloured cone. Therefore, \cref{cor:good moduli space of a horospherical stack} tells us that $\calX_{\Sigma^c,\beta}$ does not have a variety good moduli space. 
\end{example}

\begin{example}\label{ex:good moduli space of non-toric SL_3 example with beta=[1 0]}
	Consider $G=SL_3$ and the horospherical subgroup $H=U_3$. The lattice $N$ is isomorphic to $\Z^2$, the universal colour set is $\calC=\{\alpha_1,\alpha_2\}$, and we have $u_{\alpha_1}=e_1$ and $u_{\alpha_2}=e_2$. Let $(\Sigma^c,\beta)$ be the stacky coloured fan given in the diagram below.
	
	\tikzitfig{good_moduli_space_of_non-toric_SL_3_example_with_beta=10}
	
	In this case, $\Sigma^c$ is generated by a single coloured cone $(\Cone(e_1,e_2),\{\alpha_2\})$. Using \cref{subsec:local structure and affine reduction}, we can describe $X_{\Sigma^c}$ as $SL_3\times^{P_{\{\alpha_2\}}} \A^3$. An element $(1,t)$ of the subgroup $K_\beta=\{(1,t):t\in\G_m\}$ of $T_N=\G_m^2$ acts on $SL_3\times^{P_{\{\alpha_2\}}}\A^3$ by scaling the last two coordinates of this $\A^3$ by $t$. 
	
	Note that $\tau^c=(\Cone(e_2),\{\alpha_2\})$ is the unique maximal unstable coloured cone in $\Sigma^c$. Since $\beta_\R(\tau)^{\gp}=0$ and $\calF(\tau^c)=\{\alpha_2\}$, we have $L_{\gms}=L=\Z$ with universal colour set $\{\alpha_1\}$. Then $\Sigma_{\gms}^c$ is generated by a single coloured cone $(\Cone(e_1),\varnothing)$ on $L_{\gms}$; see the above diagram.
	
	The conditions of \cref{cor:good moduli space of a horospherical stack} are satisfied, so $X_{\Sigma_{\gms}^c}$ is the good moduli space of $\calX_{\Sigma^c,\beta}=[X_{\Sigma^c}/K_\beta]$. Note that $X_{\Sigma_{\gms}^c}$ is a $SL_3/\ol{K_\beta}$-horospherical variety where 
	\begin{align*}
		\ol{K_\beta}=\left\{\begin{bmatrix}
			1 & * & *
			\\ 0 & * & *
			\\ 0 & * & *
		\end{bmatrix}\right\} \subseteq SL_3
	\end{align*}
	(using \cref{rmk:basis for character lattice with G=SL_n}). It is easy to see that $\ol{K_\beta}$ is the stabilizer of the base point $(1,0,0)$ in $\A^3$, where $SL_3$ acts by usual matrix multiplication. It follows that $X_{\Sigma_{\gms}^c}$ is the blow-up of $\A^3$ at the origin. 
\end{example}

\begin{example}\label{ex:good moduli space of non-toric SL_3 example with beta=[1 -1]}
	Consider $G=SL_3$ and the horospherical subgroup $H=U_3$. The lattice $N$ is isomorphic to $\Z^2$, the universal colour set is $\calC=\{\alpha_1,\alpha_2\}$, and we have $u_{\alpha_1}=e_1$ and $u_{\alpha_2}=e_2$. Let $(\Sigma^c,\beta)$ be the stacky coloured fan given in the diagram below.
	
	\tikzitfig{good_moduli_space_of_non-toric_SL_3_example_with_beta=1-1}
	
	In this case, $\Sigma^c$ has two maximal coloured cones: $(\Cone(e_1,e_1+e_2),\{\alpha_1\})$ and $(\Cone(e_2,e_1+e_2),\varnothing)$. Using \cref{subsec:local structure and affine reduction}, we can describe $X_{\Sigma^c}$ as $SL_3\times^{P_{\{\alpha_1\}}} \Bl_0\A^3$. The subgroup $K_\beta=\{(t,t):t\in\G_m\}=\G_m$ of $T_N=\G_m^2$ acts by lifting the diagonal action of $\G_m$ on $\A^3$ to $\Bl_0\A^3$. 
	
	Note that $\tau^c=(\Cone(e_1+e_2),\varnothing)$ is the unique maximal unstable coloured cone in $\Sigma^c$. Since $\beta_\R(\tau)^{\gp}=0$ and $\calF(\tau^c)=\varnothing$, we have $L_{\gms}=L=\Z$ with universal colour set $\{\alpha_1,\alpha_2\}$. Then $\Sigma_{\gms}^c$ has two maximal coloured cones: $(\Cone(e_1),\{\alpha_1\})$ and $(\Cone(-e_1),\varnothing)$ on $L_{\gms}$; see the above diagram. 
	
	The conditions of \cref{cor:good moduli space of a horospherical stack} are satisfied, so $X_{\Sigma_{\gms}^c}$ is the good moduli space of $\calX_{\Sigma^c,\beta}=[X_{\Sigma^c}/\G_m]$. Note that $X_{\Sigma_{\gms}^c}$ is a $SL_3/\ol{K_\beta}$-horospherical variety where
	\begin{align*}
		\ol{K_\beta} = \left\{\begin{bmatrix}
			* & * & *
			\\ 0 & 1 & *
			\\ 0 & 0 & *
		\end{bmatrix}\right\} \subseteq SL_3
	\end{align*}
	(using \cref{rmk:basis for character lattice with G=SL_n}). Using the Cox construction for $X_{\Sigma_{\gms}^c}$ (see \cref{subsec:good quotient Cox construction} and \cref{rmk:geometric quotient and free action in Cox construction}) and \cref{subsec:local structure and affine reduction}, we can describe this good moduli space as $SL_3\times^{P_{\{\alpha_1\}}} \P^2$. 
\end{example}

\begin{remark}[Toric vs. non-toric examples]\label{rmk:toric vs non-toric examples - good moduli spaces}
	The horospherical stack in \cref{ex:good moduli space does not exist for P^2 modulo G_m} is toric, so it can also be studied using the theory of toric stacks. However, the other examples in this subsection are not toric; see \cref{prop:non-toric condition}.
\end{remark}

\section{Proof of \cref{thm:horospherical isomorphisms}}\label{sec:proof of isomorphism criteria}

In this section, we give the proof of \cref{thm:horospherical isomorphisms}. The idea is to first consider the case where the stacks are cohomologically affine (i.e. the cover varieties are affine, by \cref{rmk:cohomologically affine stack}); for this case we use some ideas from \cite[Appendix B]{satriano2015toric1}. Then we reduce the general case to this case using some affine reduction results summarized below.

\subsection{Local structure and affine reduction}\label{subsec:local structure and affine reduction}

Given a simple $G/H$-horospherical variety $X$, we know that it corresponds to a coloured cone $\sigma^c=(\sigma,\calF)$ on $N$. In the toric setting, $X$ is necessarily affine. However, in the general setting, $X$ might not be affine. To work with affine horospherical varieties, we can use some local structure results for $X$. Here we give a summary of these results; for more detail see \cite[Section 28]{timashev2011homogeneous}. 

Since $X$ is simple, it has a unique closed $G$-orbit $\calO$. Let $Q^-$ denote the stabilizer of all $B^-$-invariant prime divisors of $X$ which do not contain $\calO$.\footnote{Note that all $G$-invariant prime divisors contain $\calO$ since $\calO$ is the \textit{unique} closed $G$-orbit. Thus, $Q^-$ is equal to the stabilizer of all colour divisors of $X$ which do not contain $\calO$, i.e. the colours not in $\calF$.} Note that $Q$ is a parabolic subgroup of $G$ containing $B$ (since $Q^-$ contains $B^-$). In fact, $Q=P_{S\setminus(\calC\setminus\calF)}$, so we have $Q\supseteq P=P_{S\setminus\calC}$. Thus, we have a natural projection map $G/H\to G/Q$. By construction, the associated map of coloured lattices is compatible with the coloured fans for $X$ and $G/Q$ because $\Sigma^c(G/Q)$ is trivial and the dominantly mapped colours are precisely the colours in $\calF(X)$. Thus, we get a horospherical morphism $X\to G/Q$. Let $Z$ be the fibre over the base point $eQ$. 

This $Z$ satisfies the following properties; see \cite[Theorem 28.2]{timashev2011homogeneous}. First, $Z$ is an affine closed subvariety of $X$ which contains the base point, and $X\cong G\times^Q Z$.\footnote{Here $G\times^Q Z$ denotes the associated bundle $(G\times Z)/Q$ where $Q$ acts on $G\times Z$ by $q\cdot (g,z):=(gq^{-1},q\cdot z)$.} There exists a Levi subgroup $\mathcal L\subseteq Q$ (see \cite[Section 30.2]{humphreys1975linear}) such that $Z$ is a horospherical $\mathcal L$-variety. The open $\mathcal L$-orbit of $Z$ is $\mathcal L$-equivariantly isomorphic to $\mathcal L/\mathcal L\cap H$. The associated torus for $\mathcal L/\mathcal L\cap H$ is naturally isomorphic to the associated torus for $G/H$. Thus, the associated lattice for $\mathcal L/\mathcal L\cap H$ is isomorphic to the associated lattice for $G/H$. However, the colour structures (i.e. the colour points) of these lattices differ because the universal colour set for $\mathcal L/\mathcal L\cap H$ is just $\calF$. So the colour points of $N(\mathcal L/\mathcal L\cap H)$ are the same $u_\alpha$ as in $N=N(G/H)$, but only for $\alpha\in\calF$. A very important property is that the coloured cone corresponding to $Z$ is identical to that of $X$, i.e. it is $\sigma^c=(\sigma,\calF)$. As a special case, when $\calF=\varnothing$ (i.e. $X$ is toroidal), $N(\mathcal L/\mathcal L\cap H)$ has no colour points, i.e. $\mathcal L/\mathcal L\cap H$ is a torus, and $Z$ is an affine toric variety with corresponding cone $\sigma$. 

For a horospherical morphism $\pi:X_1\to X_2$ between simple $G/H_i$-horospherical varieties, we always have $Q_1\subseteq Q_2$ (using the notation of the previous paragraph with subscripts $i=1,2$). In the situation where $Q_1=Q_2$, the Levi subgroups $\mathcal L_1$ and $\mathcal L_2$ are the same, say $\mathcal L$, and $\pi$ restricts to a horospherical morphism $Z_1\to Z_2$ of horospherical $\mathcal L$-varieties. In fact, we have $\pi^{-1}(Z_2)=Z_1$ in this case. Moreover, the horospherical morphism $Z_1\to Z_2$ extends the projection $\mathcal L/\mathcal L\cap H_1\to \mathcal L/\mathcal L\cap H_2$, and the map of coloured lattices associated to this projection is the same as the one associated to $G/H_1\to G/H_2$. 

As mentioned above, the associated tori for $G/H$ and $\mathcal L/\mathcal L\cap H$ are the same $T_N$. Thus, $G$-equivariant automorphisms acting on $X$ is the same as $\mathcal L$-equivariant automorphisms acting on $Z$, and $Z$ is $T_N$-invariant.

\subsection{The proof}\label{subsec:iso - the proof}

Throughout the rest of this section, we assume that $(\Phi,\phi):(\Sigma_1^c,\beta_1)\to(\Sigma_2^c,\beta_2)$ is a map of stacky coloured fans, so we have an induced horospherical morphism $\pi_{\Phi,\phi}:\calX_{\Sigma_1^c,\beta_1}\to \calX_{\Sigma_2^c,\beta_2}$. 

\begin{lemma}\label{lemma:Iso - affine version of theorem}
	Suppose that $\Sigma_i^c=\sigma_i^c$ is generated by a single coloured cone and that $X_{\sigma_i^c}$ is affine for $i=1,2$. Suppose that the following conditions hold:
	\begin{enumerate}
		\item $\phi$ is an isomorphism of coloured lattices. Thus $\calC_\Phi=\calC_\phi=\varnothing$ and $\calC_1=\calC_2$. 
		\item $\Phi_\R(\sigma_1)=\sigma_2$ and $\Phi$ induces an isomorphism of monoids $\sigma_1\cap N_1\to \sigma_2\cap N_2$. 
	\end{enumerate}
	Then $\pi_{\Phi,\phi}:\calX_{\sigma_1^c,\beta_1}\to \calX_{\sigma_2^c,\beta_2}$ is an isomorphism. 
	
	\begin{proof}%[Proof (cf. {\cite[Lemma B.19]{satriano2015toric1}})]		
		Let $(N_1)_{\sigma_1}$ be the sublattice of $N_1$ generated by $\sigma_1\cap N_1$, and let $(N_1)_{\sigma_1}'$ be a direct complement inside $N_1$. By \cite[Theorem 28.2]{timashev2011homogeneous}, there exists an affine horospherical $\mathcal L_0$-variety $Z_{\sigma_1^c}$ with a fixed point under $\mathcal L_0$ such that $X_{\sigma_1^c}=Z_{\sigma_1^c}\times T_1'$; here $\mathcal L_0$ is between $[G,G]$ (the commutator) and $G$, and $T_1'=G/\mathcal L_0$ is a torus whose associated lattice of one-parameter subgroups is $(N_1)_{\sigma_1}'$. Thus $\calX_{\sigma_1^c,\beta_1}=[(Z_{\sigma_1^c}\times T_1')/K_{\beta_1}]$.
		
		By condition (2), the sublattice of $N_2$ generated by $\sigma_2\cap N_2$ is isomorphic to $(N_1)_{\sigma_1}$, and this isomorphism identifies $\sigma_1$ with $\sigma_2$. Let $(N_2)_{\sigma_1}'$ be a direct complement of this copy of $(N_1)_{\sigma_1}$ inside $N_2$, and let $T_2'$ be the torus corresponding to $(N_2)_{\sigma_1}'$, i.e. $(N_2)_{\sigma_1}'$ is the one-parameter subgroup lattice of $T_2'$. By \cite[Theorem 28.3]{timashev2011homogeneous} and (1), $\Phi$ induces an isomorphism between $Z_{\sigma_1^c}$ and $Z_{\sigma_2^c}$; here $Z_{\sigma_2^c}$ is similarly constructed to $Z_{\sigma_1^c}$. Hence $X_{\sigma_2^c}=Z_{\sigma_1^c}\times T_2'$, which implies $\calX_{\sigma_2^c,\beta_2}=[(Z_{\sigma_1^c}\times T_2')/K_{\beta_2}]$. 
		
		Notice that $T_{N_i}=T_{(N_i)_{\sigma_1}}\times T_i'$ for $i=1,2$. By \cite[Lemma A.1]{satriano2015toric1} and (1), the following diagram has exact rows:
		\begin{equation*}
			\begin{tikzcd}
				{\{1\}} & {K_\Phi} & {K_{\beta_1}} & {K_{\beta_2}} & {\{1\}} \\
				{\{1\}} & {K_\Phi\times\{1\}} & {T_1'\times T_{(N_1)_{\sigma_1}}} & {T_2'\times T_{(N_1)_{\sigma_1}}} & {\{1\}}
				\arrow[from=2-1, to=2-2]
				\arrow[from=1-1, to=1-2]
				\arrow[from=1-4, to=1-5]
				\arrow[from=2-4, to=2-5]
				\arrow[from=2-2, to=2-3]
				\arrow[from=2-3, to=2-4]
				\arrow[from=1-2, to=1-3]
				\arrow[from=1-3, to=1-4]
				\arrow[from=1-3, to=2-3]
				\arrow[Rightarrow, no head, from=1-2, to=2-2]
				\arrow[from=1-4, to=2-4]
			\end{tikzcd}
		\end{equation*}
		Therefore, $\pi_{\Phi,\phi}$ is the isomorphism
		\begin{align*}
			\calX_{\sigma_1^c,\beta_1} &= [(Z_{\sigma_1^c}\times T_1')/K_{\beta_1}] \isomap \left[[(Z_{\sigma_1^c}\times T_1')/K_\Phi]/K_{\beta_2}\right] = [(Z_{\sigma_1^c}\times T_2')/K_{\beta_2}] = \calX_{\sigma_2^c,\beta_2}.
			\qedhere
		\end{align*}
	\end{proof}
\end{lemma}

\begin{lemma}\label{lemma:preimage of non-filled colour is colour}
	Suppose that $\Sigma_i^c=\sigma_i^c$ are generated by single coloured cones for $i=1,2$. Let $X_i:=X_{\sigma_i^c}$, let $\calO_i\subseteq X_i$ be the unique closed $G$-orbit, and let $\pi:X_1\to X_2$ be the horospherical morphism induced by $\Phi$. For each colour divisor $D_2\subseteq X_2$, if $D_2\not\supseteq\calO_2$, then $\pi^{-1}(D_2)$ is a colour divisor of $X_1$. 
	
	\begin{proof}
		We know that $D_2:=\ol{\pi(D_1)}$ for some colour divisor $D_1\subseteq X_1$. Then $D_1\subseteq \pi^{-1}(D_2)$, and thus $D_1$ is the unique colour divisor in $\pi^{-1}(D_2)$. Indeed, the preimage of any colour divisor $D_2\subseteq X_2$ must contain some colour divisor of $X_1$, and since the number of colour divisors of $X_1$ which are not mapped dominantly under $\pi$ is equal to the number of colour divisors of $X_2$, the claim follows from the pigeonhole principle.
		
		Since $D_2$ is $B^-$-invariant and Cartier (see \cite[Lemma 17.3]{timashev2011homogeneous}), we know that $\pi^{-1}(D_2)$ is a $B^-$-invariant effective Cartier divisor. The irreducible components of $\pi^{-1}(D_2)$ must be $B^-$-invariant since $B^-$ is connected, so every irreducible component aside from $D_1$ must be $G$-invariant. In particular, every irreducible component aside from $D_1$ must contain $\calO_1$. If $D_1$ is not the only irreducible component, then $\calO_1\subseteq \pi^{-1}(D_2)$ and thus $\calO_2=\pi(\calO_1)\subseteq D_2$, which is a contradiction. Therefore $\pi^{-1}(D_2)=D_1$.
	\end{proof}
\end{lemma}

\begin{lemma}\label{lemma:preimage of Z under horospherical morphism}
	Suppose that $\Sigma_i^c=\sigma_i^c$ are generated by single coloured cones for $i=1,2$. Let $X_i:=X_{\sigma_i^c}$, let $\calO_i\subseteq X_i$ be the unique closed $G$-orbit, and let $\pi:X_1\to X_2$ denote the horospherical morphism induced by $\Phi$. If $\calF_1=\calF_2\cup\calC_\Phi$, then $\pi^{-1}(Z_2)=Z_1$ using the notation of \cref{subsec:local structure and affine reduction}. 
	
	\begin{proof}		
		Since $\calC_1=\calC_2\sqcup\calC_\Phi$ and $\calF_1=\calF_2\sqcup\calC_\Phi$, it follows that $\calC_1\setminus\calF_1=\calC_2\setminus\calF_2$. Moreover, \cref{lemma:preimage of non-filled colour is colour} implies that $\pi^{-1}(D_\alpha)=D_\alpha$ for $\alpha\in\calC_2\setminus\calF_2$, where $D_\alpha$ is regarded as the colour divisor corresponding to $\alpha$ in both $X_1$ and $X_2$. It follows that $Q_1=Q_2$ in the notation of \cref{subsec:local structure and affine reduction}, which implies $\pi^{-1}(Z_2)=Z_1$. 
	\end{proof}
\end{lemma}

\begin{proposition}\label{prop:Iso - single coloured cone version of theorem}
	Suppose that $\Sigma_i^c=\sigma_i^c=(\sigma_i,\calF_i)$ are generated by single coloured cone for $i=1,2$. Suppose that the following conditions hold:
	\begin{enumerate}
		\item $\phi$ is an isomorphism of coloured lattices. Thus $\calC_\Phi=\calC_\phi=\varnothing$ and $\calC_1=\calC_2$. 
		\item $\calF_1=\calF_2$. 
		\item $\Phi_\R(\sigma_1)=\sigma_2$ and $\Phi$ induces an isomorphism of monoids $\sigma_1\cap N_1\to \sigma_2\cap N_2$. 
	\end{enumerate}
	Then $\pi_{\Phi,\phi}:\calX_{\sigma_1^c,\beta_1}\to \calX_{\sigma_2^c,\beta_2}$ is an isomorphism. 
	
	\begin{proof}		
		Let $Z_i$ be the closed affine horospherical subvariety of $X_{\sigma_i^c}$ for $i=1,2$ in the notation of \cref{subsec:local structure and affine reduction}. By condition (2) and \cref{lemma:preimage of Z under horospherical morphism}, we can base change $\pi_{\Phi,\phi}$ to $[Z_1/K_{\beta_1}]\to [Z_2/K_{\beta_2}]$. Now we are in the situation of \cref{lemma:Iso - affine version of theorem} from conditions (1) and (3), so we know that $[Z_1/K_{\beta_1}]\to [Z_2/K_{\beta_2}]$ is an isomorphism.
		
		Since $X_{\sigma_i^c}\cong G\times^{Q}Z_i$ (where $Q=Q_1=Q_2$ is in the notation of \cref{subsec:local structure and affine reduction}), it follows that the isomorphism $[Z_1/K_{\beta_1}]\to [Z_2/K_{\beta_2}]$ extends $G$-equivariantly to an isomorphism $[X_{\sigma_1^c}/K_{\beta_1}]\to [X_{\sigma_2^c}/K_{\beta_2}]$, which is what we wanted to show. 
	\end{proof}
\end{proposition}

\begin{proof}[Proof of \cref{thm:horospherical isomorphisms}]	
	For the reverse direction, suppose that conditions \labelcref{item:Iso1,item:Iso2,item:Iso3} hold. To show that $\pi_{\Phi,\phi}:\calX_{\Sigma_1^c,\beta_1}\to \calX_{\Sigma_2^c,\beta_2}$ is an isomorphism, it suffices to look locally on the base, so we may assume that $\Sigma_2^c=\sigma_2^c$ is generated by a single coloured cone. Then the result follows from \cref{prop:Iso - single coloured cone version of theorem}. 
	
	Now we prove the forward direction, so suppose that $\pi_{\Phi,\phi}:\calX_{\Sigma_1^c,\beta_1}\to \calX_{\Sigma_2^c,\beta_2}$ is an isomorphism. Fix $\sigma_2^c\in\Sigma_2^c$. We must have that $\Phi^{-1}(\sigma_2^c)$ is generated by a single coloured cone, say $\sigma_1^c\in\Sigma_1^c$. Indeed, $\pi_{\Phi,\phi}$ restricts to an isomorphism $\calX_{\Phi^{-1}(\sigma_2^c),\beta_1}\to \calX_{\sigma_2^c,\beta_2}$, so because the codomain has a unique closed $G$-orbit, the domain must as well, which means that $\Phi^{-1}(\sigma_2^c)$ is generated by a single coloured cone.
	
	To finish the proof, we may assume that $\Sigma_i^c=\sigma_i^c$ for $i=1,2$, and we show that conditions (1)-(3) in \cref{prop:Iso - single coloured cone version of theorem} hold. 
	
	The isomorphism $\pi_{\Phi,\phi}$ restricts to an isomorphism of the open $G$-orbits $G/\ol{K_{\beta_1}}\to G/\ol{K_{\beta_2}}$. Thus $\phi$ is an isomorphism of coloured lattices, so condition (1) holds. In particular, we have $\calC_1=\calC_2=:\calC$. 
	
	By compatibility of $\Phi$, we must have $\calF_1\subseteq\calF_2$. So to show that condition (2) holds, it suffices to show the other inclusion via contraposition, i.e. we show that $\calC\setminus \calF_1\subseteq \calC\setminus\calF_2$. Let $\calO_i$ be the unique closed $G$-orbit in $X_{\sigma_i^c}$ for $i=1,2$. Let $\pi:X_{\sigma_1^c}\to X_{\sigma_2^c}$ denote the horospherical morphism induced by $\Phi$. Take a colour divisor $D_1\subseteq X_{\sigma_1^c}$ with $D_1\not\supseteq \calO_1$, i.e. this corresponds to an element of $\calC\setminus\calF_1$. The corresponding colour divisor in $X_{\sigma_2^c}$ is $D_2:=\ol{\pi(D_1)}$. Since $\pi_{\Phi,\phi}$ is an isomorphism and $\pi(\calO_1)=\calO_2$, it follows that $D_2\not\supseteq\calO_2$. Hence $D_2$ corresponds to an element of $\calC\setminus\calF_2$, which gives $\calC\setminus\calF_1\subseteq\calC\setminus\calF_2$. 
	
%	We need to show that $D_2\not\supseteq \calO_2$. By way of contradiction, assume that $D_2\supseteq \calO_2$. Then the image of $D_2$ in $\calX_{\sigma_2^c,\beta_2}$ contains the $G$-orbit $\calO_2/K_{\beta_2}$. Since $\pi_{\Phi,\phi}$ is an isomorphism, this contradicts the fact that the image of $D_1$ in $\calX_{\sigma_1^c,\beta_1}$ does not contain the $G$-orbit $\calO_1/K_{\beta_1}$. 
	
	It remains to show that condition (3) holds. Let $\wt\sigma_i^c:=(\sigma_i,\varnothing)$ be the decolourization of $\sigma_i^c$ for $i=1,2$. From \cref{subsec:decolouration}, we have the following commutative diagram of horospherical morphisms:
	\begin{equation*}
		\begin{tikzcd}
			{\calX_{\wt\sigma_1^c,\beta_1}} & {\calX_{\wt\sigma_2^c,\beta_2}} \\
			{\calX_{\sigma_1^c,\beta_1}} & {\calX_{\sigma_2^c,\beta_2}}
			\arrow["{\wt\pi_1}"', from=1-1, to=2-1]
			\arrow["{\pi_{\Phi,\phi}}"', from=2-1, to=2-2]
			\arrow["{\wt\pi_2}", from=1-2, to=2-2]
			\arrow[from=1-1, to=1-2]
		\end{tikzcd}
	\end{equation*}
	Since $\pi_{\Phi,\phi}$ is an isomorphism, \cref{cor:uniqueness of decolouration} implies that ${\calX_{\wt\sigma_1^c,\beta_1}}\to {\calX_{\wt\sigma_2^c,\beta_2}}$ is an isomorphism. The colour sets of these decolouration stacks are empty, so \cref{lemma:preimage of Z under horospherical morphism} tells us that we can base change to an isomorphism $[\wt Z_1/K_{\beta_1}]\to [\wt Z_2/K_{\beta_2}]$, where $\wt Z_i$ is the closed affine horospherical subvariety of $X_{\wt\sigma_i^c}$ for $i=1,2$ in the notation of \cref{subsec:local structure and affine reduction}. In this case, the $\wt Z_i$ are affine toric varieties with corresponding cones $\sigma_i$, so we have an isomorphism of toric stacks induced by the map of stacky fans $(\Phi,\phi)$ in the sense of \cite[Section 3]{satriano2015toric1}. Therefore, \cite[Theorem B.3]{satriano2015toric1} implies that $\Phi$ induces an isomorphism of monoids $\sigma_1\cap N_1\to \sigma_2\cap N_2$, which shows that (3) holds. 
\end{proof}

\section{Proof of \cref{thm:GMS map criteria}}\label{sec:proof of good moduli space criteria}

In this section, we give the proof of \cref{thm:GMS map criteria}. The idea is to first consider the case where the stacks are cohomologically affine, and then reduce the general case to this case. For the cohomologically affine case, we follow the proof strategy developed in \cite[Section 6.2]{satriano2015toric1}.

\subsection{Cohomologically affine case}\label{subsec:GMS - cohomologically affine case}

Throughout this subsection, we use the following setup. We have a map of stacky coloured fans $(\Phi,\phi):(\sigma_1^c,\beta_1)\to (\sigma_2^c,\beta_2)$ where $\sigma_i^c=(\sigma_i,\calF_i)$ are coloured cones on $N_i$ for $i=1,2$. Moreover, we assume that $\Phi_\R(\sigma_1)=\sigma_2$, and that $\calF_i=\calC_i$ for each $i=1,2$. The latter assumption implies that the associated horospherical varieties $X_{\sigma_i^c}$ are affine, which is equivalent to saying that the horospherical stacks $\calX_{\sigma_i^c,\beta_i}$ are cohomologically affine (see \cref{rmk:cohomologically affine stack}).

\begin{lemma}\label{lemma:GMS - characterization when base is horospherical variety}	
	If $\beta_2$ is an isomorphism (of lattices), then the induced map $$\pi_{\Phi,\phi}=\calX_{\sigma_1^c,\beta_1}\to \calX_{\sigma_2^c,\beta_2}=X_{\sigma_2^c}$$ is a good moduli space morphism if and only if, for every $l_1\in L_1^\vee$ such that $\beta_1^\vee(l_1)\in\sigma_1^\vee$, there exists a unique $l_2\in L_2^\vee$ such that $\beta_2^\vee(l_2)\in\sigma_2^\vee$ and $\phi^\vee(l_2)=l_1$.\footnote{This condition is saying that: every linear functional $L_1\to\Z$ which is nonnegative on $(\beta_1)_\R(\sigma_1)$ is the pullback under $\phi$ of a unique linear functional $L_2\to\Z$ which is nonnegative on $(\beta_2)_\R(\sigma_2)$.}
	
	\begin{proof}%[Proof (cf. {\cite[Lemma 6.10]{satriano2015toric1}})]		
		Since $X_{\sigma_i^c}$ are affine for both $i=1,2$, we can use \cite[Theorem 28.3]{timashev2011homogeneous} to write
		\begin{align*}
			k[X_{\sigma_i^c}] = \bigoplus_{m\in\sigma_i^\vee\cap N_i^\vee} k[G]^{(B^-)}_{-m}
		\end{align*}
		where $k[G]^{(B^-)}_{-m}$ is the $-m$-th graded piece of the $B^-$-eigenfunctions $k[G]^{(B^-)}$ of the coordinate ring $k[G]$. Since $K_{\beta_1}$ is linearly reductive, the induced map $\pi_{\Phi,\phi}$ is a good moduli space morphism if and only if the map of coordinate rings
		\begin{align*}
			k[X_{\sigma_2^c}]=\bigoplus_{m\in\sigma_2^\vee\cap N_2^\vee} k[G]^{(B^-)}_{-m} \to k[X_{\sigma_1^c}]=\bigoplus_{m\in\sigma_1^\vee\cap N_1^\vee} k[G]^{(B^-)}_{-m}
		\end{align*}
		is the inclusion of ring invariants $k[X_{\sigma_1^c}]^{K_{\beta_1}}$. Notice that this map of coordinate rings corresponds to the map of monoids $\res{\Phi^\vee}{\sigma_2^\vee\cap N_2^\vee}:\sigma_2^\vee\cap N_2^\vee\to \sigma_1^\vee\cap N_1^\vee$.
		
		Let $\theta:N_1^\vee\to \frakX(K_{\beta_1})$ be the map of character groups corresponding to the action of $K_{\beta_1}$ on $T_{N_1}$. Then the previous paragraph tells us that $\pi_{\Phi,\phi}$ is a good moduli space morphism if and only if $\sigma_1^\vee\cap\ker(\theta)=\Phi^\vee(\sigma_2^\vee\cap N_2^\vee)$ and $\res{\Phi^\vee}{\sigma_2^\vee\cap N_2^\vee}$ is injective. Consider the following commutative diagram of abelian groups:
		\begin{equation*}
			\begin{tikzcd}
				{\frakX(K_{\beta_1})} & {N_1^\vee} & {N_2^\vee} \\
				& {L_1^\vee} & {L_2^\vee}
				\arrow["\rotatebox{90}{$\sim$}", "{\beta_2^\vee}"', from=2-3, to=1-3]
				\arrow["{\phi^\vee}", from=2-3, to=2-2]
				\arrow["{\Phi^\vee}"', from=1-3, to=1-2]
				\arrow["\theta"', from=1-2, to=1-1]
				\arrow["0", from=2-2, to=1-1]
				\arrow["{\beta_1^\vee}"', hook, from=2-2, to=1-2]
			\end{tikzcd}
		\end{equation*}
		(note that $\beta_1^\vee$ is injective because $\beta_1$ has finite cokernel; and $\beta_2^\vee$ is an isomorphism because $\beta_2$ is). 
		
		For the forward direction, suppose that $\pi_{\Phi,\phi}$ is a good moduli space morphism. Given $l_1\in L_1^\vee$ such that $\beta_1^\vee(l_1)\in\sigma_1^\vee$, we have that $\beta_1^\vee(l_1)\in\sigma_1^\vee\cap\ker(\theta)=\Phi^\vee(\sigma_2^\vee\cap N_2^\vee)$. Since $\res{\Phi^\vee}{\sigma_2^\vee\cap N_2^\vee}$ is injective and $\beta_2^\vee$ is an isomorphism, there exists a unique $l_2\in L_2^\vee$ such that $\beta_2^\vee(l_2)\in\sigma_2^\vee$ and $\Phi^\vee(\beta_2^\vee(l_2))=\beta_1^\vee(l_1)$. Thus $\beta_1^\vee(\phi^\vee(l_2))=\beta_1^\vee(l_1)$, which gives the desired result since $\beta_1^\vee$ is injective. 
		
		Now we prove the reverse direction. The inclusion $\sigma_1^\vee\cap\ker(\theta)\supseteq \Phi^\vee(\sigma_2^\vee\cap N_2^\vee)$ is easy because the image of $\Phi^\vee$ is in $\ker(\theta)$, and since $\Phi_\R(\sigma_1)=\sigma_2$ we see that $\Phi^\vee$ sends elements of $\sigma_2^\vee$ to elements of $\sigma_1^\vee$. Now we show that $\sigma_1^\vee\cap\ker(\theta)\subseteq \Phi^\vee(\sigma_2^\vee\cap N_2^\vee)$. 
		Since $\ker(\theta)=\image(\beta_1^\vee)$, any element of $\sigma_1^\vee\cap\ker(\theta)$ is of the form $\beta_1^\vee(l_1)$ for some $l_1\in L_1^\vee$; in fact, $l_1$ is unique since $\beta_1^\vee$ is injective. By hypothesis, we can write $l_1=\phi^\vee(l_2)$ for a unique $l_2\in L_2^\vee$ such that $\beta_2^\vee(l_2)\in\sigma_2^\vee$. So any element of $\sigma_1^\vee\cap \ker(\theta)$ is of the form $\beta_1^\vee(\phi^\vee(l_2))=\Phi^\vee(\beta_2^\vee(l_2))$ with $\beta_2^\vee(l_2)\in\sigma_2^\vee\cap N_2^\vee$, which gives the desired inclusion. Finally, $\res{\Phi^\vee}{\sigma_2^\vee\cap N_2^\vee}$ is injective because of the uniqueness conditions. 
	\end{proof}
\end{lemma}

\begin{lemma}\label{lemma:GMS - isomorphism on open orbits implies gms}
	If $\phi$ is an isomorphism (of lattices), then $\pi_{\Phi,\phi}:\calX_{\sigma_1^c,\beta_1}\to \calX_{\sigma_2^c,\beta_2}$ is a good moduli space morphism.
	
	\begin{proof}%[Proof (cf. {\cite[Lemma 6.11]{satriano2015toric1}})]		
		By \cite[Lemma A.1]{satriano2015toric1}, $K_{\beta_1}$ is an extension of $K_{\beta_2}$ by $K_\Phi=\ker(\Phi_\Gm)$. So the map $\pi_{\Phi,\phi}$ is
		\begin{align*}
			[X_{\sigma_1^c}/K_{\beta_1}] = [(X_{\sigma_1^c}/K_\Phi)/K_{\beta_2}] \to [X_{\sigma_2^c}/K_{\beta_2}].
		\end{align*}
		By \cref{lemma:GMS - characterization when base is horospherical variety}, $[X_{\sigma_1^c}/K_\Phi]\to X_{\sigma_2^c}$ is a good moduli space morphism. Since the property of being a good moduli space morphism can be checked locally on the base in the smooth topology (see \cite[Proposition 4.7]{alper2013good}), we deduce that $[X_{\sigma_1^c}/K_{\beta_1}]\to [X_{\sigma_2^c}/K_{\beta_2}]$ is a good moduli space morphism.
	\end{proof}
\end{lemma}

\begin{lemma}\label{lemma:GMS - quotient of a gms is a gms}
	For each $i=1,2$, suppose that $\beta_i$ factors as $N\xrightarrow{\beta_i'} L_i'\xrightarrow{\beta_i''} L_i$ where $\beta_i'$ has finite cokernel. Suppose that $\phi':L_1'\to L_2'$ makes the following commutative diagram
	\begin{equation}\label{eq:GMS - quotient of a gms diagram}
		\begin{tikzcd}
			{N_1} & {N_2} \\
			{L_1'} & {L_2'} \\
			{L_1} & {L_2}
			\arrow["{\beta_1'}"', from=1-1, to=2-1]
			\arrow["{\beta_2'}", from=1-2, to=2-2]
			\arrow["{\beta_1''}"', from=2-1, to=3-1]
			\arrow["{\beta_2''}", from=2-2, to=3-2]
			\arrow["\Phi", from=1-1, to=1-2]
			\arrow["\phi", from=3-1, to=3-2]
			\arrow["{\phi'}", from=2-1, to=2-2]
			\arrow["{\beta_2}", curve={height=-30pt}, from=1-2, to=3-2]
			\arrow["{\beta_1}"', curve={height=30pt}, from=1-1, to=3-1]
		\end{tikzcd}
	\end{equation} 
	Suppose that $\phi'$ induces an isomorphism between $\ker(\beta_1'')$ and $\ker(\beta_2'')$, and $\phi$ induces an isomorphism between $\coker(\beta_1'')$ and $\coker(\beta_2'')$. Then $\pi_{\Phi,\phi}:\calX_{\sigma_1^c,\beta_1}\to \calX_{\sigma_2^c,\beta_2}$ is a good moduli space morphism if and only if $\pi_{\Phi,\phi'}:\calX_{\sigma_1^c,\beta_1'}\to \calX_{\sigma_2^c,\beta_2'}$ is a good moduli space morphism. 
	
	\begin{proof}%[Proof (cf. {\cite[Lemma 6.12]{satriano2015toric1}})]		
		By \cite[Lemma A.1]{satriano2015toric1}, the rows in the following diagram are exact:
		\begin{equation*}
			\begin{tikzcd}
				0 \arrow[r] & K_{\beta_1'} \arrow[r] \arrow[d] & K_{\beta_1} \arrow[r] \arrow[d] & K_{\beta_1''} \arrow[r] \arrow[d] & 0 \\
				0 \arrow[r] & K_{\beta_2'} \arrow[r]           & K_{\beta_2} \arrow[r]           & K_{\beta_2''} \arrow[r]           & 0
			\end{tikzcd}
		\end{equation*}
		Moreover, the last supposition implies that the vertical map $K_{\beta_1''}\to K_{\beta_2''}$ is an isomorphism. Thus, the morphism 
		\begin{align*}
			\pi_{\Phi,\phi}:\calX_{\sigma_1^c,\beta_1}\cong [\calX_{\sigma_1^c,\beta_1'}/K_{\beta_1''}] \to [\calX_{\sigma_2^c,\beta_2'}/K_{\beta_2''}]\cong \calX_{\sigma_2^c,\beta_2}
		\end{align*}
		is the quotient of the morphism $\pi_{\Phi,\phi'}:\calX_{\sigma_1^c,\beta_1'}\to\calX_{\sigma_2^c,\beta_2'}$ by $K_{\beta_1''}\cong K_{\beta_2''}$. We conclude that $\pi_{\Phi,\phi}$ is a good moduli space morphism if and only if $\pi_{\Phi,\phi'}$ is a good moduli space morphism, because the property of being a good moduli space morphism can be checked locally on the base in the smooth topology (see \cite[Proposition 4.7]{alper2013good}). 
	\end{proof}
\end{lemma}

\begin{proposition}\label{prop:GMS - quotienting by unstable cone is gms}
	Let $\tau^c$ be an unstable coloured face of $\sigma_1^c$. Suppose that $N_2=N_1/\tau^{\gp}$ and $L_2=L_1/(\beta_1)_\R(\tau)^{\gp}$, and suppose that $\Phi$ and $\phi$ are the quotient maps (see \cref{rmk:quotient by coloured sublattice}). Then $\pi_{\Phi,\phi}:\calX_{\sigma_1^c,\beta_1}\to \calX_{\sigma_2^c,\beta_2}$ is a good moduli space morphism. 
	
	\begin{proof}%[Proof (cf. {\cite[Proposition 6.14]{satriano2015toric1}})]
		Let $\mathfrak K:=\ker(\beta_1)\cap\tau^{\gp}$. Note that $\mathfrak K$ is saturated in $N_1$ (it is an intersection of saturated subgroups), so $N_1/\mathfrak K$ is a lattice. Moreover, $N_1/\mathfrak K$ is a coloured lattice for a horospherical homogeneous space as in \cref{rmk:quotient by coloured sublattice} by considering the universal colour set on $\mathfrak K$ as the set of colours for $\tau^c$ whose colour points lie in $\mathfrak K$. Consider the following commutative diagram:
		\begin{equation*}
			\begin{tikzcd}
				{N_1} & {N_1/\tau^{\gp}} \\
				{N_1/\mathfrak K} & {N_1/\tau^{\gp}} \\
				{L_1} & {L_1/(\beta_1)_\R(\tau)^{\gp}}
				\arrow["\phi", two heads, from=3-1, to=3-2]
				\arrow["{\phi'}", two heads, from=2-1, to=2-2]
				\arrow["\Phi", two heads, from=1-1, to=1-2]
				\arrow["{\beta_2'}"', Rightarrow, no head, from=1-2, to=2-2]
				\arrow["{\beta_2''}"', from=2-2, to=3-2]
				\arrow["{\beta_1'}", two heads, from=1-1, to=2-1]
				\arrow["{\beta_1''}", from=2-1, to=3-1]
				\arrow["{\beta_1}"', curve={height=30pt}, from=1-1, to=3-1]
				\arrow["{\beta_2}", curve={height=-30pt}, from=1-2, to=3-2]
			\end{tikzcd}
		\end{equation*}
		This is a rectangle in the form of \cref{eq:GMS - quotient of a gms diagram}. The top square induces a good moduli space morphism $\pi_{\Phi,\phi'}:\calX_{\sigma_1^c,\beta_1'}\to \calX_{\sigma_2^c,\beta_2'}$ by \cref{lemma:GMS - characterization when base is horospherical variety}. Indeed, if $l_1\in (N_1/\mathfrak K)^\vee$ satisfies $(\beta_1')^\vee(l_1)\in\sigma_1^\vee\subseteq \tau^\vee$, then $(\beta_1')^\vee(l_1)\in\tau^\perp$ because $\tau^c$ is an unstable coloured face of $\sigma_1^c$, so $l_1=(\phi')^\vee(l_2)$ for a unique $l_2\in(N_1/\tau^{\gp})^\vee$. 
		
		Notice that $\phi'$ induces an isomorphism between $\ker(\beta_1'')=\ker(\beta_1)/\mathfrak K$ and $\ker(\beta_2'')=(\ker(\beta_1)+\tau^{\gp})/\tau^{\gp}$ via the second isomorphism theorem. Similarly, $\phi$ induces an isomorphism between $\coker(\beta_1'')$ and $\coker(\beta_2'')$ via the third isomorphism theorem. 
		
		Therefore, we can apply \cref{lemma:GMS - quotient of a gms is a gms} to deduce that $\pi_{\Phi,\phi}:\calX_{\sigma_1^c,\beta_1}\to\calX_{\sigma_2^c,\beta_2}$ is a good moduli space morphism. 
	\end{proof} 
\end{proposition}

\begin{proposition}\label{prop:GMS - affine gms theorem}
	Let $\tau^c:=\Phi^{-1}(0^c)$. Then $\pi_{\Phi,\phi}:\calX_{\sigma_1^c,\beta_1}\to\calX_{\sigma_2^c,\beta_2}$ is a good moduli space morphism if and only if the following conditions hold:
	\begin{enumerate}
		\item $\tau^c$ is unstable.
		\item $\phi$ is surjective.
		\item $\ker(\phi)=(\beta_1)_\R(\tau)^{\gp}$.
	\end{enumerate}

	\begin{proof}%[Proof (cf. {\cite[Proposition 6.18]{satriano2015toric1}})]		
		For the reverse direction, suppose that conditions (1)-(3) hold. Since a composition of good moduli space morphisms is a good moduli space morphism, it suffices to factor $(\Phi,\phi)$ into three maps of stacky coloured fans as follows, and show that each square induces a good moduli space morphism:
		\begin{equation*}
			\begin{tikzcd}
				{N_1} & {N_1/\tau^{\gp}} & {N_1/\tau^{\gp}} & {N_2} \\
				{L_1} & {L_1/(\beta_1)_\R(\tau)^{\gp}} & {L_2} & {L_2}
				\arrow["{\beta_1}"', from=1-1, to=2-1]
				\arrow[two heads, from=2-1, to=2-2]
				\arrow[two heads, from=1-1, to=1-2]
				\arrow[Rightarrow, no head, from=1-2, to=1-3]
				\arrow[Rightarrow, no head, from=2-3, to=2-4]
				\arrow["{\beta_2}", from=1-4, to=2-4]
				\arrow[from=2-2, to=2-3]
				\arrow[from=1-3, to=1-4]
				\arrow["\Phi", curve={height=-30pt}, from=1-1, to=1-4]
				\arrow["\phi"', curve={height=30pt}, from=2-1, to=2-4]
				\arrow[from=1-3, to=2-3]
				\arrow[from=1-2, to=2-2]
			\end{tikzcd}
		\end{equation*}
		
		By condition (1) and \cref{prop:GMS - quotienting by unstable cone is gms}, the left square induces a good moduli space morphism. By conditions (2) and (3), the map $L_1/(\beta_1)_\R(\tau)^{\gp}\to L_2$ is an isomorphism (of lattices). Thus, \cref{lemma:GMS - isomorphism on open orbits implies gms} implies that the middle and right squares induce good moduli space morphisms. 
		
		Now we prove the forward direction, so suppose that $\pi_{\Phi,\phi}$ is a good moduli space morphism. By \cite[Proposition 4.7(i)]{alper2013good}, base changing $\pi_{\Phi,\phi}$ by the open $G$-orbit $\calX_{0^c,\beta_2}$ of $\calX_{\sigma_2^c,\beta_2}$ is a good moduli space morphism. Thus, we may replace $\sigma_2^c$ with the trivial coloured cone $0^c$ on $N_2$, replace $\sigma_1^c$ by $\tau^c$, and assume that $(\Phi,\phi)$ is of the form
		\begin{equation*}
			\begin{tikzcd}
				{\tau^c} & {0^c} \\[-2em]
				{N_1} & {N_2} \\
				{L_1} & {L_2}
				\arrow["\Phi", from=2-1, to=2-2]
				\arrow["\phi", from=3-1, to=3-2]
				\arrow["{\beta_1}"', from=2-1, to=3-1]
				\arrow["{\beta_2}", from=2-2, to=3-2]
				\arrow[from=1-1, to=1-2]
			\end{tikzcd}
		\end{equation*}
		
		Consider the following composition of maps of stacky coloured fans:
		\begin{equation*}
			\begin{tikzcd}
				{\tau^c} & {0^c} & {0^c} \\[-2em]
				{N_1} & {N_2} & {L_2} \\
				{L_1} & {L_2} & {L_2}
				\arrow["{\beta_1}"', from=2-1, to=3-1]
				\arrow["{\beta_2}", from=2-2, to=3-2]
				\arrow[Rightarrow, no head, from=2-3, to=3-3]
				\arrow[Rightarrow, no head, from=3-2, to=3-3]
				\arrow["{\beta_2}", from=2-2, to=2-3]
				\arrow["\Phi", from=2-1, to=2-2]
				\arrow["\phi", from=3-1, to=3-2]
				\arrow[from=1-1, to=1-2]
				\arrow[from=1-2, to=1-3]
			\end{tikzcd}
		\end{equation*}
		By our supposition, the left square induces a good moduli space morphism, and by \cref{lemma:GMS - isomorphism on open orbits implies gms} the right square induces a good moduli space morphism. Therefore, the full rectangle $(\beta_2\circ\Phi,\phi)$ induces a good moduli space morphism. By applying \cref{lemma:GMS - characterization when base is horospherical variety} to the full rectangle, we see that every linear functional on $L_1$ which is nonnegative on $(\beta_1)_\R(\tau)$ must be induced uniquely from $L_2$, and thus is zero on $(\beta_1)_\R(\tau)$. Therefore $\tau^c$ is unstable, i.e. condition (1) is satisfied. 
		
		Now consider the following factorization of $(\Phi,\phi)$ into two maps of stacky coloured fans:
		\begin{equation*}
			\begin{tikzcd}
				{\tau^c} & {0^c} & {0^c} \\[-2em]
				{N_1} & {N_1/\tau^{\gp}} & {N_2} \\
				{L_1} & {L_1/(\beta_1)_\R(\tau)^{\gp}} & {L_2}
				\arrow["{\beta_1}"', from=2-1, to=3-1]
				\arrow["\beta_1'", from=2-2, to=3-2]
				\arrow["{\beta_2}", from=2-3, to=3-3]
				\arrow["\phi'", from=3-2, to=3-3]
				\arrow["\Phi'", from=2-2, to=2-3]
				\arrow[two heads, from=2-1, to=2-2]
				\arrow[two heads, from=3-1, to=3-2]
				\arrow[from=1-1, to=1-2]
				\arrow[from=1-2, to=1-3]
				\arrow["\phi"', curve={height=30pt}, from=3-1, to=3-3]
			\end{tikzcd}
		\end{equation*}
		Since $\tau^c$ is unstable, \cref{prop:GMS - quotienting by unstable cone is gms} shows that the left square induces a good moduli space morphism. This combined with the supposition that $(\Phi,\phi)$ induces a good moduli space morphism allows us to apply \cite[Lemma 6.9]{satriano2015toric1} to deduce that the right square $(\Phi',\phi')$ induces a good moduli space morphism. 
		
		To show that conditions (2) and (3) are satisfied, it suffices to show that $\phi'$ is an isomorphism (of lattices). Since $\pi_{\Phi',\phi'}$ is a good moduli space morphism, it induces an isomorphism of the good moduli spaces of $\calX_{0^c,\beta_1'}$ and $\calX_{0^c,\beta_2}$. By \cref{lemma:GMS - characterization when base is horospherical variety}, these good moduli spaces are $G/\ol{K_1'}$ (the horospherical homogeneous space whose associated coloured lattice is $L_1/(\beta_1)_\R(\tau)^{\gp}$) and $G/\ol{K_2}$ (the horospherical homogeneous space whose associated coloured lattice is $L_2$), respectively. Therefore, the natural projection map $G/\ol{K_1'}\to G/\ol{K_2}$, whose associated lattice map is $\phi'$, is an isomorphism, so $\phi'$ must be an isomorphism.
	\end{proof}
\end{proposition}

\subsection{General case}\label{subsec:GMS - general case}

Now we prove the full theorem. The idea is to reduce the general case to the previous case where the stacks are cohomologically affine. For the remainder of this section, the coloured fans $\Sigma_i^c$ in the map of stacky coloured fans $(\Phi,\phi):(\Sigma_1^c,\beta_1)\to (\Sigma_2^c,\beta_2)$ are not necessarily generated by single coloured cones.

\begin{lemma}\label{lemma:GMS - preimage of Z under gms}
	Suppose that $\Sigma_i^c=\sigma_i^c$ are generated by single coloured cones for $i=1,2$. Let $X_i:=X_{\sigma_i^c}$, let $\calO_i\subseteq X_i$ be the unique closed $G$-orbit, and let $\pi:X_1\to X_2$ denote the horospherical morphism induced by $\Phi$. If $\pi_{\Phi,\phi}$ is a good moduli space morphism, then $\pi^{-1}(Z_2)=Z_1$ using the notation of \cref{subsec:local structure and affine reduction}. 
	
	\begin{proof}		
		In the notation of \cref{subsec:local structure and affine reduction}, it suffices to show that $Q_1=Q_2$. Since we always have $Q_1\subseteq Q_2$, we just have to show that $Q_2\subseteq Q_1$. To do this, it suffices to show that every colour divisor $D_1\subseteq X_1$ such that $D_1\not\supseteq \calO_1$ is of the form $\pi^{-1}(D_2)$ for some colour divisor $D_2\subseteq X_2$ such that $D_2\not\supseteq \calO_2$. For such a colour divisor $D_1$, we consider two cases: (1) when $\ol{\pi(D_1)}$ is a colour divisor of $X_2$, and (2) when $\ol{\pi(D_1)}$ is all of $X_2$. 
		
		\begin{enumerate}[leftmargin=0.9cm]
			\item In the first case, let $D_2:=\ol{\pi(D_1)}$ be the colour divisor in $X_2$. By \cref{lemma:preimage of non-filled colour is colour}, we know that $\pi^{-1}(D_2)=D_1$, which proves what we want since $\calO_1\not\subseteq D_1=\pi^{-1}(D_2)$ implies $\calO_2=\pi(\calO_1)\not\subseteq D_2$. 
			
			\item We show that this case is impossible by way of contradiction, so assume that $\ol{\pi(D_1)}=X_2$. Let $\alpha\in \calC_1$ be the colour corresponding to $D_1$. Since the property of being a good moduli space morphism is local on the base, we have a good moduli space morphism $\calX_{\Phi^{-1}(0^c),\beta_1}\to \calX_{0^c,\beta_2}=G/\ol{K_2}$. Since $\alpha$ is mapped dominantly, $\alpha$ is not in the universal colour set for $G/\ol{K_2}$. That is, $\alpha\in \calC_1$ but $\alpha\notin \calC_2$. This combined with $I_1\subseteq I_2$ implies $I_1\subseteq I_2\setminus\{\alpha\}$. Thus, we have a projection $G/\ol{K_1}\to G/P_{I_1}\to G/P_{I_2\setminus\{\alpha\}}$. The associated map of coloured lattices is compatible with $\Phi^{-1}(0^c)$ and the coloured fan for $G/P_{I_2\setminus\{\alpha\}}$ since $\alpha\notin\calF(\Phi^{-1}(0^c))$. Therefore, we get a horospherical morphism $\calX_{\Phi^{-1}(0^c),\beta_1}\to G/P_{I_2\setminus\{\alpha\}}$. 
			
			By \cite[Theorem 4.16(vi)]{alper2013good}, good moduli space morphisms are universal to schemes. Thus, there exists a unique dotted arrow which makes the following diagram commute:
			\begin{equation*}
				\begin{tikzcd}
					{\calX_{\Phi^{-1}(0^c),\beta_1}} \\
					{G/\ol{K_2}} & {G/P_{I_2\setminus\{\alpha\}}}
					\arrow[from=1-1, to=2-2]
					\arrow[from=1-1, to=2-1]
					\arrow[dashed, from=2-1, to=2-2]
				\end{tikzcd}
			\end{equation*}
			Since the solid arrows are horospherical morphisms, the dotted arrow is as well. This implies that $P_{I_2}\subseteq P_{I_2\setminus\{\alpha\}}$, which is a contradiction. \qedhere
		\end{enumerate}
	\end{proof}
\end{lemma}

\begin{remark}\label{rmk:GMS - dominantly mapped colours in domain colour set}
	Under the conditions of \cref{lemma:GMS - preimage of Z under gms}, the proof of the lemma shows that all dominantly mapped colours are in the colour set of $\sigma_1^c$, i.e. $\calC_\Phi\subseteq \calF(\sigma_1^c)$. 
\end{remark}

\begin{lemma}\label{lemma:condition for image of a cone to be a cone}
	Let $(\sigma_1^c,\beta_1)$ and $(\Sigma_2^c,\beta_2)$ be stacky coloured fans, where $\sigma_1^c$ is generated by a single coloured cone. Let $(\Phi,\phi):(\sigma_1^c,\beta_1)\to(\Sigma_2^c,\beta_2)$ be a map of stacky coloured fans. If every $G$-orbit of $\calX_{\Sigma_2^c,\beta_2}$ is in the image of $\pi_{\Phi,\phi}$ (e.g. if $\pi_{\Phi,\phi}$ is surjective), then $\Phi_\R(\sigma_1)=|\Sigma_2|$. In particular, this assumption implies that $\Sigma_2^c$ is generated by a single coloured cone. 
	
	\begin{proof}%[Proof (cf. {\cite[Lemma B.7]{satriano2015toric1}})]		
		The coloured cones of $\Sigma_2^c$ correspond to the $G$-orbits of $X_{\Sigma_2^c}$ and of $\calX_{\Sigma_2^c,\beta_2}$. Thus, the cones of $\Sigma_2$ correspond to these $G$-orbits as well. By the assumption, we see that every $G$-orbit of $X_{\Sigma_2^c}$ is in the image of $X_{\sigma_1^c}$ (under the horospherical morphism of varieties induced by $\Phi$). It follows that the relative interior of every cone in $\Sigma_2$ contains the image of some face of $\sigma_1$. Hence, $\Phi_\R(\sigma_1)$ intersects all rays of $\Sigma_2$. Since $\Phi_\R(\sigma_1)$ is a polyhedral cone, it must contain all the rays of $\Sigma_2$. Therefore $\Phi_\R(\sigma_1)\supseteq |\Sigma_2|$, and the other inclusion holds by compatibility of $\Phi$. 
		
		Since $\Phi$ is compatible with the coloured fans, the equality $\Phi_\R(\sigma_1)=|\Sigma_2|$ implies that $\Sigma_2$ has one maximal cone, i.e. $\Sigma_2$ is the fan generated by a single cone. Thus, $\Sigma_2^c$ is generated by a single coloured cone. 
	\end{proof}
\end{lemma}

\begin{proof}[Proof of forward direction of \cref{thm:GMS map criteria}]
	We first prove the forward direction of \cref{thm:GMS map criteria}, so suppose that $\pi_{\Phi,\phi}:\calX_{\Sigma_1^c,\beta_1}\to \calX_{\Sigma_2^c,\beta_2}$ is a good moduli space morphism. Fix $\sigma_2^c=(\sigma_2,\calF_2)\in\Sigma_2^c$. Since the property of being a good moduli space morphism is local on the base, the map $\calX_{\Phi^{-1}(\sigma_2^c),\beta_1}\to \calX_{\sigma_2^c,\beta_2}$ is also a good moduli space morphism. 
	
	We first show that $\Phi^{-1}(\sigma_2^c)$ is generated by a single coloured cone in $\Sigma_1^c$. To do this, it suffices to show that $X_{\Phi^{-1}(\sigma_2^c)}$ has a unique closed $G$-orbit. Suppose that $\calO_1,\calO_1'\subseteq X_{\Phi^{-1}(\sigma_2^c)}$ are closed $G$-orbits. Let $\calO_2\subseteq X_{\sigma_2^c}$ be the unique closed $G$-orbit. Then, since the morphism is horospherical, we have that $\pi_{\Phi,\phi}$ maps $\calO_1/K_{\beta_1}$ and $\calO_1'/K_{\beta_1}$ onto $\calO_2/K_{\beta_2}$. So we can take closed points $x\in\calO_1$ and $x'\in\calO_1'$ which map to the same closed point in $\calO_2/K_{\beta_2}$. We can apply \cite[Theorem 4.16(iv)]{alper2013good} to deduce that the $K_{\beta_1}$-orbit closures of $x$ and $x'$ intersect. Thus, since $\calO_1$ and $\calO_1'$ are closed and $K_{\beta_1}$-invariant, $\calO_1$ and $\calO_1'$ intersect, which implies that $\calO_1=\calO_1'$. Therefore, $\Phi^{-1}(\sigma_2^c)$ is generated by a single coloured cone, say $\sigma_1^c=(\sigma_1,\calF_1)\in\Sigma_1^c$. 
	
	Let $Z_i$ be the closed horospherical affine subvariety of $X_{\sigma_i^c}$ for $i=1,2$ in the notation of \cref{subsec:local structure and affine reduction}. By \cref{lemma:GMS - preimage of Z under gms} and the fact that the property of being a good moduli space morphism is stable under base change, we see that $[Z_1/K_{\beta_1}]\to [Z_2/K_{\beta_2}]$ is a good moduli space morphism. Now we are in the situation where the horospherical stacks are cohomologically affine. Note that $\Phi_\R(\sigma_1)=\sigma_2$ by \cref{lemma:condition for image of a cone to be a cone} and since good moduli space morphisms are surjective (see \cite[Theorem 4.16(i)]{alper2013good}). Thus, the conditions at the start of \cref{subsec:GMS - cohomologically affine case} are satisfied (see \cref{rmk:cohomologically affine stack}), so we can apply \cref{prop:GMS - affine gms theorem} to obtain \labelcref{item:GMS2,item:GMS3,item:GMS4} since the coloured cone for $Z_i$ is $\sigma_i^c$ for $i=1,2$. Lastly, \cref{rmk:GMS - dominantly mapped colours in domain colour set} and the fact that we are in the cohomologically affine setting imply that $\calF_2\cup\calC_\Phi\subseteq\calF_1$. The reverse inclusion holds because $\Phi$ is compatible with the coloured fans, so we have $\calF_2\cup\calC_\Phi=\calF_1$, which establishes \labelcref{item:GMS1}. 
\end{proof}

\begin{proof}[Proof of reverse direction of \cref{thm:GMS map criteria}]
	We now prove the reverse direction of \cref{thm:GMS map criteria}, so suppose that conditions \labelcref{item:GMS1,item:GMS2,item:GMS3,item:GMS4} hold. Since the property of being a good moduli space morphism can be checked locally on the base, it suffices to show that the map $[X_{\Phi^{-1}(\sigma_2^c)}/K_{\beta_1}]\to [X_{\sigma_2^c}/K_{\beta_2}]$ is a good moduli space morphism, where $\sigma_2^c=(\sigma_2,\calF_2)\in\Sigma_2^c$ is an arbitrary coloured cone. By \labelcref{item:GMS1}, $\Phi^{-1}(\sigma_2^c)$ is generated by a single coloured cone, say $\sigma_1^c=(\sigma_1,\calF_1)\in\Sigma_1^c$. 
	
	Let $Z_i$ be the closed affine horospherical subvariety of $X_{\sigma_i^c}$ for $i=1,2$ in the notation of \cref{subsec:local structure and affine reduction}. By the assumption on $\calF_2$ in \labelcref{item:GMS1}, we can apply \cref{lemma:preimage of Z under horospherical morphism} to base change the morphism to $[Z_1/K_{\beta_1}]\to [Z_2/K_{\beta_2}]$. Now we are in a situation where the horospherical stacks are cohomologically affine. Thus, the conditions at the start of \cref{subsec:GMS - cohomologically affine case} are satisfied, so \labelcref{item:GMS2,item:GMS3,item:GMS4} allow us to apply \cref{prop:GMS - affine gms theorem} to conclude that $[Z_1/K_{\beta_1}]\to [Z_2/K_{\beta_2}]$ is a good moduli space morphism since the coloured cone for $Z_i$ is $\sigma_i^c$ for $i=1,2$. 
	
	Since $X_{\sigma_i^c}\cong G\times^{Q}Z_i$ (where $Q=Q_1=Q_2$ is in the notation of \cref{subsec:local structure and affine reduction}), it follows that the good moduli space morphism $[Z_1/K_{\beta_1}]\to [Z_2/K_{\beta_2}]$ extends $G$-equivariantly to a good moduli space morphism $[X_{\sigma_1^c}/K_{\beta_1}]\to [X_{\sigma_2^c}/K_{\beta_2}]$, which is what we wanted to show. 
\end{proof}

\appendix
\section{Constant invertible global functions on horospherical varieties}\label{sec:constant invertible global functions on horospherical varieties}

In this appendix we prove a technical results used in the paper, which is a combinatorial characterization of when the only invertible global functions on a horospherical variety are constant. 

\begin{proposition}\label{prop:characterization of constant global units}
	Let $X=X_{\Sigma^c}$ be a $G/H$-horospherical variety coming from a coloured fan $\Sigma^c$ on $N$. Then $\mathscr O_X^*(X)=k^*$ if and only if $\{u_\alpha\in N:\alpha\in\calC\}\cup|\Sigma^c|$ spans $N_\R$. 
	
	\begin{proof}		
		For the forward direction, we prove the contrapositive, so suppose that $\{u_\alpha\in N:\alpha\in\calC\}\cup|\Sigma^c|$ does not span $N_\R$. Via the isomorphism $N^\vee\cong k(G/H)^{(B^-)}/k^*$ (mentioned in \cref{subsec:colours}), a point $m\in N^\vee$ determines a function $f_m\in k(G/H)^{(B^-)}$ up to an invertible constant. Our supposition is equivalent to the map 
		\begin{align}\label{eq:principal divisor map}
			N^\vee \map \bigoplus_D \Z D \qquad f_m \mapsto \prindiv(f_m)
		\end{align}
		being non-injective (see the exact sequence in \cite[Theorem 4.2.1]{perrin2018sanya}); here the direct sum is over all $B^-$-invariant prime divisors $D$ of $X$. Thus, there exist $f_{m_1},f_{m_2}\in k(G/H)^{(B^-)}$ which are not constant multiples of each other such that $\prindiv(f_{m_1})=\prindiv(f_{m_2})$. This implies that $\prindiv(f_{m_1}/f_{m_2})=0$, so $f_{m_1}/f_{m_2}$ is a nonconstant element of $\mathscr O_X^*(X)$, which is what we wanted to show.
		
		For the reverse direction, we prove the contrapositive, so suppose that $\mathscr O_X^*(X)\neq k^*$. After replacing $G$ by a finite cover (which does not change the combinatorics of $X$), we may assume that $G=G^{ss}\times T_0$ where $G^{ss}$ is a semisimple simply connected linear algebraic group and $T_0$ is a torus. Since $G\surjmap G/H\hookrightarrow X$ is a dominant morphism, we have $\mathscr O_X^*(X)\hookrightarrow k[G]$. Thus, $k[G]^*\neq k^*$ by our supposition. 
		
		Since $k[G^{ss}]^*=k^*$ (because $G^{ss}$ is semisimple)\footnote{This is a known fact about semisimple groups. For a nice explanation, see the Mathematics Stack Exchange post titled ``Invertible functions of connected reductive linear algebraic groups" posted on July 21, 2022.} and $k[G]^*\neq k^*$, it follows that the torus $T_0$ must be nontrivial. Moreover, if $H_0$ denotes the projection of $H$ to the $T_0$ component, then $G/(G^{ss}\times H_0)=T_0/H_0$ is a nontrivial torus. This further implies that $k[T_0/H_0]^*\neq k^*$. Since $H\subseteq G^{ss}\times H_0$, we have a natural projection $G/H\to T_0/H_0$, which has an associated map of coloured lattices $N\to N'$ where $N'$ is the coloured lattice for $T_0/H_0$ (since this is a torus, $N'$ has no colour points). 
		
		Now consider the dual map $(N')^\vee\hookrightarrow N^\vee$. As mentioned above, this is an inclusion
		\begin{align*}
			k(T_0/H_0)^{(T_0)}/k^*\hookrightarrow k(G/H)^{(B^-)}/k^*.
		\end{align*}
		Since $T_0/H_0$ and $T_0$ are tori, we have $k[T_0/H_0]^*/k^*\subseteq k(T_0/H_0)^{(T_0)}/k^*$. Since $k[T_0/H_0]^*/k^*$ is nontrivial, there exists a nonconstant $f\in k(G/H)^{(B^-)}$ such that $\prindiv(f)=0$. This implies that the map in \cref{eq:principal divisor map} is not injective, so we conclude that $\{u_\alpha\in N:\alpha\in\calC\}\cup|\Sigma^c|$ does not span $N_\R$ (again, see \cite[Theorem 4.2.1]{perrin2018sanya}). 
	\end{proof}
\end{proposition}

\section*{Acknowledgements}

I am very grateful to my PhD advisor Matthew Satriano for originally suggesting this project to me, teaching me everything that I know about stacks, and being very supportive in my completion of this project. I am also grateful for the support from NSERC via a PGS-D scholarship (reference number: PGSD3-558713-2021). Lastly, thank you to the anonymous referee for their very detailed and helpful feedback.

\printbibliography 
	
\end{document}